\documentclass[12pt, reqno]{amsart}
\usepackage{amssymb}
\usepackage{graphicx}
\usepackage{xcolor} 
\usepackage[color=yellow]{todonotes}

\usepackage{enumitem}
\usepackage{fullpage} 

\definecolor{green}{rgb}{0,0.8,0} 

\newtheorem{theorem}{Theorem}[section]
\newtheorem{corollary}[theorem]{Corollary}
\newtheorem{lemma}[theorem]{Lemma}
\newtheorem{proposition}[theorem]{Proposition}
\theoremstyle{definition}
\newtheorem{definition}[theorem]{Definition}

\theoremstyle{remark}
\newtheorem{remark}[theorem]{Remark}
\numberwithin{equation}{section}

\newcommand{\tA}{{\tilde A}}
\newcommand{\tphi}{{\tilde \phi}}
\newcommand{\A}{{\mathbf A}}
\renewcommand{\H}{{\mathcal H}}

\newcommand{\nrm}[1]{\Vert#1\Vert}
\newcommand{\abs}[1]{\vert#1\vert}
\newcommand{\brk}[1]{\langle#1\rangle}
\newcommand{\set}[1]{\{#1\}}

\newcommand{\dist}{\mathrm{dist}}

\newcommand{\supp}{{\mathrm{supp}}}

\renewcommand{\Im}{\mathrm{Im}}

\newcommand{\aeq}{\sim}
\newcommand{\aleq}{\lesssim}
\newcommand{\ageq}{\gtrsim}

\newcommand{\lap}{\Delta}

\newcommand{\ud}{d}
\newcommand{\rd}{\partial}
\newcommand{\nb}{\nabla}
\newcommand{\snabla}{ {\slash\!\!\!\!\nabla} }

\newcommand{\bb}{\Big}

\newcommand{\alp}{\alpha}
\newcommand{\bt}{\beta}

\newcommand{\dlt}{\delta}

\newcommand{\eps}{\epsilon}
\newcommand{\veps}{\varepsilon}

\newcommand{\Lmb}{\Lambda}
\newcommand{\sgm}{\sigma}

\newcommand{\tht}{\theta}

\newcommand{\omg}{\omega}

\newcommand{\bfm}{{\bf m}}

\newcommand{\bfA}{{\bf A}}

\newcommand{\bfD}{{\bf D}}

\newcommand{\bbC}{\mathbb C}
\newcommand{\C}{\mathbb C}

\newcommand{\bbR}{\mathbb R}
\newcommand{\R}{\mathbb R}
\newcommand{\bbS}{\mathbb S}

\newcommand{\bbZ}{\mathbb Z}

\newcommand{\calC}{\mathcal C}

\newcommand{\calE}{\mathcal E}
\newcommand{\calF}{\mathcal F}

\newcommand{\calH}{\mathcal H}

\newcommand{\calL}{\mathcal L}
\newcommand{\calM}{\mathcal M}
\newcommand{\calN}{\mathcal N}

\newcommand{\calP}{\mathcal P}
\newcommand{\calQ}{\mathcal Q}

\newcommand{\pfstep}[1]{\vskip.5em \noindent {\bf #1.}}
\newcommand{\qdeq}{\quad \phantom{=}}

\setlist[enumerate]{leftmargin=*, label=(\arabic*)}
\setlist[itemize]{leftmargin=2em}
\newcommand{\dless}{\ll}
\newcommand{\covD}{\bfD}
\newcommand{\opLap}{\lap_{\omg^{\perp}}}
\newcommand{\oL}{L^{\omg}}
\newcommand{\oPi}{\Pi^{\omg}}
\newcommand{\Diff}{\mathrm{Diff}}

\newcommand{\DS}{DS^1}

\newcommand{\rst}{\!\upharpoonright}		

\begin{document}

\title[]{Energy dispersed solutions for the $(4+1)$-dimensional
Maxwell-Klein-Gordon  equation}
\author{Sung-Jin Oh}%
\address{Department of Mathematics, UC Berkeley, Berkeley, CA, 94720}%
\email{sjoh@math.berkeley.edu}%

\author{Daniel Tataru}%
\address{Department of Mathematics, UC Berkeley, Berkeley, CA, 94720}%
\email{tataru@math.berkeley.edu}%


\date{\today}%
\begin{abstract}
  This article is devoted to the mass-less energy critical
  Maxwell-Klein-Gordon system in $4+1$ dimensions. In earlier work of
  the second author, joint with Krieger and Sterbenz, we have proved
  that this problem has global well-posedness and scattering in the
  Coulomb gauge for small initial data. This article is the second of
  a sequence of three papers of the authors, whose goal is to show
  that the same result holds for data with arbitrarily large energy.
  Our aim here is to show that large data solutions persist for as
  long as one has small energy dispersion; hence failure of global
  well-posedness must be accompanied with a non-trivial energy
  dispersion.
\end{abstract}
\maketitle
\setcounter{tocdepth}{1}
\tableofcontents

\section{Introduction}
This article is concerned with the mass-less energy critical
Maxwell-Klein-Gordon system (MKG) in the $4+1$ dimensional Minkowski
space $\bbR^{1+4}$ equipped with the standard Lorentzian metric $\bfm
= \text{diag}(-1,1,1,1,1)$ in the standard rectilinear coordinates
$(x^{0}, \ldots, x^{4})$.  This system is generated by adding a scalar
field component to the standard Maxwell Lagrangian,
\[
\mathcal{S}_{\mathrm{M}}[A_{\alpha}]: = \int_{\R^{1+4}}
\frac{1}{4}F_{\alpha\beta} F^{\alpha\beta} \,dxdt;
\]
to obtain
\[
\mathcal{S}[A_{\alpha}, \phi]: = \int_{\R^{1+4}}
\frac{1}{4}F_{\alpha\beta} F^{\alpha\beta} + \frac{1}{2}
\covD_{\alpha}\phi\overline{\covD^{\alpha}\phi} \,dxdt;
\]

Here $\phi: \R^{1+4}\rightarrow \C$ is a scalar function, and
$A_{\alpha}: \R^{1+4}\rightarrow \R$ is a real-valued connection
1-form, with curvature
\[
F_{\alpha\beta}: = \partial_{\alpha}A_{\beta}
- \partial_{\beta}A_{\alpha}.
\]
The connection 1-form $A_\alpha$ is then used to define the covariant
derivative
\[
\covD_{\alpha}\phi:= (\partial_{\alpha} + iA_{\alpha})\phi.
\]
Introducing the covariant wave operator
\[
\Box_{A}: = \covD^{\alpha} \covD_{\alpha}
\]
with the standard convention for raising/lowering and summing indices,
we can write the \emph{Maxwell-Klein-Gordon system} in the form
\begin{equation}\label{MKG-nogauge}
  \left\{ \begin{aligned}
      & \partial^{\beta}F_{\alpha\beta} = - J_\alpha,
      \\
      & \Box_A\phi = 0
    \end{aligned}
  \right.
\end{equation}
where the currents $J_\alpha$ are defined as
\begin{equation}\label{currents}
  J_\alpha:=- \Im(\phi\overline{\covD_\alpha \phi}) \ .
\end{equation}
 
The MKG system admits a positive definite formally conserved energy
functional,
\begin{equation}\label{energy}
  \calE[A,\phi](t) = \calE_{\set{t} \times \bbR^{4}} [A,\phi] : = \int_{\set{t} \times \R^4}\big(\frac{1}{4}\sum_{\alpha,\beta}F_{\alpha\beta}^2 +
  \frac{1}{2}\sum_{\alpha}| \covD_{\alpha}\phi|^2\big)\,dx
\end{equation}
and is also invariant under the scaling
\[
\phi(t, x) \rightarrow \lambda^{-1} \phi(\lambda^{-1} t, \lambda^{-1}
x),\,A_{\alpha}(t, x)\rightarrow \lambda^{-1} A(\lambda^{-1} t,
\lambda^{-1} x).
\]
Thus the $4+1$-MKG system is energy critical.

In order to state this system as a formally well-posed initial value
problem, we need to take into account its gauge invariance. If
$(A_{\alpha}, \phi)$ is a solution, then so is $(A_{\alpha}
- \partial_{\alpha}\chi, e^{i\chi}\phi)$ for any real-valued scalar
function $\chi$.  In the gauge covariant setting, it is natural to
define an \emph{initial data set} for MKG to consist of a pair of
1-forms $(a_{j}, e_{j})$ and complex-valued functions $f, g$ on
$\bbR^{4}$. We say that $(a, e, f, g)$ is the initial data set for a
solution $(A, \phi)$ if
\begin{equation*}
  (A_{j}, F_{0j}, \phi, \covD_{t} \phi) \rst_{\set{t = 0}} = (a_{j}, e_{j}, f, g), 
\end{equation*}
where the latin indices only run over the spatial variables $x^{1},
\ldots, x^{4}$.  The energy of the set $(a, e, f, g)$, denoted by
$\calE[a, e, f, g]$, is defined in the obvious way from
\eqref{energy}.  The $\alp = 0$ component of the MKG system imposes
the \emph{Gauss} (or \emph{constraint}) \emph{equation} for initial
data sets, namely
\begin{equation} \label{eq:MKG-gauss} \rd^{\ell} e_{\ell} = - J_{0} =
  \Im(f \overline{g}).
\end{equation}

To eliminate the gauge ambiguity, we add to the above system a single
scalar gauge condition. Here we follow the approach in
\cite{Krieger:2012vj} and work with the \emph{global Coulomb gauge}
\begin{equation}\label{Coulomb}
  \sum_{j=1}^4\partial_j A_j = 0
\end{equation}
where latin summation indices are used for summations which are only
with respect to spatial variables.  Using this gauge, the MKG system
can be written explicitly in the following form
\begin{equation} \label{MKG} \left\{ \begin{aligned} \Box A_i \ &= \
      \mathcal{P}_i J_x
      \\
      \Box_A \phi \ &= \ 0
    \end{aligned}
  \right.
\end{equation}
for the dynamic variables $(A_i,\phi)$.  The operator $\mathcal{P}$ is
the Leray projection onto divergence free vector fields,
\[
\mathcal{P}_{j} v \ = \ I - \rd_{j} \Delta^{-1} \rd^{\ell} v_{\ell}
\]
The second equation in \eqref{MKG} requires also the temporal
component $A_0$, which is determined in an elliptic fashion, together
with its time derivative, by
\begin{equation}
  \begin{aligned}
    \Delta A_0 \ = \ J_0, \qquad \Delta \partial_t A_0 \ = \ \rd^i
    J_i.
    \label{MKGa0}
  \end{aligned}
\end{equation}
Note that the first equation is precisely the Gauss equation.  These
equations uniquely determine both $A_0$ and $\partial_t A_0$ at fixed
time.

Well-posedness theory of MKG at (scaling) sub-critical regularity have
been studied extensively in various gauges. In dimensions $2+1$ and
$3+1$, this system is energy sub-critical, and hence global
well-posedness follows from an appropriate local well-posedness
result; see \cite{Cu, MR649158, MR649159, MR2784611, Klainerman:1994jb,
  Machedon:2004cu, MR579231, Selberg:2010ig} and references
therein. In $\bbR^{1+4}$, almost optimal local well-posedness of
a model problem closely related to MKG and the Yang-Mills system was proved in \cite{Klainerman:1999do};
this result was then further refined in \cite{MR1916561, Ste}. 
For a more detailed survey of
earlier works on MKG, see \cite[Section 1.3]{OT3}.

The subject of this article, as well as its companions \cite{OT1,
  OT3}, is the energy critical MKG-CG problem in 4+1 dimensions.
Given an arbitrary finite energy data set for the MKG problem, there
exists an unique gauge equivalent data set of related size which
satisfies the Coulomb gauge condition; see \cite[Section~3]{OT1}.
Hence the main question question now is to decide whether each finite
energy MKG-CG initial data set can be extended to a global-in-time
solution for the MKG-CG system. This is analogous to the celebrated
\emph{threshold conjecture} for energy critical wave maps, which has
been recently answered in the affirmative \cite{Krieger:2009uy,
  MR2657817, MR2657818, Tao:2008wn, Tao:2008tz, Tao:2008wo,
  Tao:2009ta, Tao:2009ua} (see also \cite{Lawrie:2015rr}).

The small data global well-posedness result was first obtained in high
dimension $n \geq 6$ by Rodnianski-Tao \cite{MR2100060}. The low
dimensional result $n \geq 4$ was obtained more recently by
Krieger-Sterbenz-Tataru~\cite{Krieger:2012vj}. The theorem in
\cite{Krieger:2012vj} asserts the following:
\begin{theorem}[\cite{Krieger:2012vj}]\label{t:small}
  There exists a universal constant $\eps_{\ast} > 0$ such that the
  following hold.
  \begin{enumerate}
  \item (Existence and uniqueness) Let $(a, e, f, g)$ be a $C^\infty$
    Coulomb data set (i.e., $\rd^{\ell} a_{\ell} = 0$) satisfying
    \begin{equation}
      \calE[a, e, f, g] < \epsilon_{*}^{2} \ .
      \label{small-energy} \end{equation}
    Then the
    MKG-CG  system \eqref{Coulomb}-\eqref{MKG} admits a unique global smooth solution
    $({A}, \phi)$ on $\R^{1+4}$ with these data. 
  \item (Continuous dependence) In addition, for every compact time
    interval $J$ containing $0$, the data-to-solution operator extends
    continuously on the set \eqref{small-energy} to a map
    \[
    \calH^{1}(\bbR^{4}) \ni (a, e, f, g) \to ({A}, \phi) \in C(J; \dot
    H^1(\R^4)) \cap \dot C^1(J; L^2(\R^4))
    \]
    where the space $\calH^{1} = \calH^{1}(\bbR^{4})$ of finite energy
    initial data sets is defined by the norm
    \begin{equation} \label{fin-en-id} \nrm{(a, e, f, g)}_{\calH^{1}}
      := \nrm{a}_{\dot{H}^{1}} + \nrm{e}_{L^{2}} +
      \nrm{f}_{\dot{H}^{1}} + \nrm{g}_{L^{2}}.
    \end{equation}
%
  \end{enumerate}
\end{theorem}

The last statement allows us to define the following notion of finite
energy solutions:
\begin{definition} \label{def:adm-sol} Let $I$ be a time interval. We
  define the space $C_{t} \calH^{1}(I \times \bbR^{4})$ by the norm
  \begin{equation*}
    \nrm{(A, \phi)}_{C_{t} \calH^{1}(I \times \bbR^{4})} 
    = \mathop{\mathrm{ess\,sup}}_{t \in I} \bb( \sup_{\mu} \nrm{A_{\mu}[t]}_{\dot{H}^{1} \times L^{2}} + \nrm{\phi[t]}_{\dot{H}^{1} \times L^{2}} \bb) \ .
  \end{equation*}
  We say that a pair $(A, \phi) \in C_{t} \calH^{1}(I \times
  \bbR^{4})$ is an \emph{admissible $C_{t} \calH^{1}$ solution} to MKG
  on $I$ if there exists a sequence $(A^{(n)}, \phi^{(n)})$ of
  classical\footnote{By \emph{classical}, we mean that $A, \phi \in
    \cap_{n, m=0}^{\infty} C_{t}^{m} (I; H^{n})$.} solutions to MKG on $I
  \times \bbR^{4}$ such that
  \begin{equation*}
    \nrm{(A, \phi) - (A^{(n)}, \phi^{(n)})}_{C_{t} \calH^{1}(J \times \bbR^{4})} \quad \hbox{ as } n \to \infty
  \end{equation*}
  for every compact subinterval $J \subseteq I$.
\end{definition}

In the process of proving the above result in \cite{Krieger:2012vj},
stronger spaces $S^1, Y^{1} \subset C(\dot H^1) \cap \dot C^1 (L^2)$
are introduced, and it is shown that the above solutions obeys the
bound
\begin{equation}
  \nrm{A_{0}}_{Y^{1}} + \| (A_{x},\phi)\|_{S^1} \lesssim \nrm{(a, e, f, g)}_{\calH^{1}}
\end{equation}
with a continuous (but not uniformly continuous) data-to-solution map
on each compact time interval.  We provide the definition\footnote{We
  remark that the precise definition of $S^{1}$ differs in
  \cite{Krieger:2012vj}, \cite{OT1} and in the present paper. The
  difference is however minor, and all the theorems stated here hold
  with respect to any of these three definitions. See
  Remark~\ref{rem:S1}.} of the spaces $S^{1}$ and $Y^{1}$ in
Section~\ref{s:sn}.

Our goal, in a sequence of three papers, is to prove that a similar
result holds for all finite energy data.  The three steps in our proof
are as follows:

\begin{description}
\item[Global Coulomb gauge \cite{OT1}] Here we use the above small
  data result to show that the large data problem is locally
  well-posed in the Coulomb gauge, and that the solution can be
  extended for as long as energy concentration does not occur.

\item[Energy dispersed solutions (present paper)] Here we prove a more
  refined continuation criterion, namely that the solution can be
  extended for as long as it remains energy dispersed.  Moreover, if
  the solution already exists up to $t = \infty$, then we prove that
  small energy dispersion implies scattering.

\item[Blow-up analysis \cite{OT3}] Here we complete the proof of the
  large data well-posedness result, showing that no blow-up is allowed
  at the tip of a light cone. We also prove the corresponding
  scattering result.
\end{description}

At least in a broad outline, the second and third step above follow
the scheme successfully developed in \cite{MR2657817} and
\cite{MR2657818} in the context of wave maps. The first step in
\cite{OT1} is specific to the MKG problem, and is due to the long
range effect of the Gauss equation as well as the inherent gauge
ambiguity of MKG. Precisely, in order to truncate a large energy
initial data into small energy data sets, the Gauss equation
$\rd^{\ell} e_{\ell} = \Im(f \overline{g})$ must be taken into
account.  Furthermore, the local gauges given by the small data result
applied to these truncated data differ in their common domains, and
need to be aggregated into a single global Coulomb gauge.  An overview
of the whole sequence is provided in \cite[Sections~2 and 3]{OT3}.

\begin{remark} 
  To understand the issue of gauge invariance clearly, it is
  advantageous to take a more geometric point of view and consider
  $\phi$ as a section of a complex line bundle $L$ with structure
  group $U(1) = \set{e^{i \chi} : \chi \in \bbR}$ over $\bbR^{1+4}$,
  and $A$ as a connection on $L$. Since the base manifold $\bbR^{1+4}$
  is contractible, $L$ is always topologically trivial; hence $\phi$
  can be identified with a $\bbC$-valued scalar function, and
  $A$ with a real-valued 1-form on $\bbR^{4}$ by using the trivial
  connection $\ud$ as a reference. The choice of a gauge then
  corresponds to a particular choice of bases on the fibers to
  describe $(A, \phi)$. This viewpoint is taken to some extent in the
  other papers of the series \cite{OT1, OT3} to facilitate the usage
  of local gauges. In the present paper, however, we need not worry
  about such issues, as we work exclusively in the global Coulomb
  gauge.
\end{remark}

Roughly speaking, the main result in \cite{OT1} is local
well-posedness of MKG-CG for data with any finite energy $E$, with a
lower bound on the lifespan in terms of the \emph{energy concentration
  scale}
\begin{equation} \label{eq:ecs-def} r_{c} = r_{c}[a, e, f, g] := \sup
  \set{r > 0 : \forall x \in \bbR^{4}, \, \calE_{B_{r}(x)} [a, e, f,
    g] < \dlt_{0}(E, \eps_{\ast}^{2})}
\end{equation}
where $\dlt_{0}(E, \eps_{\ast}^{2}) > 0$ is some fixed
function\footnote{In \cite{OT1} we use $\dlt_{0}(E,
  \eps_{\ast}^{2})\approx \eps_{\ast}^6 E^{-2}$ for $E >
  \eps_{\ast}^2$.} , $\eps_{\ast}^{2}$
is the threshold in Theorem~\ref{t:small} and $\calE_{B_{r}(x)}$ is
the energy measured on the ball $B_{r}(x)$ of radius $r$ centered at
$x$. Observe that $r_{c}[a, e, f, g] > 0$ for any $(a, e, f, g) \in
\calH^{1}$.

The result in \cite{OT1} also admits a formulation in terms of the
$S^1$, $Y^{1}$ norms; for that we need a generalization of these norms
to bounded time intervals, which we denote by $S^1[t_0,t_1]$,
$Y^{1}[t_{0}, t_{1}]$ (see Section~\ref{subsec:intervals} for the
definition). The precise statement is as follows.
\begin{theorem}[Large energy local well-posedness theorem in global
  Coulomb gauge \cite{OT1}]\label{t:local}
  Let $(a, e, f, g)$ be an $\calH^{1}$ initial data set satisfying the
  global Coulomb gauge condition $\rd^{\ell} a_{\ell} = 0$ with energy
  $\calE[a, e, f, g] \leq E$. Let $r_{c} = r_{c}[a, e, f, g]$ be
  defined as in \eqref{eq:ecs-def}. Then the following statements
  hold:
  \begin{enumerate}
  \item (Existence and uniqueness) There exists a unique admissible
    $C_{t} \calH^{1}$ solution $(A, \phi)$ to MKG-CG on $[-r_{c},
    r_{c}] \times \bbR^{4}$ with $(a, e, f, g)$ as its initial data.
  \item (A-priori $S^{1}$ regularity) We have the additional regularity
    properties
    \begin{equation*}
      A_{0} \in Y^{1}[-r_{c}, r_{c}], \quad A_{x}, \phi \in S^{1}[-r_{c}, r_{c}].
    \end{equation*}
  \item (Persistence of regularity) The solution $(A, \phi)$ is
    classical if $(a, e, f, g)$ is classical.\footnote{Here, by
      \emph{classical} we mean $a, e, f, g \in \cap_{n=0}^{\infty}
      H^{n}$.}
  \item (Continuous dependence) Consider a sequence $(a^{(n)},
    e^{(n)}, f^{(n)}, g^{(n)})$ of $\calH^{1}$ Coulomb initial data
    sets such that
    \begin{equation*}
      \nrm{(a^{(n)} - a, e^{(n)} - e, f^{(n)}- f, g^{(n)} - g)}_{\calH^{1}} \to 0 \quad \hbox{ as } n \to \infty.
    \end{equation*}
    Then the lifespan of $(A^{(n)}, \phi^{(n)})$ eventually contains
    $[-r_{c}, r_{c}]$, and we have
    \begin{equation*}
      \nrm{A_{0} - A^{(n)}_{0}}_{Y^{1}[-r_{c}, r_{c}]}
      + \nrm{(A_{x} - A^{(n)}_{x}, \phi - \phi^{(n)})}_{S^{1}[-r_{c}, r_{c}]} \to 0 \quad \hbox{ as } n \to \infty. 
    \end{equation*}
  \end{enumerate}

\end{theorem}
In other words, this result says that even if the initial data is
large, we can continue the solution as a global Coulomb solution with
good $S^1$ bounds for as long as energy does not concentrate to
arbitrarily small balls.

Our main result here is based on the notion of energy dispersion
introduced in \cite{MR2657817}. Adapted to our context, the energy
dispersed norm we use is
\begin{equation} \label{ed} \| \phi\|_{ED(t_1,t_2)} = \sup_k 2^{-k}
  \|(P_k\phi, 2^{-k} P_k \phi_t) \|_{L^\infty[(t_1,t_2)\times \R^4]}
\end{equation}
We measure the energy dispersion only for $\phi$, and not for $A$.
%
The main theorem is as follows:

\begin{theorem}[Energy Dispersed Regularity Theorem]\label{t:ed}
  There exist two functions $1\ll F(E)$ and $0 < \epsilon(E)\ll 1$ of
  the energy \eqref{energy} such that the following statement is true:

  If $(A,\phi)$ is an admissible $C_{t} \calH^{1}$ solution to MKG-CG
  on the open interval $(t_1,t_2)$ with energy $\leq E$ and energy
  dispersion at most $\epsilon(E)$, i.e.,
  \begin{equation*}
    \nrm{\phi}_{ED(t_{1}, t_{2})} \leq \epsilon(E),
  \end{equation*}
  then the following a-priori bound holds:
  \begin{equation} \label{S_est-ed} \nrm{ (A_{x}, \phi) }_{S(t_1,t_2)}
    \leq F(E) .
  \end{equation}
\end{theorem}
We remark that \eqref{S_est-ed} implies the bound (see
Theorem~\ref{t:structure})
\begin{equation*}
  \nrm{A_{0}}_{Y^{1}(t_{1}, t_{2})} + \nrm{(A_{x}, \phi)}_{S^{1}(t_{1}, t_{2})} \aleq_{F(E)} 1.
\end{equation*}
We also prove a continuation and scattering result, which may be
applied in conjunction with Theorem~\ref{t:ed}.
\begin{theorem}[Continuation and scattering of solutions with finite
  $S^{1}$ norm] \label{thm:cont-scat} Let 
  $(A, \phi)$ be an admissible $C_{t} \calH^{1}$ solution to MKG-CG on
  $[0, T_{+}) \times \bbR^{4}$, with $0 < T_{+} \leq \infty$, obeying the bound
  \begin{equation*}
    \nrm{A_{0}}_{Y^{1}[0, T_{+})} + \nrm{(A_{x}, \phi)}_{S^{1}[0, T_{+})} < \infty.
  \end{equation*}
  Then the following statements hold.
  \begin{enumerate}
  \item If $T_{+} < \infty$, then $(A, \phi)$ extends to an admissible
    $C_{t} \calH^{1} $ solution with finite $S^{1}$ norm past $T_{+}$.
  \item If $T_{+} = \infty$, then $(A_{x}, \phi)$ scatters as $t \to
    \infty$ in the following sense: There exists a solution
    $(A_{x}^{(\infty)}, \phi^{(\infty)})$ to the linear system
    \begin{equation*}
      \Box A_{j}^{(\infty)} = 0, \quad (\Box + 2 i A^{free}_{\ell} \rd^{\ell}) \phi^{(\infty)} = 0,
    \end{equation*}
    with initial data $A^{(\infty)}_{x}[0], \phi^{(\infty)} [0] \in
    \dot{H}^{1} \times L^{2}$ such that
    \begin{equation*}
      \nrm{A_{x}[t] - A_{x}^{(\infty)}[t]}_{\dot{H}^{1} \times L^{2}}
      + \nrm{\phi[t] - \phi^{(\infty)}[t]}_{\dot{H}^{1} \times L^{2}}
      \to 0 \quad \hbox{ as } T \to \infty. 
    \end{equation*}
    Here $A^{free}_{x}$ is a homogeneous wave with\footnote{This choice is  somewhat robust, in that one can freely perturb $A_x^{free}[0]$ by any function in $\ell^1(H^1 \times L^2)$ 
where $\ell^1$ stands for dyadic summation in frequency. In particular one can take 
$A_x^{free} = A_x^{(\infty)}$.}
  $A^{free}_{x}[0] =
    A_{x}[0]$.
  \end{enumerate}
  Analogous statements hold in the past time direction as well.
\end{theorem}

Our strategy for proving Theorem~\ref{t:ed} is to use an induction on
energy argument; this is imposed by the requirement to renormalize
paradifferential interactions of the solution with itself. This is
somewhat similar to the proof of the corresponding result for wave
maps in \cite{MR2657817}. See also \cite{koch2014dispersive} for an
exposition of this argument in the context of wave maps, and Section 2
in the main paper of the sequence \cite{OT3} for a brief summary of
our strategy.

\begin{remark} 
  We remark that the same results hold in all higher dimensions for
  data in the scale invariant space $\dot H^{\frac{d}2-1} \times \dot
  H^{\frac{d}2-2}$.  We have chosen to restrict our exposition to the
  more difficult case $d = 4$ in order to keep the notations simple,
  but our analysis easily carries over to higher dimension $d \geq
  5$. The main difference in higher dimension is that we no longer
  have a conserved energy which is equivalent to the critical Sobolev
  norm. However, the small energy dispersion guarantees that the
  critical energy is almost conserved.
\end{remark}


\begin{remark} 
  We note that an independent proof of global well-posedness and
  scattering of MKG-CG has been recently announced by
  Krieger-L\"uhrman, following a version of the Bahouri-G\'erard
  nonlinear profile decomposition \cite{MR1705001} and Kenig-Merle
  concentration compactness/rigidity scheme \cite{MR2257393,
    MR2461508} developed by Krieger-Schlag \cite{Krieger:2009uy} for
  the energy critical wave maps.
\end{remark}
\subsection{Notation and Conventions}
We use the asymptotic notation $A \aleq B$ and $A = O(B)$ to mean $A
\leq CB$ for some $C > 0$.  We write $A \ll B$ if the implicit
constant should be regarded as small. The dependence of the constant
is specified by a subscript.

Our convention regarding indices is as follows. The greek indices
$\alp, \bt, \ldots$ run over $0, \ldots, 4$, whereas the latin
indices $i, j, \ldots$ only run over the spatial indices $1, \ldots,
4$. We raise and lower indices using the Minkowski metric, and sum
over repeated upper and lower indices.

We refer to each directional derivative by $\rd_{\mu}$, and the full
space-time gradient by $\nb$. We denote the (gauge) covariant
derivative by $\covD_{\mu} = \rd_{\mu} + i A_{\mu}$. For (Fourier)
multipliers and pseudodifferential operators, it is convenient to use
$D_{\mu} = \frac{1}{i}\rd_{\mu}$, whose symbol is $\xi_{\mu}$.

\subsubsection*{Global small constants}
We introduce a string of globally defined small constants, which are
used in our main argument contained in
Sections~\ref{s:energy}-\ref{s:multi}:
\begin{equation*}
  0 < \dlt_{\ast\ast} \ll \dlt_{\ast} \ll \dlt_{0} \ll \dlt_{1} \ll c \ll \dlt \ll 1.
\end{equation*}
Logically, each constant is chosen to be small enough depending on the
one to the immediate right. For the convenience of the reader, we
summarize the role of each constant as follows: $\dlt$ is the exponent
for dyadic gains in bilinear and multilinear estimates, most which
come from \cite{Krieger:2012vj}; $c$ enters in the gain in large
frequency gaps $m$; $\dlt_{1}$ is used for the gain in small energy
dispersion; $\dlt_{0}$ is reserved for the definition of admissible
frequency envelopes; $\dlt_{\ast}$ and $\dlt_{\ast \ast}$ are the
small constants used in the induction on energy argument in
Section~\ref{s:induction}.

\subsubsection*{Littlewood-Paley projections}
Let $m_{\leq 0}(r)$ be a smooth cutoff that equals $1$ on $\set{r \leq
  1}$ and vanishes on $\set{r \geq 2}$. For $k \in \bbZ$, let $m_{\leq
  k}(r) := m_{\leq 0}(r / 2^{k})$ and $m_{k}(r) := m_{\leq k}(r) -
m_{\leq k-1}(r)$; then $\supp \, m_{k} \subseteq \set{2^{k-1} \leq r
  \leq 2^{k+1}}$ and forms a locally finite partition of unity, i.e.,
$\sum_{k} m_{k} = 1$. Using the space-time Fourier transform $\calF$,
we define various dyadic (or \emph{Littlewood-Paley}) projections as
follows:
\begin{align*}
  P_{k} \varphi = \calF^{-1}[m_{k}(\abs{\xi}) \calF[\varphi]], \quad
  Q_{j} \varphi = \calF^{-1}[m_{j}(\abs{\abs{\tau} - \abs{\xi}})
  \calF[\varphi]], \quad S_{\ell} \varphi =
  \calF^{-1}[m_{k}(\abs{(\tau, \xi)}) \calF[\varphi]].
\end{align*}
We also define $Q^{\pm}_{j} := Q^{\pm} Q_{j}$, where $Q_{\pm} :=
\calF^{-1}[ 1_{[0, \infty)}(\pm \tau) \calF[\varphi]]$ restricts to
the $\pm$ frequency half-space. For an interval $I \subseteq \bbZ$, we
define $P_{I} = \sum_{k \in I} P_{k}$, etc.  At one place, we allow
$P_{k}$ to depend continuously on $k \in \bbR$; see the definition of
$(\tA[0], \tphi[0])$ in Section~\ref{s:induction}.
\subsubsection*{Frequency envelopes}
For some more accurate bounds at various places we need to keep better
track of the frequency distribution of norms. This is done using the
language of frequency envelopes. An \emph{admissible frequency
  envelope} will be any sequence $\{c_k\}_{k \in \mathbb Z}$ of
positive numbers which is slowly varying,
\[
c_j /c_k \leq 2^{\delta_{0} |j-k|}
\]
with a small universal constant $\delta_{0}$. Given such a sequence
and a norm $X$, we define the norm
\[
\|\phi\|_{X_c} = \sup_k c_k^{-1} \| P_{k} \phi\|_{X}.
\]
We say that $c$ is a frequency envelope for the data $(A_{x}[0],
\phi[0])$ if for every $k \in \bbZ$, we have
\begin{equation*}
  \nrm{(P_{k}A_{x}[0], P_{k}\phi[0])}_{\dot{H}^{1} \times L^{2}} \leq c_{k}.
\end{equation*}
Given any $A_{x}[0], \phi[0] \in \dot{H}^{1} \times L^{2}$, we may
construct such a $c$ by convolving with $2^{-\dlt_{0}\abs{\cdot}}$,
i.e.,
\begin{equation*}
  c_{k} := \sum_{k'} 2^{-\dlt_{0}\abs{k - k'}} \nrm{(P_{k'}A_{x}[0], P_{k'}\phi[0])}_{\dot{H}^{1} \times L^{2}} \ .
\end{equation*}
By Young's inequality, we have $\nrm{c}_{\ell^{2}} \aleq
\nrm{(A_{x}[0], \phi[0])}_{\dot{H}^{1} \times L^{2}} $.

\subsection{Structure of the paper}

In Section~\ref{s:energy}, we begin with some elliptic gauge related
fixed time estimates. In particular these will help us relate the full
nonlinear gauge independent energy with the linear energy associated
to the MKG-CG system.

In the following section we switch to space-time analysis, and define
the function spaces $S^1$ and $N$; with minor changes this follows
\cite{Krieger:2012vj}. We also recall some useful estimates from
\cite{Krieger:2012vj}, and add to that some additional properties
related to the interval decomposition of the $S^1$ and $N$ spaces.

In Section~\ref{s:dec} we describe the decomposition of the
nonlinearity, and state the main bilinear and multilinear bounds which
enter into the proof of our main result. To overcome difficulties
related to large data, here we consider two additional classes of
estimates, namely energy dispersed bounds and time divisible
estimates.

In Section~\ref{s:large} we consider MKG waves of finite $S^{1}$ norm,
and we establish further regularity properties for such waves. Based
on these properties, we establish Theorem~\ref{thm:cont-scat}.  We
also consider the special case of MKG waves with small energy
dispersion, and show that some other norms of such waves must also be
small.

Section~\ref{s:induction} contains the proof of our main result in
Theorem~\ref{t:ed}.  This is achieved using an induction of energy
argument, following the principles introduced in \cite{MR2657817}.

The following two sections contain the proof of the bilinear and the
trilinear estimates, where, in addition to results from
\cite{Krieger:2012vj}, we bring in the energy dispersion and divisible
norms.  Heuristically, we will see that the role played by the small
energy dispersion is to improve all the balanced frequency
interactions in the bilinear estimates in Section~\ref{s:bi}.  In the
trilinear estimates in Section~\ref{s:multi}, there are possibly large
unbalanced frequency interactions for which the small energy
dispersion does not seem effective. Nevertheless, we show that the
bulk can be bounded by a time divisible norm. This property allows us
to carry out an induction on energy scheme as in
Section~\ref{s:induction}.

Finally, the last section contains our paradifferential parametrix
construction, based on those in \cite{Krieger:2012vj, MR2100060}.
While very different technically, at the conceptual level this is
similar to the argument in \cite{MR2657817}.  The main idea there is
that a large frequency gap, rather than the small energy dispersion,
is used to control the large paradifferential term.

\subsection*{Acknowledgements}
Part of the work was carried out during the trimester program `Harmonic Analysis and Partial Differential Equations' at the Hausdorff Institute for Mathematics in Bonn; the authors thank the institute for hospitality.
S.-J. Oh is a Miller Research Fellow, and acknowledges the Miller
Institute for support. D. Tataru was partially supported by the NSF
grant DMS-1266182 as well as by the Simons Investigator grant from the
Simons Foundation.

\section{Fixed time elliptic bounds and the energy}
\label{s:energy}
While the energy \eqref{energy} $\calE[A,\phi]$ of the MKG system is
gauge independent, when considering the system in the Coulomb gauge it
is convenient to view $(A_x,\phi)$ as the main dynamic variable, while
$A_0$ and $\partial_t A_0$ are derived quantities obtained via the
equations \eqref{MKGa0}.  Correspondingly, we view
\begin{equation*}
  (A_{x}[0], \phi[0]) = (A_{x}, \rd_{t} A_{x}, \phi, \rd_{t} \phi)(0)
\end{equation*}
as the initial data for the MKG-CG system, and determine the gauge
covariant initial data set $(a, e, f, g)$ via \eqref{MKGa0}. We remark
that $(A_{x}[0], \phi[0])$ can be freely prescribed up to the Coulomb
condition $\rd^{\ell} A_{\ell}(0) = 0$.  In this context, it is
convenient to work with the linear energy
\begin{equation}
  E_{lin} [A_x,\phi](t) 
  = E_{lin} (A_x[t],\phi[t]) 
  := \frac{1}{2} \int \sum_{\substack{\mu=0, \ldots, 4 \\ j = 1, \ldots, 4}} \abs{\rd_{\mu} A_{j}(t)}^{2} + \sum_{\mu = 0, \ldots, 4}\abs{\rd_{\mu} \phi(t)}^{2} \, \ud x.
\end{equation}
In order to justify this, we need to show that $A_0$ is indeed
uniquely determined by $(A_x,\phi)$ at each time, and that the two
energies are in some sense comparable. This is the goal of the main
result here. In the process, we will also obtain some further
solvability estimates for the equations \eqref{MKGa0} for $A_0$ that
will also come in handy in the context of space-time bounds. We have:

\begin{proposition} \label{p:energy} The following statements hold.
  \begin{enumerate}
  \item \label{item:energy:1} Let $(A_{x}, F_{0x}, \phi, \covD_{t}
    \phi)(0)$ be a finite energy initial data set for the MKG-CG
    system.  Then $(A_x[0],\phi[0]) \in \dot{H}^{1} \times L^{2}$ and
    we have the estimate
    \begin{equation}\label{elin<e}
      E_{lin} (A_x[0],\phi[0]) \lesssim \calE[A,\phi]+ \calE[A,\phi]^2,
    \end{equation}
    where $\calE[A, \phi]$ denotes the energy of the initial data set
    $(A_{x}, F_{0x}, \phi, \covD_{t} \phi)(0)$.
  \item \label{item:energy:2} Conversely, suppose that $ (A_x[0],\phi[0]) \in
    {\dot H^1 \times L^2}$. Then there exist unique solutions
    $(A_0,\partial_t A_0) \in {\dot H^1 \times L^2} $ for the
    equations \eqref{MKGa0}, depending smoothly on $ (A_x[0],\phi[0])$
    in the above topologies.  Further, $\calE(A,\phi)$ depends
    smoothly on $(A_x[0],\phi[0]) $, and we have the energy relation
    \begin{equation} \label{e<elin} \calE[A,\phi] \lesssim E_{lin}
      (A_x[0],\phi[0]) + E_{lin} (A_x[0],\phi[0])^2 .
    \end{equation}

  \item \label{item:energy:3} Assume in addition that $\phi[0]$ obeys
    the fixed time energy dispersion bound
    \begin{equation} \label{eq:energy:ED} \nrm{\phi[0]}_{ED} :=
      \sup_{k} 2^{-k} \nrm{(P_{k} \phi, 2^{-k} P_{k} \rd_{t}
        \phi)(0)}_{L^{\infty}(\bbR^{4})} \leq \eps
    \end{equation}
    with $\epsilon \ll_{\calE[A, \phi]} 1$. Then we have
    \begin{equation}
      \calE[A,\phi] =  E_{lin} (A_x[0],\phi[0]) + O_{\calE[A, \phi]}(\epsilon^\frac14) .
    \end{equation}
  \end{enumerate}
\end{proposition}

\begin{proof}
  All estimates here are at fixed time, so we dispense with the time
  variable from the notations. We denote the two energies $\calE[A,
  \phi]$ and $E_{lin}(A_{x}[0], \phi[0])$ simply by $E$ and $E_{lin}$,
  respectively.

 \ref{item:energy:1}). We begin with the spatial components of the
  energy, where we have
  \[
  \frac{1}{2} \sum_{1 \leq j < k \leq 4} \| \rd_{j} A_{k} - \rd_{k}
  A_{j} \|_{L^2}^2 \leq E.
  \]
  Combined with the gauge condition $\rd^{j} A_{j} =0$, this gives the
  linear elliptic bound
  \[
  \frac{1}{2} \nrm{A_{x}}_{\dot{H}^{1}}^{2} = \frac{1}{2} \sum_{1 \leq
    j, k \leq 4} \nrm{\rd_{j} A_{k}}_{L^{2}}^{2}\leq E,
  \]
  and Sobolev embeddings further yield
  \[
  \|A_x\|_{L^4}^2 \lesssim E.
  \]
  On the other hand we also have
  \[
  \frac{1}{2} \| \covD_x \phi\|_{L^2}^2 \leq E.
  \]
  By the diamagnetic inequality and Sobolev embeddings we obtain
  \[
  \| \phi\|_{L^4}^2 \lesssim \| \nabla |\phi|\|_{L^2}^2 \leq \|
  \covD_x \phi\|_{L^2}^2 \lesssim E.
  \]
  Then we can further estimate
  \[
  \| \phi\|_{\dot H^1}^2 = \| \nabla_x \phi\|_{L^2}^2 \leq \| \covD_x
  \phi\|_{L^2}^2 + \| A_x \phi\|_{L^2}^2 \lesssim E+ E^2
  \]

  Next we turn our attention to the temporal components. We first have
  \[
  \frac{1}{2} \| \partial_t A_x - \nabla_x A_0\|_{L^2}^2 \leq E.
  \]
  Applying the divergence and using the Coulomb gauge condition we
  obtain
  \[
  \|A_0\|_{\dot H^1}^2 \lesssim \|\Delta A_0\|_{\dot H^{-1}}^2
  \lesssim E.
  \]
  As the energy $E$ also controls $\| \covD_t \phi \|_{L^2}^2$,
  arguing as above we also obtain
  \[
  \| \partial_t \phi \|_{L^2}^2 \lesssim E+ E^2,
  \]
  which concludes the proof of \eqref{elin<e}.

\bigskip

  \ref{item:energy:2}). We begin with the analysis of the first equation
  in \eqref{MKGa0}, which is rewritten as
  \[
  (-\Delta + |\phi|^2) A_0 = - \Im(\phi \overline{\rd_{t} \phi}).
  \]
  We first need to know that this equation is solvable. More
  generally, we consider the inhomogeneous problem
  \begin{equation}\label{eq-a0f}
    (-\Delta + |\phi|^2) u = f
  \end{equation}
  The solvability of this equation is dealt with via the following
  fixed time lemma:

\begin{lemma}\label{l:ell}
  Consider the equation \eqref{eq-a0f} with $\phi \in \dot H^1$. Set
  $E_0 = \|\phi\|_{\dot H^1}$. Then
  \begin{enumerate}[label=(\alph*)]
  \item \label{item:ell:1} If $f \in \dot H^{-1}$ then there exists a
    unique solution $u \in \dot H^1$, satisfying
    \begin{equation}
      \| u\|_{\dot H^1} \lesssim \|f\|_{\dot H^{-1}} 
    \end{equation}
    Further, the map $(\phi,f) \to u$ is smooth in the $\dot H^1
    \times \dot H^{-1} \to \dot H^1$ topology.

  \item \label{item:ell:2} If $f \in H^{-\frac12}$ then there exists a
    unique solution $u \in \dot H^\frac32$, satisfying
    \begin{equation}
      \| u \|_{\dot H^\frac32} \lesssim_{E_0} \|f\|_{\dot H^{-\frac12}} 
    \end{equation}
    Further, the map $(\phi,f) \to A_0$ is smooth in the $\dot H^1
    \times \dot H^{-\frac12} \to \dot H^\frac32$ topology.

  \item \label{item:ell:3} In addition, for any frequency envelope $c
    \in \ell^2$ we have the bounds
    \begin{equation}
      \| u\|_{\dot H^1_c} \lesssim_{E_0} \|f\|_{\dot H^{-1}_c} , 
      \qquad \| u \|_{\dot H^\frac32_c} \lesssim_{E_0} \|f\|_{\dot H^{-\frac12}_c} 
    \end{equation} 
  \end{enumerate}

\end{lemma}

\begin{proof}
  \ref{item:ell:1} By Sobolev embeddings we have
  \[
  \| |\phi|^2\|_{\dot H^1 \to \dot H^{-1}} \lesssim \|\phi\|_{L^4}^2
  \lesssim \| \phi\|_{\dot H^1}^2
  \]
  Hence the operator $-\Delta +|\phi|^2$ is bounded from ${\dot H^1
    \to \dot H^{-1}}$.  It is also self-adjoint and coercive, so the
  bound
  \[
  \| u\|_{\dot H^1} \leq \|f\|_{\dot H^{-1}}
  \]
  immediately follows. The regularity of the map $(\phi,f) \to u$ is
  obtained in a similar manner, by looking at the linearized equation.

  \medskip

  \ref{item:ell:2} More generally, we will take $f \in
  \dot{H}^{\sigma}$ and prove that we can solve for $u \in
  \dot{H}^{\sigma+2}$ for any $ -2 < \sigma < 0$. This in particular
  easily implies the frequency envelope bounds in part
  \ref{item:ell:3}. By duality it suffices to consider the case $-1
  \leq \sigma < 0$.

  To solve the problem perturbatively in $\dot H^{\sigma+2}$ it
  suffices to construct a multiplier $\Lambda$ so that $\Lambda(\xi)
  \approx_{E_0} |\xi|^{\sigma+1}$ and
  \begin{equation}\label{Lambda}
    \| \Lambda |\phi|^2 \Lambda^{-1} - |\phi|^2\|_{  \dot H^{1} \to \dot H^{-1}} \ll 1
  \end{equation}
  Then we can rewrite the equation as
  \[
  (-\Delta + |\phi|^2) \Lambda u = \Lambda f - ( \Lambda |\phi|^2
  \Lambda^{-1} - |\phi|^2) \Lambda u,
  \]
  and the above estimate allows us to solve the above equation
  perturbatively based on the $\dot H^1$ solvability in part
  \ref{item:ell:1}.

  By duality and a Littlewood-Paley decomposition, \eqref{Lambda}
  would follow if we had the stronger bound
  \begin{equation} \label{eq:ell:pf-goal} I = \sum_{k_i} \left| \int
      \Lambda u_{k_1} \Lambda^{-1} u_{k_2} \phi_{k_3} \phi_{k_4} -
      u_{k_1} u_{k_2} \phi_{k_3} \phi_{k_4} dx \right| \ll \|
    u\|_{\dot H^1}^2
  \end{equation}
  We will denote each summand on the left hand side by $I(k_{1},
  k_{2}, k_{3}, k_{4})$.  To achieve \eqref{eq:ell:pf-goal} we will
  choose $\Lambda$ radial, with the property that $\Lambda(r)$ is
  non-decreasing and
  \[
  \Lambda(s) \leq \Lambda(r) \left(\frac{s}{r}\right)^{\sgm+1}, \qquad
  s > r.
  \]
  Estimating each dyadic contribution using Sobolev embeddings we have
  \[
  \left| \int u_{k_1} u_{k_2} \phi_{k_3} \phi_{k_4} dx \right |
  \lesssim 2^{-(k_{max} - k_{min})} \|u_{k_1}\|_{\dot H^1}
  \|u_{k_2}\|_{\dot H^1} \|\phi_{k_3}\|_{\dot H^1}
  \|\phi_{k_4}\|_{\dot H^1}
  \]
  and similarly
  \[
  \left| \int \Lmb u_{k_1} \Lmb^{-1} u_{k_2} \phi_{k_3} \phi_{k_4} dx
  \right | \lesssim 2^{\sgm(k_{max} - k_{min})} \|u_{k_1}\|_{\dot H^1}
  \|u_{k_2}\|_{\dot H^1} \|\phi_{k_3}\|_{\dot H^1}
  \|\phi_{k_4}\|_{\dot H^1}
  \]
  where $k_{max} = \max\set{k_{1}, \ldots, k_{4}}$, $k_{min} =
  \min\set{k_{1}, \ldots, k_{4}}$.  Hence contributions from widely
  separated frequencies are small.  To measure that, we fix a
  frequency gap parameter $m$ (which will be chosen depending only on
  $E_{0}$) and split
  \[
  I = I_{close} + I_{far} : = \sum_{k_{max} - k_{min} < m} I(k_{1},
  \ldots, k_{4}) + \sum _{k_{max} - k_{min} \geq m} I(k_{1}, \ldots,
  k_{4}).
  \]
  For $I_{far}$ we have
  \[
  I_{far} \lesssim 2^{\sigma \frac{m}{2}} \| u\|_{\dot H^1}^2
  \|\phi\|_{\dot H^1}^{2}
  \]
  which can be made sufficiently small by choosing $m$ large enough
  compared to $E_{0}$.  For $I_{close}$ we use the off-diagonal decay
  to obtain
  \[
  I_{close} \lesssim_{m} \| u\|_{\dot H^1}^2
  \|\phi\|_{B^{1,2}_{\infty}}^{2}
  \]
  Hence only the large dyadic parts of $\phi$ have nontrivial
  contributions.  To account for those, we choose a finite set of
  dyadic indices $K \subset \mathbb Z$ outside of which we have
  \begin{equation} \label{eq:ell:pf-Kc}
    \|\phi\|_{B^{1,2}_{\infty}(K^c)} := \sup_{k \in K^{c}} 2^{k}
    \nrm{\phi_{k}}_{L^{2}} \ll_{E_0} 1.
  \end{equation}
  Note that the number of indices in $K$ can be bounded by a constant
  depending only on $E_{0}$.  Since $I_{close}$ only allows
  interactions of frequencies at most $m$ apart, it is natural to
  expand $K$ by $m$ to $K^{m} := \set{k + k' : k \in K, \abs{k'} \leq
    m}$. Then all unfavorable (i.e., large) interactions occur only
  for frequencies within $K^{m}$, i.e.,
  \begin{align*}
    I_{close} & \leq \sum_{k_{i} : [k_{min}, k_{max}] \cap K = \emptyset} I(k_{1}, \ldots, I_{k_{4}}) + \sum_{k_{i} : [k_{min}, k_{max}] \cap K \neq \emptyset} I(k_{1}, \ldots, I_{k_{4}}) \\
    & \aleq_{m} \nrm{u}_{\dot{H}^{1}}^{2}
    \nrm{\phi}_{B^{1,2}(K^{c})}^{2} + \sum_{k_{i} : [k_{min}, k_{max}]
      \subseteq K^{m}} I(k_{1}, \ldots, I_{k_{4}}).
  \end{align*}
  The first term on the last line is small enough thanks to
  \eqref{eq:ell:pf-Kc}. The second term can be eliminated altogether
  by refining the choice of $\Lmb$.  Precisely, we set $\Lmb(r)$ to be
  a piecewise smooth function which is constant for $\log_{2} r \in
  K^{2m}$ and equals an appropriate constant multiple of $r^{\sgm+1}$
  outside. Then it is easy to check that $I(k_{1}, \ldots, k_{4}) = 0$
  if $k_{1}, k_{2} \in K^{m}$; hence \eqref{eq:ell:pf-goal}
  follows. Furthermore, since $m$ and the number of indices in $K$ are
  bounded by $E_{0}$, it follows that $\Lmb(r) \approx_{E_{0}}
  r^{\sgm+1}$ as required.
%
%

  \medskip

  \ref{item:ell:3} For $\sgm_{0} = -1, -\frac{1}{2}$ and $f = f_k$  we claim that
  \begin{equation} \label{eq:ell-off-diag}
    \nrm{\phi_{j}}_{\dot{H}^{\sgm_{0}+2}} \aleq_{E_{0}} 2^{-\dlt
      \abs{j - k}} \nrm{f_{k}}_{\dot{H}^{\sgm_{0}}} \, .
  \end{equation}
  In fact, a similar bound holds for any $-2 < \sgm_{0} < 0$ with
  $\dlt > 0$ depending on $\sgm$.  By linearity, we may fix $k$, and
  by scaling (which leaves $E_{0}$ invariant), we may assume that $k =
  0$. Then the bound \eqref{eq:ell-off-diag} follows by applying
  \ref{item:ell:2} with $- 2 < \sgm < \sgm_{0}$ and $\sgm_{0} < \sgm <
  0$ to control the solution $\phi$ in upper and lower Sobolev spaces,
  which implies that $\phi_{j}$ decays in $L^{2}$ away from $j =
  0$. \qedhere
\end{proof}

We now continue the proof of part \ref{item:energy:2} of
Proposition~\ref{p:energy}.  From part \ref{item:ell:1} of the above
lemma we obtain the estimate
\[
\| A_0\|_{\dot H^1} \lesssim \|\phi\|_{\dot H^1} \|\partial_t
\phi\|_{L^2} \lesssim E_{lin} .
\]
Then, using the embedding $\dot H^1 \subset L^4$, we directly obtain
the estimate \eqref{e<elin}.

\bigskip

\ref{item:energy:3}). Comparing $E$ with $E_{lin}$ we have
\[
E = E_{lin} + E_{lin}^\frac12 O(\| A \phi\|_{L^2} + \| \nabla
A_0\|_{L^2})+ O(\| A \phi\|_{L^2}^2 + \| \nabla A_0\|_{L^2}^2)
\]
therefore it suffices to establish the bounds
\begin{equation}\label{E-diff}
  \| A \phi\|_{L^2} \lesssim_{E_{lin}} \epsilon^\frac14, \qquad   \| \nabla A_0\|_{L^2}
  \lesssim _{E_{lin}} \epsilon^\frac14.
\end{equation}
The first is easily obtained using the standard Littlewood-Paley
trichotomy.  For high-low interactions we have
\[
\| A_j \phi_k\|_{L^2} \lesssim 2^{k-j} \| A_j\|_{L^2} \|\phi_k\|_{ED},
\qquad j > k
\]
For low-high interactions we have
\[
\| A_j \phi_k\|_{L^2} \lesssim \| A_j\|_{L^8} \|\phi_k\|_{L^\frac83}
\lesssim 2^{\frac18(j-k)} \| A_j\|_{\dot H^1} \| \phi
\|^{\frac34}_{\dot H^1} \|\phi\|_{ED}^{\frac14} , \qquad j < k
\]
Finally for high-high interactions we have
\[
\| P_j (A_k \phi_k)\|_{L^2} \lesssim 2^{\frac12(j-k)} \| A_j\|_{\dot
  H^1} \| \phi \|^{\frac12}_{\dot H^1} \|\phi\|_{ED}^{\frac12} ,
\qquad j \leq k .
\]
In all cases we have favorable off-diagonal decay, so the $l^2$ dyadic
summation for the output is inherited from $A$. Hence the first bound
in \eqref{E-diff} follows.

For the second bound in \eqref{E-diff} we use the lemma to reduce it
to
\[
\| \phi \partial_t \phi\|_{\dot H^{-1}} \lesssim_{E_{lin}}
\epsilon^{\frac14} .
\]
The argument for this is similar to the one above, and is left for the
reader.
\end{proof}

\section{Space-time function spaces}
\label{s:sn}

\subsection{The $S^1$, $N$, $Z$ and $Y^{1}$ spaces} 
We begin our discussion with the function spaces introduced in \cite{Krieger:2012vj},
namely $S^1$ for the MKG waves $(A,\phi)$ and $N$ for the inhomogeneous terms 
in both the $\Box$ and the $\Box_A$ equation. These are spaces of functions defined 
over all of $\R^{n+1}$, together with the related spaces $S$ and $N^*$.
They are all defined via their dyadic subspaces, with norms 
\[
\| \phi\|_{X}^2 = \sum_k \|\phi_k\|_{X_k}^2, \qquad X \in \{ S,S^1,N\} 
\]
We recall the definition of their norms. With minor modifications at high modulations, we follow
\cite{Krieger:2012vj}. For $N_k$ we set
\begin{equation}
  N_k \ = \ {L^1 L^2} +  X_1^{0,-\frac{1}{2}},
\label{n}
\end{equation}
where
\begin{equation*}
\nrm{\phi}_{X^{s, b}_{r}} := \bb( \sum_{k} \big( \sum_{j} (2^{sk} 2^{bj} \nrm{P_{k} Q_{j} \phi}_{L^{2} L^{2}})^{r} \big)^{\frac{2}{r}}\bb)^{\frac{1}{2}} .
\end{equation*}
The $N_{k}$ norm is the same as in \cite{Krieger:2012vj}.
%

The $S_k$ space is a strengthened version of $N_k^*$,
\begin{equation} \label{s-vs-n}
 X_1^{0,\frac{1}{2}}   \subseteq S_k\subseteq L^\infty L^2  \cap X_\infty^{0,\frac{1}{2}} = N_{k}^{\ast}, 
\end{equation}
while $S_k^1$ is defined as 
\begin{equation}\label{s1}
  \| \phi\|_{S_k^1} = \| \nabla \phi\|_{S_k} 
  + 2^{-\frac{k}2} \|\Box \phi\|_{L^2 L^{2}}
 + 2^{-\frac{4k}9} \|\Box \phi\|_{L^{\frac95} L^2} .
\end{equation}
Compared to \cite{Krieger:2012vj} we have loosened the $\ell^{1}$ summability of the $\Box^{-1} L^{2}L^{2}$ norm and added the $\Box^{-1} L^{\frac95} L^2$
norm above. Both of these modifications are of interest only at high modulations.  The
exact exponent $9/5$ is not really important, for our purposes it only
matters that it is less than two and greater than $5/3$.

\begin{remark} \label{rem:S1}
In \cite{OT1}, yet another definition of the $S^{1}$ norm is employed, namely 
\begin{equation*}
\nrm{\phi}_{S_k^1} = \| \nabla \phi\|_{S_k} 
  +2^{-\frac{k}2} \|\Box \phi\|_{L^2 L^{2}}.
\end{equation*} 
Our justification for keeping the same notation $S^{1}$ (besides notational simplicity) is that the difference among these three definitions is minor. 
For a solution to MKG-CG, one can easily pass from one definition to another using the high modulation bounds in Propositions \ref{p:axnl-high-mod} and \ref{p:phi-high-mod}. In particular, in every theorem stated in the introduction, statements with respect to one of these definitions of $S^{1}$ easily implies those with respect to others.
\end{remark}

We now recall the definition of the space $S_{k}$ from \cite{Krieger:2012vj}. The space $S_k$ scales like free waves with $L^{2} \times \dot{H}^{-1}$ initial data, and is defined by
\begin{equation}
  \| \phi\|_{S_k}^2 \ = \ \| \phi\|_{S^{str}_k}^2 + \|\phi\|_{S^{ang}_k}^2
  +  \|\phi\|_{X_\infty^{0,\frac{1}{2}}}^2  \ , \notag
\end{equation}
where:
\begin{equation}\label{str_and_defn}
\begin{aligned}
\nrm{\phi}_{S^{str}_{k}} \ =& \sup_{2 \leq q, r, \leq \infty, \ \frac{1}{q}+ \frac{3/2}{r}\leq \frac{3}{4}} 2^{(\frac{1}{q}+\frac{4}{r}-2)k} \nrm{(\phi, 2^{-k} \rd_{t} \phi)}_{L^q L^r} \ , \quad
  \nrm{\phi}_{S^{ang}_{k}} = \sup_{l < 0} \nrm{\phi}_{S^{ang}_{k, k+2l}} \ , \\
 \nrm{\phi}_{S^{ang}_{k, j}}^{2}
 =& \sum_{\omega}\| P^\omega_l
    Q_{<k+2l}\phi\|_{S_k^\omega(l)}^2 \qquad \hbox{ with } l = \lceil \frac{j-k}{2} \rceil. 
\end{aligned}
\end{equation}
The $S^{str}_{k}$ norm controls all admissible Strichartz norms on $\bbR^{1+4}$. The $\omg$-sum in the definition of $S^{ang}_{k, j}$ is over a covering of $\bbS^{3}$ by caps $\omg$ of diameter $2^{l}$ with uniformly finite overlaps, and the symbols of $P^{\omg}_{l}$ form a smooth partition of unity associated to this covering.
The angular sector norm $S_{k}^{\omg}(l)$ combines the null frame space as in wave maps \cite{Tao:2001gb, MR1827277} with additional square-summed norms over smaller radially directed blocks $\calC_{k'}(l')$ of dimensions $2^{k'} \times (2^{k'+l'})^{3}$. We first define

\begin{align}
  \| \phi\|_{P\!W^\pm_\omega(l)} \ &=\ \inf_{\phi=\int \!\!
    \phi^{\omega'} } \int_{|\omega-\omega'|\leqslant 2^{l}}
  \| \phi^{\omega'} \|_{L^2_{\pm\omega' }(L^\infty_{(\pm\omega')^\perp}
    )} d\omega' \ , \notag\\
  \| \phi\|_{N\!E} \ &= \ \sup_\omega \|\snabla_\omega
    \phi\|_{L^\infty_{ \omega} (L^2_{\omega^\perp})} \ , \notag
\end{align}
where the norms are with respect to $\ell_\omega^\pm = t\pm
\omega\cdot x$ and the transverse variable in the $(\ell^{\pm}_{\omg})^{\perp}$ hyperplane (i.e., constant $\ell^{\pm}_{\omg}$ hyperplanes). Moreover,  $\snabla_\omega$ denotes tangential derivaties on the $(\ell^+_\omega)^\perp$ hyperplane. As in \cite{Krieger:2012vj}, we set:
\begin{multline}
  \| \phi\|_{S_k^\omega(l)}^2 \ = \ \| \phi\|_{S_k^{str}}^2 +
  2^{-2k}\|\phi\|_{N\!E}^2 + 2^{-3k}\sum_\pm
  \| Q^\pm  \phi\|_{P\!W^\mp_\omega(l) }^2 \\
  + \sup_{\substack{k'\leqslant k ,   l'\leqslant 0\\
      k+2l\leqslant k'+l'\leqslant k+l }} \sum_{\mathcal{C}_{k'}(l') }
  \Big( \|P_{\mathcal{C}_{k'}(l')} \phi\|_{S_k^{str}}^2
  + 2^{-2k}\| P_{\mathcal{C}_{k'}(l')} \phi\|_{N\!E}^2\\
  + 2^{-2k'-k}\|P_{\mathcal{C}_{k'}(l')} \phi\|_{L^2(L^\infty)}^2 +
  2^{-3(k'+l')}\sum_\pm \| Q^\pm P_{\mathcal{C}_{k'}(l')}
    \phi\|_{P\!W^\mp_\omega(l) }^2 \Big) \ , \label{Sl_def}
\end{multline}
where the $\calC_{k'}(l')$ sum runs over a covering of $\bbR^{4}$ by the blocks $\calC_{k'}(l')$ with uniformly finite overlaps, and the symbols of $P_{\calC_{k'}(l')}$ form an associated partition of unity.
We also define the smaller space $S_k^\sharp \subset S_k$ (see the
bound \eqref{lin-S} below) by
\[
\| u \|_{S_k^\sharp} = \| \Box u\|_{N_k} + \|\nabla u\|_{L^\infty L^2}.
\]
On occasion we need to separate the two characteristic cones
$\{ \tau = \pm |\xi|\}$. Thus we define the spaces $N_{k, \pm}$, $S^{\sharp}_{k, \pm}$ and $N^{\ast}_{k, \pm}$ in an obvious fashion, so that
\begin{equation*}
N_k = N_{k,+} \cap  N_{k,-}, \quad
S_k^\sharp = S_{k,+}^\sharp + S_{k,-}^\sharp, \quad
N^*_k = N^*_{k,+} + N_{k,-}^* \ .
\end{equation*}

Next we describe an auxiliary space of the type $L^1(L^\infty)$ which
will be useful for decomposing the nonlinearity:
\begin{equation}
  \|\phi\|_{Z} \ =\  \sum_k  \| P_k\phi \|_{Z_k} \ , \ \	
  \| \phi\|_{Z_k}^2 \ =\  \sup_{l<C} 
  \sum_{\omega }2^{l}\|P^\omega_lQ_{k+2l} \phi \|_{L^1(L^\infty)}^2
  \ . \notag
\end{equation}
Note that as defined this space already scales like $\dot{H}^1$ free
waves. In addition, note the following useful embedding which is a direct
consequence of Bernstein's inequality:
\begin{equation}
  \Box^{-1}  L^1(L^2) \ \subseteq \ Z \ . \label{B_embed}
\end{equation}
Finally, the function space for $A_0$ is simple to describe,
since the $A_0$ equation is elliptic:
\begin{equation}
\|A_0\|_{Y^1}^2 = \| \nabla A_0\|_{L^\infty L^2}^2 + \| \nabla A_0\|_{L^2 
\dot H^{\frac12}}^2  \ , \notag
\end{equation}
where we recall that $\nb$ denotes the full space-time gradient.

Let $E$ denote the linear energy space, i.e.,
\begin{equation*}
E = \dot{H}^{1} \times L^{2} .
\end{equation*}
One of the results in \cite{Krieger:2012vj} asserts that we have linear solvability for the 
d'Alembertian in our setting.  

\begin{proposition}
 We have  the linear estimates
 \begin{align}
 \nrm{\nb \phi}_{S}
 &\lesssim \  \| \phi[0]\|_{E} + \|\Box \phi\|_{N}  \ , 	\label{lin-S} \\
 \|\phi\|_{S^1} \ 
 &\lesssim \  \| \phi[0]\|_{E} + \|\Box \phi\|_{N \cap L^{2} \dot{H}^{-\frac{1}{2}} \cap L^{\frac{9}{5}} \dot{H}^{-\frac{4}{9}}}  \ . 	\label{lin}
  \end{align}
\end{proposition}
Here \eqref{lin-S} is the embedding $S^{\sharp} \subset S$, whereas \eqref{lin} follows immediately from \eqref{lin-S}.

\subsection{Interval localization} \label{subsec:intervals} So far, we
have described the global setting in \cite{Krieger:2012vj}.  However,
in this article we work on compact time intervals, therefore we also
need suitable interval localized function spaces.  This is not
straightforward, since our function spaces are defined using
modulation localizations, which are nonlocal in time. To start with,
we take the easy way out and define
\begin{equation}\label{sn(i)-def}
  \| \phi\|_{S^1[I]} = \inf_{\phi = \tphi_{|I}}  \|\tphi\|_{S^1},  \qquad
  \| f \|_{N[I]} = \inf_{f = \tilde f_{|I}}  \|\tilde f\|_{N}
\end{equation}
However, the next result allows us to simplify somewhat these
definitions:

\begin{proposition} \label{p:intervals}
  \begin{enumerate}
  \item \label{item:intervals:1} Consider a time interval $I$, and its
    characteristic function $\chi_I$. Then we have the bounds
    \begin{equation}\label{sn(i)}
      \|\chi_I \phi\|_{S} \lesssim \|\phi\|_{S}, \qquad \|\chi_I f\|_{N} \lesssim \| f\|_{N},
    \end{equation}
    The latter norm is also continuous as a function of $I$. We also
    have the linear estimates
    \begin{align}
      \|\nb \phi\|_{S[I]} \ &\lesssim \| \phi[0]\|_{E} + \|\Box \phi\|_{N[I]} , \label{lin-S(i)} \\
      \|\phi\|_{S^1[I]} \ &\lesssim \| \phi[0]\|_{E} + \|\Box
      \phi\|_{(N \cap L^{2} \dot{H}^{-\frac{1}{2}} \cap
        L^{\frac{9}{5}} \dot{H}^{-\frac{4}{9}}) [I]} . \label{lin(i)}
    \end{align}
  
  \item \label{item:intervals:2} Consider any partition $I = \bigcup
    I_k$. Then the $N$ norm is interval divisible, i.e.
    \begin{equation}\label{n-div}
      \sum_k \| f \|_{N[I_k]}^2 \lesssim \|f\|_{N[I]}^2 
    \end{equation}
    and the $S$ and $S^1$ norms are interval square summable, i.e.
    \begin{equation}\label{s-sum}
      \| \phi\|_{S[I]}^2 \lesssim \sum_{k} \|\phi\|_{S[I_k]}^2,
      \qquad \| \phi\|_{S^1[I]}^2 \lesssim \sum_{k} \|\phi\|_{S^1[I_k]}^2
    \end{equation}
  \end{enumerate}
\end{proposition}

We remark that a consequence of part \ref{item:intervals:1} is that,
up to equivalent norms, we can replace the arbitrary extensions in
\eqref{sn(i)-def} by the zero extension in the $N$ case, respectively
by homogeneous waves with $(\phi, \rd_{t} \phi)$ as the data at each
endpoint outside $I$ in the $S^1$ case.

\begin{proof}
  \ref{item:intervals:1}). It suffices to prove the desired bounds for
  frequency dyadic pieces of $\phi$ and $f$.  In the $N$ case it also
  suffices to work with the space $L^1 L^2 + X_1^{0,-\frac{1}{2}}$.
  But in this case this is exactly the proof of (158) in
  \cite{MR2657817}, where just steps 1 and 2 are needed.

  By duality, we have the same bound for $L^\infty L^2 \cap
  X_\infty^{0,\frac{1}{2}}$, which is a part of the $S$ norm. We now
  consider the remaining parts of the $S$ norm. The only difficulty is
  with the $S^{ang}_k$ norm, due to the modulation
  localization. Fixing a modulation scale $j = k+2l$, we consider
  either $Q^+_{<j} (\chi_{I} \phi)$ or $Q^-_{<j} (\chi_I\phi)$. There
  are two cases to consider:

  (i) Short intervals, $|I| < 2^{-j}$. Then
  \[
  \| Q^+_{<j} (\chi_{I} \phi)\|_{S_{k, j}^{ang}} \lesssim \| Q^+_{<j}(
  \chi_{I} \phi)\|_{X^{0,\frac12}_1} \lesssim 2^{\frac{j}2} \|
  \chi_{I} \phi\|_{L^2 L^{2}} \lesssim \|\phi\|_{L^\infty L^2}
  \]

  (ii) Long intervals, $|I |> 2^{-j}$. Then we write
  \begin{equation} \label{eq:interval:pf-1} Q^+_{<j} \chi_{I} \phi =
    Q_{< j-30} \chi_{I} Q^+_{< j-2} \phi + Q^+_{<j} ( Q_{> j-30}
    \chi_{I} Q^+_{< j-2} \phi) + Q^+_{<j} ( \chi_{I} Q^+_{> j-2} \phi)
  \end{equation}
  For the first term we use the bound
  \begin{equation*}
    \nrm{Q_{<j-30} \chi_{I} Q^{+}_{<j-2} \phi}_{S^{ang}_{k, j}}
    \aleq \nrm{Q_{<j-30} \chi_{I}}_{L^{\infty}} \nrm{\phi}_{S_{k}}
    \aleq \nrm{\phi}_{S_{k}}
  \end{equation*}
  which was proved\footnote{Technically speaking,
    \cite[Lemma~7.1]{OT1} is stated for $\chi_{I}$ which decays in
    space, but we may simply approximate $\chi_{I}$ by smooth
    compactly supported functions.} in \cite[Lemma~7.1]{OT1}. The
  other two terms in \eqref{eq:interval:pf-1} are estimated in $L^2
  L^{2}$ as in (i): For the second term, we use
  \[
  \begin{split}
    \| Q^+_{<j} ( Q_{> j-30} \chi_{I} Q^+_{< j-2} \phi) \|_{S_{k,
        j}^{ang}} \lesssim & \ 2^{\frac{j}2} \| Q_{> j-30} \chi_{I}
    Q^+_{< j-2} \phi\|_{L^2 L^{2}} \lesssim 2^{\frac{j}2} \| Q_{>
      j-30} \chi_{I}\|_{L^2} \|Q^+_{< j-2} \phi\|_{L^\infty L^2}
    \\
    \lesssim & \ \|\phi\|_{L^\infty L^2}.
  \end{split}
  \]
  In the last inequality, we used the bound
  \begin{equation*}
    \nrm{Q_{j} \chi_{I}}_{L^{2}} \aleq 2^{-\frac{j}{2}},
  \end{equation*} 
  which follows from Plancherel in $t$ and the fact that the Fourier
  transform of the $\chi_{I}$ is a suitable rescaling and modulation
  of $\sin \tau / \tau$. Finally, the third term in
  \eqref{eq:interval:pf-1} is treated as follows:
  \[
  \| Q^+_{<j} ( \chi_{I} Q^+_{> j-2} \phi) \|_{S_{k, j}^{ang}}
  \lesssim 2^{\frac{j}2} \| Q^+_{> j-2} \phi\|_{L^2 L^{2}} \lesssim \|
  \phi\|_{S_{k}} \, .
  \]


  \ref{item:intervals:2}). The $N$ bound \eqref{n-div} is exactly as in
  Proposition 5.4 in (159) in \cite{MR2657817}.  The $S^1$ bound
  \eqref{s-sum} reduces easily to the corresponding $S$ bound.  The
  bound \eqref{s-sum} for the $N^*$ part of the $S$ norm follows by
  duality from \eqref{n-div}.  Of the remaining components of the $S$
  part we have the same difficulty as in part \ref{item:intervals:1},
  namely with the modulation localizations occurring in the
  $S_k^{ang}$ norms.  The solution is also the same as in part
  \ref{item:intervals:1}; precisely that for each modulation scale $j$
  we split the intervals into short and long, and estimate the two
  contributions as above:

  (i) Short intervals, $|I_m| < 2^{-j}$. Then the modulation
  localization operator $Q^+_{<j} $ can cause significant overlapping
  of outputs coming from inputs in different intervals $I_m$. Hence
  our strategy is to harmlessly discard $Q^+_{<j} $ as follows:
  \[
  \| Q^+_{<j}(\sum_m \chi_{I_m} \phi)\|_{S_{k, j}^{ang}}^2 \lesssim
  2^{\frac{j}2} \| \sum_m \chi_{I_m} \phi\|_{L^2}^2 \lesssim \sum_{m}
  \|\chi_{I_m}\phi\|_{L^\infty L^2}^{2}
  \]

  (ii) Long intervals, $|I_m |> 2^{-j}$. Then for each $m$ we use a
  partition of unity adapted to $I_m$ to write
  \[
  1 = \tilde \chi_{I_m} + \sum_{l>0} \tilde \chi_{I_m}^l
  \]
  where $\tilde \chi_{I_m} $ is a smooth cutoff selecting a $2^{-j}$
  neighborhood of $I_m$, while $ \chi_{I_m}^l $ select the region at
  distance $2^{-j+l}$ from $I_m$. Correspondingly, we write
  \[
  Q^+_{<j} \sum_{m} \chi_{I_m} \phi = \sum_{m} \tilde \chi_{I_m}
  Q^+_{<j} (\chi_{I_m} \phi) + \sum_{l > 0} \sum_m \tilde \chi_{I_m}^l
  Q^+_{<j} (\chi_{I_m} \phi)
  \]
  Now we estimate each of the sums above. For the first one we use the
  fact that the bump functions $ \tilde \chi_{I_m} $ have finite
  overlapping to write
  \[
  \|\sum_{m} \tilde \chi_{I_m} Q^+_{<j} (\chi_{I_m} \phi) \|_{S_{k,
      j}^{ang}}^2 \lesssim \sum_{m} \| \tilde \chi_{I_m} Q^+_{<j}
  (\chi_{I_m} \phi) \|_{S_{k, j}^{ang}}^2 \lesssim \sum_{m} \|
  Q^+_{<j} (\chi_{I_m} \phi) \|_{S_{k, j}^{ang}}^2
  \]

  which suffices thanks to part \ref{item:intervals:1}. On the other
  hand, in the second sum, for each $l$ we have at most $2^l$
  overlapping bump functions. So we obtain
  \[
  \begin{split}
    \|\sum_m \tilde \chi_{I_m}^l Q^+_{<j} (\chi_{I_m} \phi) \|_{S_{k,
        j}^{ang}}^2 \lesssim & \ 2^l \sum_m \| \tilde \chi_{I_m}^l
    Q^+_{<j} (\chi_{I_m} \phi) \|_{S_{k, j}^{ang}}^2 \lesssim 2^l
    2^{j} \sum_m \| \tilde \chi_{I_m}^l Q^+_{<j} (\chi_{I_m} \phi)
    \|_{L^2}^{2}
    \\
    \lesssim & \ 2^l \sum_m 2^{-2 Nl} \| \chi_{I_m} \phi \|_{L^\infty
      L^2}^{2}
  \end{split}
  \]
  which again suffices. Here, at the last stage, we have used the fact
  that the operator $\tilde \chi_{I_m}^l Q^+_{<j} \chi_{I_m}$ has a
  $2^{-\frac{j}2} 2^{-Nl}$ norm from $L^\infty L^{2}$ to $L^2 L^{2}$,
  which is due to the separation of supports of the two cutoff
  functions. \qedhere
\end{proof}

Last but not least, we consider the effect of extension on some of our
Strichartz or energy dispersed norms; the role of these norms in our
work will be explained in Section~\ref{subsec:smallness}.  For an
interval $I$ we denote by $\chi_I^k$ a generalized cutoff function,
which is adapted to the $2^k$ frequency scale:
\[
\chi_I^k(t) = (1+ 2^k \dist(t,I))^{-N} .
\]
For a function $\phi_I$ in $I$ we denote by $\phi_I^{ext}$ its
extension as homogeneous waves. Then we have:

\begin{proposition}\label{p:ext}
  Assume that $|I| \geq 2^{-k}$. Then the following estimates hold for
  $\phi_I$ localized at frequency $2^k$: {
    \begin{align}
      \| \chi_I^k (\phi_I^{ext}, 2^{-k} \rd_{t} \phi_{I}^{ext})
      \|_{L^{p} L^{q}}
      \lesssim & \| \phi_I \|_{L^{p} L^{q}[I]} + 2^{(\frac{1}{2} - \frac{1}{p} - \frac{4}{q}) k} \nrm{\Box \phi_{I} }_{L^{2} L^{2}[I]} , \label{inhom-ext-pre} \\
      \chi_I^k(t) \| \phi_I^{ext}(t) \|_{ED} \lesssim & \| \phi_I
      \|_{ED[I]} , \label{ed-ext}
    \end{align}
    where $(p, q)$ is any pair of admissible Strichartz exponents on
    $\bbR^{1+4}$.  }
\end{proposition}

\begin{proof}
  By rescaling, we can take $k = 0$. It suffices to consider the case
  when $I$ has the minimal length, i.e., $\abs{I} = 1$; the general
  case is then easily recovered by applying the same proof to
  unit-length intervals at each end of $I$. By translation invariance,
  we may take $I = [0, 1]$.

  We first consider the bound \eqref{inhom-ext-pre}. It suffices to
  show that for any Strichartz norm $L^p L^q$ and $\phi_I$ localized
  at frequency $1$ we have
  \begin{equation}\label{ds-ext}
    \| \chi_I^{0} (\phi_I^{ext}, \rd_{t} \phi_{I}^{ext}) \|_{L^p L^q} \lesssim \|  \phi_I \|_{L^p L^q[I]} + \|\Box \phi_I\|_{L^2[I]} .
  \end{equation}
  By the inhomogeneous Strichartz estimates, this reduces to the case
  when $\Box \phi = 0$.

  We prove this in two steps. First we notice that the Cauchy data at
  time $0$ satisfies
  \begin{equation}\label{ds-exta}
    \| \phi_{I}[0]\|_{L^q} \lesssim  \|  \phi_I \|_{L^p L^q[I]}
  \end{equation}
  Equivalently, we have to show that for functions $f_{\pm}$ localized
  at frequency $1$ we have
  \[\| f_{\pm} \|_{L^q} \lesssim \|e^{it|D|} f_+ + e^{-it|D|}
  f_-\|_{L^p([0,1];L^q)}
  \]
  We may easily find $t_{1} \in [0, 1/3]$ and $t_{2} \in [2/3, 1]$
  such that the $L^{q}$ norm of $e^{i t_{j} \abs{D}} f_{+} + e^{-i
    t_{j} \abs{D}} f_{-}$ is bounded by the right hand side. Then the
  desired conclusion follows from the linear independence of the
  symbols $(e^{i t_{j} \abs{D}}, e^{- i t_{j} \abs{D}})$ for $j=1,2$.

  Secondly, we have the bound
  \begin{equation} \label{ds-extb} \| (\phi,\partial_t \phi)(t)
    \|_{L^q} \lesssim (1+t^8)^{\frac12 - \frac1p} \| (\phi,\partial_t
    \phi)(0)\|_{L^q}, \qquad 2 \leq q \leq \infty
  \end{equation}
  This is trivial for $q = 2$.  In the case $q = \infty$, for a fixed
  $(t, x) \in \bbR^{1+4}$, by finite speed of propagation we may
  truncate the initial data for $\phi[0]$ outside a ball of radius
  $C(1 + t)$ without changing $(\phi, \rd_{t} \phi)(t,x)$. Then the
  desired bound for $\abs{(\phi, \rd_{t} \phi)(t,x)}$ follows from
  H\"older's inequality (to control the energy with the $L^{\infty}$
  norm), the energy estimate and Bernstein's inequality.

  Putting together \eqref{ds-exta} and \eqref{ds-extb} we obtain
  \eqref{ds-ext}.  Finally, the bound \eqref{ed-ext} follows from
  \eqref{ds-extb} with $q = \infty$.

\end{proof}

\subsection{Smallness: energy dispersion and
  divisibility} \label{subsec:smallness} Since our goal is to work
with large data MKG-CG solutions, it is crucial to have at our
disposal sufficient tools to gain smallness in appropriate settings.
One such source of smallness in this article is the \emph{energy
  dispersion}, which is used as an a-priori bound. Another venue for
gaining smallness is to partition the time in finitely many
subintervals, on each of which the norm is small. A space-time norm
for which this procedure works is said to be \emph{divisible}. In this
short subsection, we provide heuristic explanation of both concepts
and their use in our context.

We start by discussing the use of energy dispersion as a source of
smallness.  For application, it is useful to quantify the smallness of
the energy dispersion norm $\nrm{\cdot}_{ED}$ in comparison with the
norm $\nrm{\cdot}_{S^{1}}$, which is stronger and have the same
scaling. We therefore define:
\begin{definition} \label{def:eps-ed} For any interval $I \subseteq
  \bbR$ and $\veps >0$, we say that $\phi \in S^{1}[I]$ is
  \emph{$\veps$-energy dispersed} (with respect to the $S^{1}$ norm)
  if
  \begin{equation} \label{eq:eps-ed} \nrm{\phi}_{ED(I)} \leq \veps
    \nrm{\phi}_{S^{1}[I]}.
  \end{equation}
\end{definition}
Observe that the $S^{1}$ norm is stronger than the $ED$ norm by
Bernstein's inequality, i.e.,
\begin{equation*}
  \nrm{\phi}_{ED(I)} \aleq \nrm{\nb \phi}_{L^{\infty} L^{2}[I]} \aleq \nrm{\phi}_{S^{1}[I]}.
\end{equation*}
Hence the dimensionless quantity $\veps > 0$ can be thought of as
measuring the improvement relative to Bernstein's inequality.

Roughly speaking, small energy dispersion improves balanced frequency
bilinear interactions. In \cite{MR2657817}, this improvement was
obtained by interpolating the $ED$ norm with the Wolff-Tao bilinear
estimate in $L^{p} L^{p}$ with $p < 2$. In the present setting, as we
have stronger dispersion due to higher dimensionality, we can achieve
the same end by simply interpolating the $ED$ norm with Strichartz
norms. Indeed, the following linear lemma covers essentially all of
our usage of small energy dispersion:
\begin{lemma} \label{l:eps-ed-int} Let $\phi \in S^{1}[I]$ be
  $\veps$-energy dispersed. Then for any $k \in \bbZ$ and any non-sharp
  pair of Strichartz exponents $(p, q) \in [2, \infty]$ (i.e.,
  $\frac{2}{p} + \frac{3}{q} < \frac{3}{2}$ and $p \neq
  2$), we have
  \begin{equation*}
    \sup_{k} \nrm{(P_{k} \phi, 2^{-k} P_{k} \rd_{t} \phi)}_{L^{p} L^{q} [I]} \leq \veps^{\dlt_{1}}
2^{(2- \frac{1}{p}-\frac{4}{q})k} \nrm{\phi}_{S^{1}[I]},
  \end{equation*}
  where $\dlt_{1} = \dlt_{1}(p, q) > 0$.
\end{lemma}
As discussed, this lemma easily follows by interpolating the $ED$ norm
with the $S^{str}_{k}$ component of the $S^{1}$ norm, which is
possible thanks to the non-sharpness of $(q, r)$. We will often
combine this lemma with H\"older's inequality to gain smallness for
multilinear estimates.

We now turn to the use of divisibility in our work.
The bound \eqref{n-div} shows that the $N$ norm is divisible.
However, the $S^1$ norm is not\footnote{See however the result in
  Theorem~\ref{t:structure}\ref{item:structure:5}} divisible, and this
is a source of trouble. Our workaround is to introduce a weaker norm,
denoted $\DS$, which collects a subset of the components of the $S^1$
norm which are divisible. This is defined as follows:
\begin{equation}
  \| \phi \|_{\DS[I]} = \| (|D|^\frac16 \phi, |D|^{-\frac{5}{6}} \rd_{t} \phi)\|_{L^2 L^{6}[I]} 
  + \| (|D|^\frac56 \phi, |D|^{-\frac{1}{6}} \rd_{t} \phi)\|_{L^{10} L^\frac{30}{13} [I]} 
  + \||D|^{-\frac12} \Box u \|_{L^2 L^{2}[I]} .
\end{equation}
Precisely, we may include here any divisible Strichartz norm as long
as we stay away from the $L^\infty L^2$ endpoint (i.e., the energy).
To gain divisibility for $A_{0}$, we use the norm
\begin{equation*}
  \nrm{\nb A_{0}}_{L^{2} \dot{H}^{\frac{1}{2}}[I]} = \nrm{(A_{0}, \rd_{t} A_{0})}_{L^{2} \dot{H}^{\frac{3}{2}} \times L^{2} \dot{H}^{\frac{1}{2}}[I]},
\end{equation*}
which is a divisible component of the $Y^{1}$ norm.

By Proposition~\ref{p:ext}, we see that the homogeneous wave extension
$\phi^{ext}_{I}$ of a function $\phi_{I}$ in $I$ obeys the bound
\begin{equation} \label{inhom-ext} \nrm{\chi_{I}^{k}
    \phi_{I}^{ext}}_{\DS} \aleq \nrm{\phi_{I}}_{\DS[I]}
\end{equation}
when $\phi_{I}$ is localized at frequency $2^{k}$ and $\abs{I} \geq
2^{-k}$.

Our strategy will be to use as much as possible the divisible norms
(such as $\DS$ or $L^{2} \dot{H}^{\frac{3}{2}}$) in our bilinear and
multilinear estimates, and try to prove smallness for the remainder.

\section{The decomposition of the nonlinearity}
\label{s:dec}

Recalling the definition of the currents $J_\alpha = - \Im
(\phi\overline{D_{\alpha} \phi})$ we write the MKG-CG system again
here as:
\begin{subequations}
  \begin{align}
    \Box A_i \ &= \ \mathcal{P}_i J_x \ ,
  \label{MKGa}\\
    \Box_A \phi \ &= \ 0  \  \label{MKGphi}
  \end{align}
\end{subequations}
The second equation also includes $A_0$ and $\partial_t A_0$, which are 
obtained from the elliptic equations
\begin{equation}
  \begin{aligned}
    \Delta A_0 \ = \  J_0, \qquad 
    \Delta \partial_t A_0 \ = \ \nabla^i J_i.
    \label{MKGa0+}
  \end{aligned}
\end{equation}
We now discuss the bounds for each of the components of MKG equation.
For the purpose of this section, all analysis is done in a fixed time interval $I=[0,T]$.

We remark that for the most part,  bilinear and trilinear estimates for the nonlinearities
were already proved in \cite{Krieger:2012vj} in the context of the small data problem.
Our goal here is to understand when and how we can regain smallness 
in the study of the large data. As discussed in Section~\ref{subsec:smallness}, there are two such sources of smallness:

a) Arising from norm divisibility for either $A$ or $\phi$, where a
large but divisible norm is made small by selecting a suitable time
interval partition. Here we seek to use the $\DS$ part of the $S^1$
norm to measure the bulk of the nonlinearities.

b) Arising from small energy dispersion for $\phi$. This is often
considered coupled with the additional high modulation bound
\begin{equation}\label{box-phi}
\| \Box \phi\|_{L^2 \dot H^{-\frac12}} \leq \veps^{\dlt_{1}} \|\phi\|_{S^{1}}
\end{equation}
which for MKG-CG solutions is an easy consequence of the $\veps$-energy dispersion; see Theorem~\ref{t:structure-ed} below.

Two easy ways to gain the two types of estimates above is by using 
suitable Strichartz estimates. Precisely, for divisibility we need $L^pL^q$ 
norms with $p < \infty$. On the other hand for energy dispersion we need
$p > 2$, as well as non-sharp pairs of exponents $(p,q)$, so that Lemma~\ref{l:eps-ed-int} is applicable. Often we can 
fulfill both at once, and prove the two types of estimates simultaneously.

\subsection{The terms $A_x$.}
We decompose $A_i$ into a free and a nonlinear component,
\[
A_i = A^{free}_i + A^{nl}_i
\]
where 
\[
\Box A_i^{free} = 0, \qquad A_i^{free}[0] = A_{i}[0] 
\]
and 
\[
\Box A_i^{nl} =  \mathcal{P}_i J_x, \qquad A_i^{nl}[0] = 0 
\]

Given the expression of the currents $J_\alpha = - \Im(\phi \partial_\alpha
\bar \phi) + A_{\alpha} |\phi|^2$, we will think of $A_x^{nl}$ given by the
above equation as a multilinear expression in $\phi$ and $A$, i.e., $A_{i}^{nl}=
\A_i(\phi,\phi,A)$. We can also extend this to a symmetric quadratic
form in the first two variables, $\A_i(\phi_1,\phi_2,A)$.  We also
split it into a quadratic and a cubic part,
\begin{align*}
\A_i(\phi_1,\phi_2,A) 
=& \A_i^{2}(\phi_1,\phi_2) + \A_x^{3}(\phi_1,\phi_2,A) \\
=& - \frac{1}{2} \Box^{-1} \calP_{i}(\phi_{1} \rd_{x} \overline{\phi_{2}} + \overline{\rd_{x} \phi_{1}} \phi_{2}) 
	+ \frac{1}{2} \Box^{-1} \calP_{i} (\phi_{1} \overline{\phi_{2}} A_{x} + \overline{\phi_{1}} \phi_{2} A_{x}),
\end{align*}
where $\Box^{-1} f$ denotes the solution to the inhomogeneous wave equation $\Box u = f$ with $u[0] = 0$.
The $N$ bounds we need for $\Box A_x$ are as follows:
\begin{proposition} \label{p:ax}
Let $\phi_{1}, \phi_{2}, A$ be test functions defined on a time interval $I$ containing $0$. 
\begin{enumerate}
\item \label{item:ax:1} For all admissible frequency envelopes $c,d,e$ we have 
\begin{equation} \label{axnl}
\| \Box \A_x^{2}(\phi_1,\phi_2) \|_{N_{cd}[I]}
+\| \nb \A_x^{2}(\phi_1,\phi_2) \|_{S_{cd}[I]} 
\lesssim \|\phi_1\|_{S^1_c[I]} \|\phi_2\|_{S^1_d[I]}
\end{equation}
respectively
\begin{equation} \label{axnl3}
\| \Box \A_x^{3}(\phi_1,\phi_2,A) \|_{L^1 L^2_{cde}[I]}
+\| \nb \A_x^{3}(\phi_1,\phi_2, A) \|_{S_{cde}[I]} 
\lesssim \|\phi_1\|_{\DS_c[I]} \|\phi_2\|_{\DS_d[I]} \|A\|_{\DS_e[I]}
\end{equation}
 
\item \label{item:ax:2} Further, for each $m > 0$ there is a decomposition 
\[
\A_x^{2}(\phi_1,\phi_2) = \A_{x,small}^{2}(\phi_1,\phi_2) +   \A_{x,large}^{2}(\phi_1,\phi_2)
\]
so that we have
\begin{equation} \label{axnl-small}
\| \Box \A_{x,small}^{2}(\phi_1,\phi_2) \|_{N_{cd}[I]}
+\| \nb \A_{x,small}^{2}(\phi_1,\phi_2) \|_{S_{cd}[I]}  
\lesssim 2^{-cm}  \|\phi_1\|_{S^1_c[I]} \|\phi_2\|_{S^1_d[I]}
\end{equation}
respectively
\begin{equation} \label{axnl-large}
\| \Box \A_{x,large}^{2}(\phi_1,\phi_2) \|_{N_{cd}[I]} 
+ \| \nb \A_{x,large}^{2}(\phi_1,\phi_2) \|_{S_{cd}[I]}  
\lesssim 2^{Cm}  \|\phi_1\|_{\DS_c[I]} \|\phi_2\|_{\DS_d[I]}
\end{equation}

\item \label{item:ax:3} In addition, if  $\phi_1$  is $\veps$-energy dispersed and satisfies 
\eqref{box-phi} then
\begin{equation} \label{axnl-eps}
\| \Box \A_x^{2}(\phi_1,\phi_2) \|_{N_{c}[I]}
+ \| \nb \A_x^{2}(\phi_1,\phi_2) \|_{S_{c}[I]} 
\lesssim \veps^{\dlt_{1}}  \|\phi_1\|_{S^1[I]} \|\phi_2\|_{S^1_c[I]}
\end{equation}
and
\begin{equation} \label{axnl-eps3}
\| \Box \A_x^{3}(\phi_1,\phi_2,A) \|_{N_{de}[I]}
+ \| \nb \A_x^{3}(\phi_1,\phi_2,A) \|_{S_{de}[I]} 
\lesssim  \veps^{\dlt_{1}}   \|\phi_1\|_{S^1[I]} \|\phi_2\|_{S^1_d[I]}\|A\|_{S^1_e[I]}
\end{equation}
\end{enumerate}
\end{proposition}
\begin{remark} \label{rem:ax-S}
The $\nrm{\nb(\cdot)}_{S}$ norm bounds follow immediately from the control of $\nrm{\Box(\cdot)}_{N}$ thanks to \eqref{lin-S(i)} and the fact that the initial data vanish for $\bfA^{2}_{x}, \bfA^{3}_{x}$. As we see from \eqref{s1}, this norm is slightly weaker than the main `solution norm' $S^{1}$ for high modulations; nevertheless these bounds will prove useful in the proof of the multilinear estimates in Section~\ref{s:multi}. 
\end{remark}
\begin{remark} \label{rem:disp-freq-loc}
Given a test function $\phi_{1}$ on $I$ which is $\veps$-energy dispersed and obeys \eqref{box-phi}, the bounds \eqref{axnl-eps} and \eqref{axnl-eps3} still hold with the same right hand sides if we replace $\phi_{1}$ by its frequency projection (e.g., $P_{< k_{\ast}} \phi_{1}$ or $P_{\geq k_{\ast}} \phi_{1}$) on the left hand side. This fact will be evident from the proof. The same remark applies to all the other estimates in this section that rely on $\veps$-energy dispersion.
\end{remark}
This proposition is proved in Section~\ref{s:bi}. 

We also state high modulation bounds for $\Box A_{x}$, which do not require a null structure nor an extra decomposition:
\begin{proposition} \label{p:axnl-high-mod}
Let $\phi_{1}, \phi_{2}, A$ be test functions defined on a time interval $I$ containing $0$. 
For all admissible frequency envelopes $c, d, e$ we have
\begin{align} 
	\nrm{\Box \A_{x}^{2}(\phi_{1}, \phi_{2})}_{(L^{2} \dot{H}^{-\frac{1}{2}} \cap L^{\frac{9}{5}} \dot{H}^{-\frac{4}{9}})_{cd}[I]} 
	\aleq & \nrm{\phi_{1}}_{DS^{1}_{c}[I]} \nrm{\phi_{2}}_{DS^{1}_{d}[I]} \label{axnl-hm} \\
	\nrm{\Box \A_{x}^{3}(\phi_{1}, \phi_{2}, A)}_{(L^{2} \dot{H}^{-\frac{1}{2}} \cap L^{\frac{9}{5}} \dot{H}^{-\frac{4}{9}})_{de}[I]} 
	\aleq & \nrm{\phi_{1}}_{DS^{1}_{c}[I]} \nrm{\phi_{2}}_{DS^{1}_{d}[I]} \nrm{A}_{DS^{1}_{e}[I]} \label{axnl-hm3} 
\end{align}
In addition, if $\phi_{1}$ is $\veps$-energy dispersed, then
\begin{align} 
	\nrm{\Box \A_{x}^{2}(\phi_{1}, \phi_{2})}_{(L^{2} \dot{H}^{-\frac{1}{2}} \cap L^{\frac{9}{5}} \dot{H}^{-\frac{4}{9}})_{d}[I]} 
	\aleq & \veps^{\dlt_{1}} \nrm{\phi_{1}}_{S^{1}[I]} \nrm{\phi_{2}}_{S^{1}_{d}[I]} \label{axnl-hm-eps} \\
	\nrm{\Box \A_{x}^{3}(\phi_{1}, \phi_{2}, A)}_{(L^{2} \dot{H}^{-\frac{1}{2}} \cap L^{\frac{9}{5}} \dot{H}^{-\frac{4}{9}})_{de}[I]} 
	\aleq & \veps^{\dlt_{1}} \nrm{\phi_{1}}_{S^{1}[I]} \nrm{\phi_{2}}_{S^{1}_{d}[I]} \nrm{A}_{S^{1}_{e}[I]} \label{axnl-hm-eps3} 
\end{align}
\end{proposition}
\begin{proof} 
The whole proposition is a simple consequence of Bernstein's inequality, non-sharp Strichartz estimates and Lemma~\ref{l:eps-ed-int}. For instance, the $L^{2} \dot{H}^{-\frac{1}{2}}$ norm of the Littlewood-Paley piece $P_{k} \Box \bfA^{2}_{x}(\phi_{k_{1}}, \phi_{k_{2}})$ (where $\phi_{k_{i}}$ is a shorthand for $P_{k_{i}} \phi_{i}$) is bounded as follows:
\begin{align*}
	\nrm{P_{k} \Box \bfA^{2}_{x}(\phi_{k_{1}}, \phi_{k_{2}})}_{L^{2} \dot{H}^{-\frac{1}{2}}[I]}
	\aleq & \nrm{P_{k} (\phi_{k_{1}} \rd_{x} \overline{\phi}_{k_{2}})}_{L^{2} \dot{H}^{-\frac{1}{2}}[I]}
		+ \nrm{P_{k} (\rd_{x}  \overline{\phi}_{k_{1}} \phi_{k_{2}})}_{L^{2} \dot{H}^{-\frac{1}{2}}[I]} \\
	\aleq & 2^{-\dlt \max \set{ \abs{k - k_{i}} } } 
			\prod_{i=1,2} \nrm{\abs{D}^{\frac{9}{16}}\phi_{k_{i}}}_{L^{4} L^{\frac{64}{21}}[I]}
\end{align*}
where the off-diagonal gain arises from applying Bernstein's inequality to the lowest frequency. As $(4, \frac{64}{21})$ is a non-sharp Strichartz estimate, this bound suffices for both \eqref{axnl-hm} and \eqref{axnl-hm-eps} (via Lemma~\ref{l:eps-ed-int}). Similarly, for $P_{k} \Box \bfA^{3}_{x}(\phi_{k_{1}}, \phi_{k_{2}}, A_{k_{3}})$ (where $A_{k_{3}} = P_{k_{3}} A$), we have
\begin{align*}
	\nrm{P_{k} \Box \bfA^{3}_{x}(\phi_{k_{1}}, \phi_{k_{2}}, A_{k_{3}})}_{L^{2} \dot{H}^{-\frac{1}{2}}[I]}
	\aleq & 2^{-\dlt \max \set{ \abs{k - k_{i}} } } 
			\bb( \prod_{i=1,2} \nrm{\abs{D}^{\frac{1}{6}}\phi_{k_{i}}}_{L^{6} L^{4}[I]} \bb) \nrm{\abs{D}^{\frac{1}{6}}A_{k_{3}}}_{L^{6} L^{4}[I]}
\end{align*}
The argument for the $L^{\frac{9}{5}} \dot{H}^{-\frac{4}{9}}$ norm is analogous. \qedhere
\end{proof}

\subsection{The term $A_0$.}
Here we consider bounds for both $A_0$ and its time derivative, which
are given by \eqref{MKGa0+}.  The first equation can be written in a
more explicit form as
\begin{equation}\label{eq-a0}
(-\Delta + |\phi|^2) A_0 = \Im(\phi \partial_t \bar \phi)
\end{equation}
which was analyzed earlier in Lemma~\ref{l:ell}.  As an immediate
corollary of Lemma~\ref{l:ell} we obtain the following estimate for
$A_0$:
\begin{equation}\label{a0-start}
\|A_0\|_{L^\infty \dot H^{1}_{c^{2}}[I]} + \|A_0\|_{L^2 \dot H^{\frac32}_{c^{2}}[I]} \lesssim_{\nrm{\phi}_{L^{\infty} \dot{H}^{1}[I]}} \| \phi\|^2_{S^1_{c}[I]} \ .
\end{equation}
Given this bound, we return to the equations \eqref{MKGa0+} and view
them simply as Laplace equations, whose solutions are quadratic
expressions in $\phi$,
\[
A_0= \A_0(\phi,\phi,A) \qquad \partial_t A_0= \partial_t \A_0(\phi,\phi,A)
\]
which are given by 
\begin{align*}
\A_0(\phi,\phi,A_0)=& \A_{0}^{2}(\phi, \phi) + \A_{0}^{3}(\phi, \phi, A_{0}) \\
=& -\Delta^{-1} \Im(\phi \rd_{t} \overline{\phi}) + \Delta^{-1} (\phi \overline{\phi} A_{0}), \\
\partial_0 \A_0(\phi,\phi,A_{x}) 
=& \rd_{0} \A_{0}^{2}(\phi, \phi) + \rd_{0} \A_{0}^{3}(\phi, \phi, A_{x}) \\
=& - \Delta^{-1} \partial^j   \Im (\phi \rd_j \overline{\phi}) + \Delta^{-1} \partial^{j} (\phi \overline{\phi} A_{j}) .
\end{align*}
We also extend these to symmetric quadratic forms in the first two variables $\phi_{1}, \phi_{2}$.
 Our estimates for $A_{0}$ and $\rd_{0} A_{0}$ are as follows:
\begin{proposition}\label{p:a0}
Let $\phi_{1}, \phi_{2}, A$ be test functions defined on a time interval $I$. Let $c, d, e$ be admissible frequency envelopes.
\begin{enumerate}
\item For any exponent $2 \leq p \leq \infty$, we have 
\begin{equation} \label{a0}
\begin{split}
\| \A_0^{2}(\phi_1,\phi_2) \|_{L^p \dot H^{1+\frac{1}{p}}_{cd}[I]} \lesssim & \  \|\phi_1\|_{\DS_c[I]} \|\phi_2\|_{\DS_d[I]}
\\
\| \A_0^{3}(\phi_1,\phi_2,A_0) \|_{L^p \dot H^{1 + \frac{1}{p}}_{cde}[I]} \lesssim & \  \|\phi_1\|_{\DS_c[I]} \|\phi_2\|_{\DS_d[I]}
\|A_0\|_{L^{p} \dot H^{1+\frac{1}{p}}_{e} [I]} \ ,
\end{split}
\end{equation}
\begin{equation} \label{d0a0}
\begin{split}
\| \partial_t \A_0^{2}(\phi_1,\phi_2) \|_{L^p \dot H^{\frac{1}{p}}_{cd}[I]} \lesssim & \
\|\phi_1\|_{\DS_c[I]} \|\phi_2\|_{\DS_c[I]}
\\
\| \partial_t \A_0^{3}(\phi_1,\phi_2,A_x) \|_{L^p \dot H^{\frac{1}{p}}_{cde}[I]} \lesssim & \
\|\phi_1\|_{\DS_c[I]} \|\phi_2\|_{\DS_d[I]} \nrm{A_{x}}_{\DS_{e}[I]} \ .
\end{split}
\end{equation}

\item In addition, if  $\phi_1$  is $\veps$-energy dispersed then
\begin{equation} \label{a0-disp}
\begin{split}
\| \A_0(\phi_1,\phi_2,A_0) \|_{L^{p} \dot H^{1+\frac{1}{p}}_c[I]} \lesssim & \ \veps^{\dlt_{1}} \|\phi_1\|_{S^1[I]} \|\phi_2\|_{S^1_c[I]}(1+
\|A_0\|_{L^{p} \dot H^{1+\frac1p} [I]} )
\\
\| \partial_t \A_0(\phi_1,\phi_2,A_x) \|_{L^p \dot H^{\frac{1}{p}}_c[I]} \lesssim & \ \veps^{\dlt_{1}}
\|\phi_1\|_{S^1[I]} \|\phi_2\|_{S^1_c[I]}(1+\|A_x\|_{S^1[I]} ) \ .
\end{split}
\end{equation}
\end{enumerate}
\end{proposition}
We omit the proof, as it is similar to Proposition~\ref{p:axnl-high-mod}. 


\subsection{The $\phi$ equation}

We will split  the $\phi$ equation into  a leading 
 order paradifferential approximation plus a perturbative part.
The paradifferential approximation  is given by
\begin{equation}
  \Box_{A}^{p,m} \ = \ \Box + 2 i \sum_k P_{<k-m} A^{\alpha} \partial_{\alpha} P_{k} \ . 
\label{para_wave}
\end{equation}
Here we retain the freedom to choose $m$ arbitrarily large later on.
Then the operator $\Box_A$ is written as
\begin{equation}\label{phieq-nonlin}
 \Box_A =  \Box_{A}^{p,m} +  {\mathcal M}^{m}_A
\end{equation}
where $\mathcal M_A^{m} = \mathcal M_{A}^{m, 2} + \mathcal M_{A, A}^{m, 3}$ is given by 
\begin{equation}\label{Mdef}
\begin{split}
\mathcal  M_A^{m, 2} \psi  = &   \ 2 i  \sum_k
  P_{\geq k-m}A^{\alpha}\partial_\alpha P_k  \psi -  i \partial_t A_0 \psi \ ,\\
\mathcal  M_{A, B}^{m, 3} \psi = &  A^{\alpha} B_\alpha \psi \ .
\end{split}
\end{equation}

The operator $\mathcal M_A^m$ will play a perturbative role in our
analysis, just based on the $S^1$ and $L^2 \dot H^\frac12$ bounds for
the coefficients $A_x$, $\nabla A_0$.  Precisely, for its quadratic and cubic parts we have:
\begin{proposition}\label{p:ma}
Let $A, B, \psi$ be test functions defined on a time interval $I$. Let $c, d, e$ be admissible frequency envelopes.

\begin{enumerate}
\item \label{item:ma:1} 
The cubic part $\mathcal  M_{A, B}^{m,3} $ satisfies the bound
\begin{equation}\label{man}
  \| \mathcal  M_{A, B}^{m,3} \psi\|_{N_{c d e}[I]} 
  \lesssim 2^{Cm} \| (A_x, \nb A_{0})\|_{(\DS \times L^{2} \dot{H}^{\frac{1}{2}})_{c}[I]}
  				\| (B_x, \nb B_{0})\|_{(\DS \times L^{2} \dot{H}^{\frac{1}{2}})_{d}[I]} \|\psi\|_{S^1_e}
\end{equation}
where $\| (A_x, \nb A_{0})\|_{(\DS \times L^{2} \dot{H}^{\frac{1}{2}})_{c}[I]}$ is a shorthand for $(\|A_x\|_{\DS_{c}[I]} + \|  \nabla A_0\|_{L^2 \dot H^{\frac12}_{c}[I]})$.

\item \label{item:ma:2}
The quadratic part $\mathcal  M_A^{m,2} $
admits a decomposition 
\begin{equation}\label{man-dec}
  \mathcal M_A^{m,2} =  \mathcal M_{A,small}^{m,2}+ \mathcal M_{A,large}^{m,2}
\end{equation}
so that we have
\begin{equation}\label{man-small}
  \| \mathcal  M_{A,small}^{m,2} \psi\|_{N_{cd}[I]} 
  \lesssim 2^{-cm} \| (A_x, \nb A_{0})\|_{(S^{1} \times L^{2} \dot{H}^{\frac{1}{2}})_{c}[I]} \|\psi\|_{S^1_d[I]}
\end{equation}
while 
\begin{equation}\label{man-large}
  \| \mathcal  M_{A,large}^{m,2} \psi\|_{N_{cd}[I]} 
    \lesssim 2^{Cm} \| (A_x, \nb A_{0})\|_{(\DS \times L^{2} \dot{H}^{\frac{1}{2}})_{c}[I]} \|\psi\|_{S^1_d[I]} \ .
\end{equation}

\item \label{item:ma:3} Further, if $\psi$ is $\veps$-energy dispersed and obeys \eqref{box-phi}, then the
quadratic and cubic parts of $\mathcal M_A^m $ satisfy
\begin{equation}\label{man-disp}
  \| \mathcal  M_A^{m,2} \psi\|_{N_c[I]} 
   \lesssim \bb( 2^{Cm} \veps^{\dlt_{1}} \nrm{A_{x}}_{S^{1}_{c}[I]} + \| \nb A_{0}\|_{L^{2} \dot{H}^{\frac{1}{2}}_{c}[I]} \bb) \|\psi\|_{S^1[I]} 
\end{equation}
as well as 
\begin{equation}\label{man-dispa}
  \| \mathcal  M_{A, B}^{m,3} \psi\|_{N_{cd}[I]} 
   \lesssim 2^{Cm} \veps^{\dlt_{1}} \| (A_x, \nb A_{0})\|_{(S^{1} \times L^{2} \dot{H}^{\frac{1}{2}})_{c}[I]}
  				\| (B_x, \nb B_{0})\|_{(S^{1} \times L^{2} \dot{H}^{\frac{1}{2}})_{d}[I]} \|\psi\|_{S^1[I]} \ .
\end{equation}
\end{enumerate}
\end{proposition}

This result is proved in Section~\ref{s:bi}. We remark the different
roles of $\mathcal M_{A,small}^{m,2}$ versus $\mathcal
M_{A,large}^{m,2} $. The first one is small, and thus directly
perturbative. The second is not small, but is instead estimated using
only a divisible norm of $A$; thus we can partition time into finitely
many intervals where it is small.

\bigskip

Our next goal is to compare the operators $ \Box_{A}^{p,m}$ and
$\Box_{A^{free}}^{p,m}$, where we use the convention $A^{free}_{0} = 0$. 
We define the bilinear operator $\Diff_{A}^{m} \psi$ by
\begin{equation*}
	\Diff^{m}_{A} \psi = \Box^{p, m}_{A} - \Box = 2 i \sum_{k} P_{<k-m} A^{\alp} \rd_{\alp} P_{k} \psi.
\end{equation*}
Hence we have the decomposition
\begin{equation*}
	\Box_{A}^{p, m} = \Box_{A^{free}}^{p, m} + \Diff^{m}_{A^{nl}} 
\end{equation*}
  For the last term, we no longer use only the $S^1$ and
$L^2 \dot H^\frac12$ bounds for $A_x$ and $\nabla A_0$, but instead 
we rely on the fact that $A_x$ and $A_0$ come from the equations \eqref{MKGa}, \eqref{MKGa0+}.
Thus, we replace $\Diff^{m}_{A^{nl}} \psi$ with the multilinear operator
\[
\Diff_{\A}^m(\phi,\phi,A) =  2i \sum_k P_{< k-m}  \A^{\alpha} (\phi,\phi,A) 
\partial_\alpha P_k \ .
\]
As before, we extend this operator to a symmetric quadratic form in the first two inputs.  
For the multilinear operator $\Diff_{\A}^{m}(\phi, \phi, A) \psi$, we have the following estimates:
\begin{proposition}\label{p:diff}
Let $\phi_{1}, \phi_{2}, \psi, A$ be test functions on a time interval $I$ containing $0$. Let $c, d, e$ be admissible frequency envelopes.
\begin{enumerate}
\item \label{item:diff:1} 
The quadratic and cubic parts of the operator $
  \Diff_{\A}^m(\phi,\phi,A) $ satisfy the bounds
\begin{equation}\label{diffa}
  \| \Diff_{\A}^{m,2}(\phi_1,\phi_2)  \psi\|_{N_f[I]} \lesssim 
 \|\phi_1\|_{S^1_c[I]} \|\phi_2\|_{S^1_d[I]} \|\psi\|_{S^1_e[I]}
\end{equation}
respectively
\begin{equation}\label{diffaa}
  \| \Diff_{\A}^{m,3} (\phi_1,\phi_2,A)  \psi\|_{N_f[I]} \lesssim 
 \nrm{(A_{x}, \nb A_{0})}_{(DS^{1} \times L^{2} \dot{H}^{\frac{1}{2}})[I]}
 \|\phi_1\|_{\DS_c[I]}\|\phi_2\|_{\DS_d[I]} \|\psi\|_{S^{1}_e[I]}
\end{equation}
where 
\begin{equation} \label{diffa-freqenv}
f(k) = e(k) \| c_{\leq k-m}\|_{\ell^2} \| d_{\leq k-m}\|_{\ell^2} \ .
\end{equation}

\item \label{item:diff:2} Further, for each $m > 0$,  $\Diff_{\A}^{m,2}$  admits a decomposition 
\begin{equation}\label{diffa-dec}
 \Diff_{\A}^{m,2} = \Diff_{\A,small}^{m,2} + \Diff_{\A,large}^{m,2}   
\end{equation}
so that $\Diff_{\A,small}^{m,2} $ satisfies a better bound,
\begin{equation}\label{diffa-small}
  \|\Diff_{\A,small}^{m,2} (\phi,\phi)  \psi\|_{N_c[I]} \lesssim 2^{-cm} \|\phi\|_{S^1[I]}^2 \|\psi\|_{S^1_c[I]}
\end{equation}
while $  \Diff_{\A,large}^{m,2}   $ is estimated directly in a divisible norm,
\begin{equation}\label{diffa-large}
\| \Diff_{\A,large}^{m,2} (\phi,\phi)  \psi\|_{N_c[I]}  \lesssim 2^{Cm} \|\phi\|_{\DS[I]}^2\|\psi\|_{S^1_c[I]}
\end{equation}
\end{enumerate}
\end{proposition}

This result is proved in Section~\ref{s:multi}. Again, we remark that
the large part is estimated using a divisible norm, which can be made
small by subdividing the time interval. We also remark that here we
are concerned with unbalanced frequency interactions, so the energy
dispersion plays no role.

\bigskip

For the gradient terms in $\Box_{A^{free}}^{p,m}$  we only have the following 
dyadic bound from \cite{Krieger:2012vj}:

\begin{proposition}
For a divergence free homogeneous wave $A$ we  have the dyadic bound
\begin{equation}\label{null(Afree)}
	\| P_{k} A^j \partial_j P_{l} \psi\|_{N_l} \lesssim \|P_{k} A[0]\|_{E} \|  P_{l} \psi\|_{S^1}, \qquad \hbox{ for }k < l .
\end{equation} 
\end{proposition}

Due to the lack of $\ell^2$ dyadic summation with respect to $k$ in the above bound,
 the gradient terms in $\Box_{A^{free}}^{p,m}$ need to be treated in a nonperturbative 
manner.    This issue was addressed in the small data case  in \cite{Krieger:2012vj} 
by  constructing  a microlocal parametrix. Here we adopt the same strategy, but 
using a different source for the smallness, namely the frequency gap $m$:

\begin{theorem}\label{t:para-free}
  Let $\Box_A^{p,m}$ be the paradifferential gauge-covariant wave operator
  defined on line \eqref{para_wave}, and suppose that $\Box
  A^{free}=0$ with $\| A^{free}[0]\|_{\dot{H}^1\times L^2}\leq
  E$. If $m$ is sufficiently large, $m \gg_E 1$,  then we have the linear bound:
  \begin{equation}\label{para-free}
    \| \phi\|_{S^1} \ \lesssim_E \ \|\phi[0]\|_{E} +
 \| \Box_{A^{free}}^{p,m} \phi\|_{N \cap L^{2} \dot{H}^{-\frac{1}{2}} \cap L^{\frac{9}{5}} \dot{H}^{-\frac{4}{9}}}. 
  \end{equation}
\end{theorem}
Section ~\ref{s:para} is devoted to the proof of this result. 

Finally, we end this section with estimates that are relevant for high modulation bounds for $\phi$. As before, no null structure is necessary. 
\begin{proposition} \label{p:phi-high-mod}
Let $A, B, \psi$ be test functions defined on a time interval $I$. For all admissible frequency envelopes $c, d, e$, we have
\begin{align} 
	\nrm{\calM^{m,2}_{A} \psi}_{(L^{2} \dot{H}^{-\frac{1}{2}} \cap L^{\frac{9}{5}} \dot{H}^{-\frac{4}{9}})_{cd}[I]}
	\aleq &\nrm{(A_{x}, \nb A_{0})}_{(\DS \times L^{2} \dot{H}^{\frac{1}{2}})_{c}[I]} \nrm{\psi}_{\DS_{d}[I]} \label{phi-main-hm} \\
	\nrm{\calM^{m, 3}_{A, B} \psi}_{(L^{2} \dot{H}^{-\frac{1}{2}} \cap L^{\frac{9}{5}} \dot{H}^{-\frac{4}{9}})_{cde}[I]}
	\aleq &	\nrm{(A_{x}, \nb A_{0})}_{(\DS \times L^{2} \dot{H}^{\frac{1}{2}})_{c}[I]} 		\label{phi-cubic-hm}  \\
		& \times	\nrm{(B_{x}, \nb B_{0})}_{(\DS \times L^{2} \dot{H}^{\frac{1}{2}})_{d}[I]}
			\nrm{\psi}_{\DS_{e}[I]} \ .  \notag
\end{align}
For every $m > 0$, we also have the bound
\begin{equation} \label{phi-para-hm} 
	\nrm{\Diff^{m}_{A} \psi}_{(L^{2} \dot{H}^{-\frac{1}{2}} \cap L^{\frac{9}{5}} \dot{H}^{-\frac{4}{9}})_{cd}[I]}
	\aleq \nrm{(A_{x}, \nb A_{0})}_{(\DS \times L^{2} \dot{H}^{\frac{1}{2}})_{c}[I]} \nrm{\psi}_{\DS_{d}[I]} 
\end{equation}
with an implicit constant independent of $m$.

In addition, if $\psi$ is $\veps$-energy dispersed, then
\begin{align} 
	\nrm{\calM^{m, 2}_{A} \psi}_{(L^{2} \dot{H}^{-\frac{1}{2}} \cap L^{\frac{9}{5}} \dot{H}^{-\frac{4}{9}})_{c}[I]}
	\aleq & \veps^{\dlt_{1}} \nrm{(A_{x}, \nb A_{0})}_{(S^{1} \times L^{2} \dot{H}^{\frac{1}{2}})_{c}[I]} \nrm{\psi}_{S^{1}[I]} \label{phi-main-hm-eps} \\
	\nrm{\Diff_{A}^{m} \psi}_{(L^{2} \dot{H}^{-\frac{1}{2}} \cap L^{\frac{9}{5}} \dot{H}^{-\frac{4}{9}})_{c}[I]}
	\aleq& \veps^{\dlt_{1}} \nrm{(A_{x}, \nb A_{0})}_{(S^{1} \times L^{2} \dot{H}^{\frac{1}{2}})_{c}[I]} \nrm{\psi}_{S^{1}[I]}  \label{phi-para-hm-eps} \\
	\nrm{\calM^{m, 3}_{A, B} \psi}_{(L^{2} \dot{H}^{-\frac{1}{2}} \cap L^{\frac{9}{5}} \dot{H}^{-\frac{4}{9}})_{cd}[I]}
	\aleq &\veps^{\dlt_{1}} 	\nrm{(A_{x}, \nb A_{0})}_{(S^{1} \times L^{2} \dot{H}^{\frac{1}{2}})_{c}[I]}  \label{phi-cubic-hm-eps}   \\
	& \phantom{\veps^{\dlt_{1}}} \times
					\nrm{(B_{x}, \nb B_{0})}_{(S^{1} \times L^{2} \dot{H}^{\frac{1}{2}})_{d}[I]}
					\nrm{\psi}_{S^{1}[I]} 	\ .										\notag
\end{align}
\end{proposition}
\begin{proof} 
The proof is similar to Proposition~\ref{p:axnl-high-mod} and \ref{p:a0}. We sketch the case of the $L^{2} \dot{H}^{-\frac{1}{2}}$ norm and leave the case of the $L^{\frac{9}{5}} \dot{H}^{-\frac{4}{9}}$ norm (which is a simple variant) to the reader.

Compared to the frequency dyadic estimates in the proof of Proposition~\ref{p:axnl-high-mod}, it suffices to note that the following estimates hold:
\begin{gather*}
	\nrm{P_{k} (A_{k_{1}} \rd_{t} \psi_{k_{2}})}_{L^{2} \dot{H}^{-\frac{1}{2}}}
\!	+ \nrm{P_{k} (\rd_{t} A_{k_{1}} \psi_{k_{2}})}_{L^{2} \dot{H}^{-\frac{1}{2}}}
\!	\aleq  2^{-\dlt \max \set{\abs{k - k_{i}}}} \nrm{\nb A_{k_{1}}}_{L^{2} \dot{H}^{\frac{1}{2}}} \nrm{\abs{D}^{-\frac{2}{5}} \nb \psi_{k_{2}}}_{L^{\infty} L^{\frac{5}{2}}} \\
	\nrm{P_{k}(A_{k_{1}} B_{k_{2}} \psi_{k_{3}})}_{L^{2} \dot{H}^{-\frac{1}{2}}} 
	\aleq  2^{-\dlt \max \set{\abs{k - k_{i}}}} \nrm{A_{k_{1}}}_{L^{4} \dot{H}^{\frac{5}{4}}} 
										\nrm{B_{k_{2}}}_{L^{4} \dot{H}^{\frac{5}{4}}}
										\nrm{\abs{D}^{\frac{3}{5}} \psi_{k_{3}}}_{L^{\infty} L^{\frac{5}{2}}}
\end{gather*}
where we omitted $[I]$ and used the shorthands $A_{k} = P_{k} A_{0}$,
$B_{k} = P_{k} B_{0}$ and $\psi_{k} = P_{k} \psi$.  The off-diagonal
gain is again due to the freedom of choosing where to apply
Bernstein's inequality.  Moreover, by interpolation with the Sobolev
trace theorem, note that
\begin{equation*}
	\nrm{A_{k}}_{L^{4} \dot{H}^{\frac{5}{4}}} \aleq \nrm{\nb A_{k}}_{L^{2} \dot{H}^{\frac{1}{2}}} \, .
\end{equation*}
Since $(\infty, \frac{5}{2})$ is a non-sharp Strichartz exponent, the above estimates suffice for both divisibility and $\veps$-energy dispersed bounds (via Lemma~\ref{l:eps-ed-int}).
\end{proof}

\section{The structure of finite $S^{1}$ norm MKG waves.}
\label{s:large}

Here we consider an MKG solution $(A,\phi)$ on a time interval $I = [0,T]$, with finite 
$S^1$ norm for $(A_x,\phi)$.  Our main result is an accurate characterization 
of such maps:

\begin{theorem}\label{t:structure}
  Let $(A,\phi)$ be an admissible $C_{t} \calH^{1}$ solution to the MKG system
  \eqref{MKG} in the Coulomb gauge \eqref{Coulomb} on the time interval $I = [0, T]$ which has energy
  $E$ and $S^1$ norm $F$, i.e., $\nrm{(A_{x}, \phi)}_{S^{1}[I]} \leq F$. Let $c$ be a frequency envelope for the
  initial data $(A,\phi)[0]$ in the energy space $\dot H^1 \times
  L^2$. Then the following properties hold:

\begin{enumerate}

\item \label{item:structure:1} (Linear well-posedness for $\Box_A$) The linear equation
\[
\Box_A \psi = f, \qquad \psi[0] = \psi_0
\]
is well-posed, with bounds
\begin{equation} \label{boxA-lin}
 \|\psi\|_{S^1_d[I]} \lesssim_F  \|\psi[0]\|_{E_d} + \|f\|_{(N \cap L^{2} \dot{H}^{-\frac{1}{2}} \cap L^{\frac{5}{9}} \dot{H}^{-\frac{4}{9}})_d[I]} 
\end{equation}
for any admissible frequency envelope $d$.

\item \label{item:structure:2} (Frequency envelope bound) The solution $(A,\phi)$ satisfies
\begin{equation}
  \|(A_x, \phi)\|_{S^1_c[I]}   \lesssim_F 1 .
\end{equation}

\item \label{item:structure:3} (Refined Maxwell field bounds) We have
\begin{equation} \label{boxA}
    \| \Box A_i\|_{(N \cap L^{2} \dot{H}^{-\frac{1}{2}} \cap L^{\frac{9}{5}} \dot{H}^{-\frac{4}{9}})_{c^2}[I]}  \lesssim_F 1, \qquad 
\| \nabla A_0\|_{Y^{1}_{c^2}[I]} 
\lesssim_F 1 .
\end{equation}

\item \label{item:structure:4} (Refined scalar field bounds) 
We have 
\begin{equation}\label{boxphi-free}
  \| \Box_{A^{free}}  \phi\|_{(N \cap L^{2} \dot{H}^{-\frac{1}{2}} \cap L^{\frac{9}{5}} \dot{H}^{-\frac{4}{9}})_c [I]}  \lesssim_F  1, 
\end{equation}
and for each $m > 0 $ the following paradifferential estimates hold:
\begin{equation}\label{boxphi}
  \| \Box_A^{p,m}  \phi\|_{(N \cap L^{2} \dot{H}^{-\frac{1}{2}} \cap L^{\frac{9}{5}} \dot{H}^{-\frac{4}{9}})_{c^{2}}[I]}
  + \| \Box_{A^{free}}^{p,m}  \phi\|_{(N \cap L^{2} \dot{H}^{-\frac{1}{2}} \cap L^{\frac{9}{5}} \dot{H}^{-\frac{4}{9}})_c[I]}  \lesssim_F  2^{Cm}  \ .
\end{equation}

\item \label{item:structure:5} (Weak divisibility of $S^1$ norm) There exists a partition  $I = \bigcup_{k = 1}^K  I_k$
with $K \lesssim_F 1$ so that
\begin{equation}
  \|(A, \phi)\|_{S^1[I_k]}  \lesssim_{E} 1
\end{equation}
where the implicit constant is $C (E + E^{2})$ times the constant in Theorem~\ref{t:para-free}.

\end{enumerate}
\end{theorem}


\begin{proof}
As a preliminary step, we observe that  
from Lemma~\ref{l:ell}, \eqref{a0-start}, \eqref{a0} and \eqref{d0a0} we obtain
the bound
\begin{equation}
\label{dA0-noc}
\| \nabla A_0\|_{L^{2} \dot{H}^{\frac{1}{2}}_{c^2}[I]} 
\lesssim_F 1 .
\end{equation}
We remark that this bound will later be refined when we prove \ref{item:structure:3}.

\medskip 

\ref{item:structure:1}). We first prove well-posedness for 
the equation 
\[
\Box_A^{p,m} \psi = f, \qquad \psi[0] = (\psi_0,\psi_1)
\]
together with the bound 
\begin{equation}\label{para-nl}
\| \psi\|_{S^1_e} \lesssim \|\psi[0]\|_{E_e} + \|f\|_{(N \cap L^{2} \dot{H}^{-\frac{1}{2}} \cap L^{\frac{5}{9}} \dot{H}^{-\frac{4}{9}})_e}
\end{equation}
provided that $m \gg_F 1$. This is done perturbatively, based on the similar result
for $\Box_{A^{free}}^{p,m}$ in Theorem~\ref{t:para-free}.  
Using also \eqref{dA0-noc}, we can split time into $O_{F}(2^{100 Cm})$ intervals $I_n$ so that  
\[
\| \phi \|_{\DS[I_n]}+ \|A_x\|_{\DS[I_n]} + \|\nabla A_0\|_{L^2 \dot H^{\frac12} [I_n]} 
 \lesssim_F 2^{-2Cm}  \ .
\]
Then within each interval $I_n$ we write the equation above 
in the form 
\[
\Box_{A^{free}}^{p,m} \psi = 
 - \Diff_{\A}^{m} \psi  + f
= - (\Diff_{\A,large}^{m,2}  + \Diff_{\A}^{m,3} +  \Diff_{\A,small}^{m,2}) \psi  + f \ . 
\]
By Propositions~\ref{p:diff} and \ref{p:phi-high-mod}, all the terms on the right are
perturbative in $N \cap L^{2} \dot{H}^{-\frac{1}{2}} \cap L^{\frac{9}{5}} \dot{H}^{-\frac{4}{9}}[I_{n}]$, so if $m \gg_F 1$ then within each such interval we can
solve the above equation perturbatively. Reiterating, the global
solvability along with \eqref{para-nl} follows.  We note that in this
argument the free part of $A$ is reinitialized in each interval $I_k$.
The nonlinear part $A_x^{nl} = \A(\phi,\phi,A_x)$ is also defined
separately for each interval.

To get the well-posedness for the $\Box_A$ equation, we repeat the
above argument for the expression $\mathcal M_A^m$. For the $N$ norm, we apply Proposition~\ref{p:ma}. Then the small part is
treated perturbatively by taking $m \gg_{F} 1$, while for the large part we use again a time
interval division in order to gain smallness. For the $L^{2} \dot{H}^{-\frac{1}{2}} \cap L^{\frac{9}{5}} \dot{H}^{-\frac{4}{9}}$ norm, we use Proposition~\ref{p:phi-high-mod} and rely on divisibility for smallness.

\medskip

\ref{item:structure:2}).  The $\phi$ bound is a direct consequence of the bound
\eqref{boxA-lin} applied to $\phi$.  
Then we get the $A_x$ bound from \eqref{lin}, 
\eqref{axnl}-\eqref{axnl3} (for the $N$ norm) and \eqref{axnl-hm}-\eqref{axnl-hm3} (for the $L^{2} \dot{H}^{-\frac{1}{2}} \cap L^{\frac{9}{5}} \dot{H}^{-\frac{4}{9}}$ norm).
 
\medskip 

\ref{item:structure:3}). The $A_x$ bound has been proved in \ref{item:structure:2}, while the desired $A_0$ estimate follows from \eqref{a0-start} and Proposition~\ref{a0}.

\medskip

\ref{item:structure:4}). For the $N$ norm, the bound for $\Box_{A}^{p,m} \phi$ is a consequence of the
estimates \eqref{man}, \eqref{man-small} and \eqref{man-large} for the
components of $\Box_{A}^{p,m} \phi = - \mathcal M_A^m $.
For transition to $\Box_{A^{free}}^{p,m} \phi$ we use in addition 
the bounds \eqref{diffa} and \eqref{diffaa}. We can switch back from 
$\Box_{A^{free}}^{p,m} \phi$ to $\Box_{A^{free}} \phi$ using again 
the estimates \eqref{man}, \eqref{man-small} and \eqref{man-large} but for 
$A= A^{free}$. Finally, for the $L^{2} \dot{H}^{-\frac{1}{2}} \cap L^{\frac{9}{5}} \dot{H}^{-\frac{4}{9}}$ bound, we use Proposition~\ref{p:phi-high-mod} for all parts.

\medskip

\ref{item:structure:5}). 
By Proposition~\ref{p:energy} and conservation of energy, the linear energy $E_{lin}(A_{x}[t], \phi[t])$ is bounded by $E + E^{2}$ uniformly in time. Moreover, the $N$ norm is divisible by \eqref{n-div}, hence
the $A$ part is a direct consequence of \eqref{boxA}. The similar assertion 
for $\phi$ follows similarly from the divisibility of the $N$ norm and the second bound 
\eqref{boxphi}, since for a fixed $m \gg E$, the $\Box_{A^{free}}^{p,m}$ equation is well-posed in $S^1$
with implicit constants depending only on $E$. \qedhere



\end{proof}
 
With Theorem~\ref{t:structure} in hand, we may easily prove the continuation and scattering theorem (Theorem~\ref{thm:cont-scat}).
\begin{proof} [Proof of Theorem~\ref{thm:cont-scat}]
We start with the continuation result. The idea is to use the frequency envelope bound in Theorem~\ref{t:structure} to show a uniform lower bound on the energy concentration scale $r_{c}$ for all $t \in I$, which allows us to apply Theorem~\ref{t:local}.

By Theorem~\ref{t:structure}, we see that $(A, \phi)$ obeys the frequency envelope bound
\begin{equation*}
	\nrm{A_{0}}_{Y^{1}_{c^{2}}[I]} + \nrm{(A_{x}, \phi)}_{S^{1}_{c}[I]} \leq \tilde{F}
\end{equation*}
where $\nrm{c}_{\ell^{2}} \aleq \nrm{(A_{x}, \phi)[0]}_{\dot{H}^{1} \times L^{2}}$. In particular, $\lim_{\ell \to \infty} \nrm{c_{k}}_{\ell^{2}(k > \ell)} = 0$. 
Recall also that both $Y^{1}_{c}$ and $S^{1}_{c}$ control $\nrm{\nb (\cdot)}_{L^{\infty} L^{2}_{c}}$. Hence given any small number $\dlt > 0$, there exists $\ell \in \bbZ$ such that the splittings $A = A_{low} + A_{high} := A_{<\ell} + A_{\geq \ell}$ and $\phi = \phi_{low} + \phi_{high} := \phi_{< \ell} + \phi_{\geq \ell}$ obey
\begin{gather*}
	\nrm{\nb A_{\mu, high}(t)}_{L^{2}} + \nrm{\nb \phi_{high}(t)}_{L^{2}} < \frac{\dlt}{10},
\end{gather*}
and by Bernstein's inequality,
\begin{gather*}
	\nrm{A_{\mu, low}(t)}_{L^{\infty}}
	+ \nrm{\nb A_{\mu, low}(t)}_{L^{\infty}}
	+ \nrm{\phi_{low}(t)}_{L^{\infty}}
	+ \nrm{\nb \phi_{low}(t)}_{L^{\infty}} \aleq_{\tilde{F}, c} 1. 
\end{gather*}
Both bounds are uniform in $t \in I$. Using H\"older's inequality for the low frequency part, we can find $\tilde{r} = \tilde{r}(\tilde{F}, c, \dlt) > 0$ such that
\begin{equation*}
	\nrm{(A_{\mu}, \rd_{t} A_{\mu})(t)}_{(\dot{H}^{1} \cap L^{4}) \times L^{2}(B_{\tilde{r}}(x))} 
	+ \nrm{(\phi, \rd_{t} \phi(t)}_{(\dot{H}^{1} \cap L^{4}) \times L^{2}(B_{\tilde{r}}(x))} < \dlt
\end{equation*}
for every $t \in I$ and ball $B_{\tilde{r}}(x)$ of radius $\tilde{r}$ and arbitrary center $x \in \bbR^{4}$. Recalling the definition \eqref{eq:ecs-def}, we see that the energy concentrations scale of the data for $(A, \phi)$ at time $t$ is uniformly bounded below by $\tilde{r} > 0$, if $\dlt > 0$ is chosen sufficiently small depending only on $E$. Hence by Theorem~\ref{t:local}, $(A, \phi)$ can be continued past the endpoints of $I$ as an $C_{t} \calH^{1}$ admissible solution with appropriate $S^{1}$ and $Y^{1}$ bounds.

The scattering statement is an easy consequence of \eqref{boxA}, \eqref{boxphi-free}, and divisibility of the $N \cap L^{2} \dot{H}^{-\frac{1}{2}} \cap L^{\frac{9}{5}} \dot{H}^{-\frac{4}{9}}$ norm. \qedhere
\end{proof}

\subsection{ MKG waves with small energy dispersion.}

Here we continue the analysis above, but add to it the  small energy dispersion condition.
\begin{theorem} \label{t:structure-ed}
  Let $(A,\phi)$ be an admissible $C_{t} \calH^{1}$ solution to the MKG system
  \eqref{MKG} in the Coulomb gauge \eqref{Coulomb} on the time interval $I = [0, T]$, which has energy
  $E$ and $S^1$ norm $F$. Suppose furthermore that $\phi$ is $\veps$-energy dispersed.
  Then the following properties hold:

\begin{enumerate}

\item \label{item:structure-ed:1} (Elliptic bounds) We have
\begin{equation}\label{A0-eps}
 \| \nabla A_0\|_{Y^{1}[I]}  \lesssim_F  \veps^{\dlt_{1}} \nrm{\phi}_{S^{1}}^{2}
\end{equation}

\item \label{item:structure-ed:2} (High modulation bound) 
\begin{equation}\label{high-mod}
\begin{aligned}
\| \Box \phi \|_{L^2 \dot H^{-\frac12}[I]} + \|\Box \phi \|_{L^\frac95 \dot H^{-\frac49}[I]}
 & \lesssim_F   \veps^{\dlt_{1}} \nrm{\phi}_{S^{1}} 
\end{aligned}
\end{equation}

\item \label{item:structure-ed:3} (Maxwell field bounds) We have
\begin{equation}\label{BoxA-eps}
\| A^{nl}_x\|_{S^1[I]}  + \| \Box A_x\|_{(N \cap L^{2} \dot{H}^{-\frac{1}{2}} \cap L^{\frac{9}{5}} \dot{H}^{-\frac{4}{9}})[I]}  \lesssim_F \veps^{\dlt_{1}} \nrm{\phi}_{S^{1}}^{2}
\end{equation}

\item \label{item:structure-ed:4} (Scalar field bounds) For $m > 0 $ we have
\begin{equation}\label{Boxphi-eps}
  \| \Box_A^{p,m}  \phi\|_{(N \cap L^{2} \dot{H}^{-\frac{1}{2}} \cap L^{\frac{9}{5}} \dot{H}^{-\frac{4}{9}})[I]}  \lesssim_F  2^{Cm} \veps^{\dlt_{1}} \nrm{\phi}_{S^{1}}
\end{equation}

\end{enumerate}
\end{theorem}
\begin{proof}
\ref{item:structure-ed:1}). The bound \eqref{A0-eps} follows directly from the estimate \eqref{a0-disp}.

\medskip

\ref{item:structure-ed:2}).
These bounds follow from \eqref{phi-main-hm-eps}-\eqref{phi-cubic-hm-eps}.
%
%
%

\medskip

\ref{item:structure-ed:3}).
The estimate \eqref{BoxA-eps} follows from 
\eqref{axnl-eps}-\eqref{axnl-eps3}, \eqref{axnl-hm-eps}-\eqref{axnl-hm-eps3} for $\Box A_x^{nl}$ and \eqref{lin(i)}.


\medskip

\ref{item:structure-ed:4}).  The bound \eqref{Boxphi-eps} is a consequence of \eqref{man-disp}-\eqref{man-dispa}, \eqref{phi-main-hm-eps} and \eqref{phi-cubic-hm-eps}. \qedhere
 \end{proof}

\section{Induction on energy}
\label{s:induction}
Here we provide the induction on energy argument which gives the proof
of our main result in Theorem~\ref{t:ed}. Our induction hypothesis
is that the conclusion of the theorem holds up to energy $E$. Thus we
have $F(E)$ and $\epsilon(E)$.  
Our goal is to show that there exists $c_{0} = c_{0}(E) > 0$ so that the conclusion holds up to energy $E+c_{0}$. Moreover, we do not allow $c_{0}(E)$ to depend on $F(E)$, but only on $E$. The independence of $c_{0}(E)$ on $F(E)$ allows us to additionally ensure that $c_{0}(\cdot)$ is a positive non-increasing function on the whole $[0, \infty)$; this property is what makes our induction argument work for all energies\footnote{We refer to the beginning of Step~2.3 in the proof of Proposition~\ref{p:key-boot} for the precise dependence of $c_{0}$ on $E$. The conclusion is that $c_{0}$ needs to be chosen small enough compared to the constant in Theorem~\ref{t:para-free}.}.

To begin with, we observe that it suffices to establish Theorem~\ref{t:ed} for smooth solutions. Indeed, Theorem~\ref{t:local} implies that any admissible $C_{t} \calH^{1}$ solution $(A, \phi)$ can be approximated by smooth solutions in the $S^{1}[J]$ norm (and hence also in the $ED[J]$ norm) for any compact interval $J$.
Thus, we consider smooth data $(A_{x}[0],\phi[0])$ with energy $E+c_{0}$, generating 
a smooth solution $(A,\phi)$ in $[0,T)$ with $\nrm{\phi}_{ED(0, T)} \leq \epsilon \ll_E 1$.
Then the $S^1$ norm $\| (A_{x} ,\phi)\|_{S^1(0,t)}$ is a continuous function of time $t \in (0,T)$,
satisfying 
\[
\lim_{t \to 0} \| (A,\phi)\|_{S^1(0,t)} \lesssim E^{1/2}.
\]
Hence, in order to prove a uniform bound 
\begin{equation}\label{s-prove}
  \| (A_{x},\phi)\|_{S^1(0,t)} \leq F
\end{equation}
we can make the bootstrap assumption 
\begin{equation}\label{s-boot}
\| (A_{x},\phi)\|_{S^1(0,t)} \leq 2F,
\end{equation}
where $F$ is a positive to be determined in the proof. By scaling we harmlessly take $t = T$.

Indeed, once we show that \eqref{s-prove} holds assuming \eqref{s-boot}, a simple continuous induction argument in time implies that $S^{1}(0, T)$ norm of $(A_{x}, \phi)$ is bounded by $F$. This bound is precisely \eqref{S_est-ed} with $F(E + c_{0}) = F$. Note that the parameter $\eps$ becomes $\eps(E + c_{0})$ in Theorem~\ref{t:ed}.

Next, we dispense the easy case when the $S^{1}$ norm of $\phi$ is disproportionally small compared to the overall energy $E$ of $(A, \phi)$. This procedure allows us to link the small energy dispersion assumption $\nrm{\phi}_{ED(0, T)} \leq \eps$ to the notion of $\veps$-energy dispersion (Definition~\ref{def:eps-ed}) for some $\veps = \veps(\eps)$. More precisely, given $\veps > 0$ to be determined, we consider two cases: (i) $\nrm{\phi}_{S^{1}(0, T)} \leq \veps E^{1/2}$ or (ii) $\nrm{\phi}_{S^{1}(0, T)} > \veps E^{1/2}$. In case (i), a direct application of \eqref{axnl}-\eqref{axnl3}, \eqref{axnl-hm}-\eqref{axnl-hm3} and the bootstrap assumption \eqref{s-boot} gives
\begin{equation*}
	\nrm{\Box A_{x}}_{N \cap L^{2} \dot{H}^{-\frac{1}{2}} \cap L^{\frac{9}{5}} \dot{H}^{-\frac{4}{9}}(0, T)} 
	\aleq \veps^{2} E (1 + F).
\end{equation*}
Applying the linear estimate \eqref{lin} and taking $\veps$ sufficiently small compared to to $F$, \eqref{s-prove} follows directly. Thus we are left with case (ii), in which we may assume that $\phi$ is $\veps$-energy dispersed (according to Definition~\ref{def:eps-ed}) by taking $\eps= \veps^{2} E^{1/2}$. Henceforth we eliminate $\eps$ (which has the dimension of $(energy)^{1/2}$) in favor of the dimensionless parameter $\veps$. 

To establish the $S^{1}$ bound \eqref{s-prove} under the assumption that $\phi$ is $\veps$-energy dispersed on $[0, T]$, we will compare the solution $(A,\phi)$ with the MKG wave $(\tA,\tphi)$ 
generated by frequency truncated data
\[
(\tA_x[0],\tphi[0]) =  P_{\leq k*} (A_x[0],\phi[0])
\]
where the \emph{cut frequency} $k^* \in \bbR$ (hence $P_{\leq k*}$ is a continuous version of Littlewood-Paley projection) is selected so that $(\tA,\tphi)$ has
energy $E$.   Note here that we only truncate $\phi[0]$ and
$A_j[0]$. The functions $A_0$ and $\partial_t A_0$, which are also
part of the energy, are defined directly from the compatibility
conditions \eqref{MKGa0}.  The fact that such a $k^*$ exists is a consequence 
of the continuity with respect to $k$ of the $A_0$
component generated by $  P_{\leq k} (A_x[0],\phi[0])$, see Proposition~\ref{p:energy}. 
We further remark that by part \ref{item:energy:3} of Proposition~\ref{p:energy}, the energy 
of both $(A,\phi)$ and  $(\tA,\tphi)$ is $\epsilon^\frac14$ close to the 
corresponding linear energy of $(A_x[0],\phi[0])$, respectively $(\tA_x[0],\tphi[0])$.

We wish to apply the induction hypothesis to obtain an $S^{1}$ bound for $(\tA, \tphi)$ on $[0, T]$, namely
\begin{equation}
	\nrm{(\tA_{x}, \tphi)}_{S^{1}(0, T)} \leq F(E).
\end{equation} 
For this purpose, we need to know that energy dispersion of $\tphi$ is sufficiently small on $[0, T]$. We achieve this smallness by 
transferring the information for $\nrm{\phi}_{ED(0, T)}$ to $\nrm{\tphi}_{ED(0, T)}$ by another continuous induction in time.

Indeed, at time $t =0$, the solution $(\tA,\tphi)$ has smooth data and $\nrm{\tphi[0]}_{ED} \leq \veps^{2} E^{1/2} \ll \epsilon(E)$. 
Thus for some short time it will still have energy dispersion $\leq \epsilon(E)$. We claim
that $(\tA,\tphi)$ extends smoothly up to time $T$, so that the
stronger bound
\begin{equation}\label{ed-prove}
\| \tphi\|_{ED(0,t_0)} \leq  \frac12 \epsilon(E)
\end{equation}
holds for all $t_0 \in (0,T]$. We will establish \eqref{ed-prove} under the additional bootstrap assumption
\begin{equation}\label{ed-boot}
\| \tphi\|_{ED(0,t_0)} \leq   \epsilon(E)
\end{equation}
As before, note that we may take $t_{0} = T$ by scaling.

To see how the claim follows from this bootstrap procedure, let $T^*$ be the maximal time $T^* \leq T$ up to which
\eqref{ed-prove} holds.  Then by our induction hypothesis and Theorem~\ref{thm:cont-scat}, the
solution $(\tA,\tphi)$ extends smoothly past time $T^*$. Hence
\eqref{ed-boot} holds past the time $T^* $ by continuity, therefore \eqref{ed-prove} also
holds past time $T^*$ by our claim. This contradicts the maximality of $T^*$ unless
$T^* = T$.

To summarize, we have to prove that we can find $c_{0} = c_{0}(E)$, $F \gg_E 1$ and 
$\veps \ll_F 1$ so that the following statement holds:

\begin{proposition} \label{p:key-boot}
  Assume that the MKG waves $(A,\phi)$, respectively $(\tA,\tphi)$,
  with initial data $(A_x[0],\phi[0])$, respectively $
  (\tA_{x}[0],\tphi[0]) = P_{\leq k*} (\tA_x[0],\tphi[0])$, and energies
  $E+c_{0}$, respectively $E$, are smooth in $[0,T]$ and obey the following hypotheses:
  \begin{enumerate}[leftmargin=*]
\item[(i)] The $S^{1}$ norm of $(A_{x}, \phi)$ satisfies \eqref{s-boot}.
\item[(ii)] The solution $\phi$ is $\veps$-energy dispersed (as in Definition~\ref{def:eps-ed}).
\item[(iii)] The ED norms of $\phi$ and $\tphi$ obey
\begin{equation}
\| \tphi\|_{ED(0,T)} \leq   \epsilon(E), \qquad \| \phi\|_{ED(0,T)} \leq   \veps^{2} E^{1/2}.
\label{ed-hyp}
\end{equation}
\end{enumerate}
Then the following statements hold:
\begin{enumerate}
\item \label{item:key-boot:1} The $S^1$ norm of $ (A_{x},\phi)$ satisfies \eqref{s-prove}.
\item \label{item:key-boot:2} The ED norm of $\tphi$  satisfies \eqref{ed-prove}.
\end{enumerate}
\end{proposition}

\begin{proof} 
Unless otherwise stated, all norms below are taken over the time interval $(0, T)$. We will prove the proposition in two steps.

\pfstep{Step 1: The low frequency bound} Here we estimate the 
difference $(B,\psi)$ given by
\[
B = \tA- A_{<k^*}, \qquad \psi = \tphi - \phi_{<k^*}
\]
and prove that it satisfies the bound
\begin{equation}\label{bd-low}
\| ( B_x,\psi) \|_{S^1_{c^*}}+ \|\nabla B_0\|_{L^{2} \dot{H}^{\frac{1}{2}}_{c^*}}  \lesssim_F  \  \veps^{\dlt_{\ast}}, \qquad 
c^{*}(k) = 2^{- \dlt_{0} |k-k^*|}
\end{equation}
One consequence of the $\psi$ bound above, combined with 
\eqref{ed-hyp}, is that \eqref{ed-prove} holds. 

Before we begin, note that Theorem~\ref{t:structure-ed} implies the following a-priori bounds:
\begin{align} 
	\nrm{\Box \phi}_{L^{2} \dot{H}^{-\frac{1}{2}}} & \aleq_{F} \veps^{\dlt_{1}} \nrm{\phi}_{S^{1}}, \label{box-phi-key-boot} \\
	\nrm{\nb A_{0}}_{L^{2} \dot{H}^{\frac{1}{2}}} & \aleq_{F} \veps^{\dlt_{1}} \nrm{\phi}_{S^{1}}^{2}, \label{a0-key-boot}
\end{align}
In particular, the bound for $\phi$ ensures that \eqref{box-phi} holds, allowing us to apply Propositions~\ref{p:ax} and \ref{p:ma}, whereas the bound for $\nb A_{0}$ provides smallness in applications of \eqref{axnl-eps}.

\pfstep{Step 1.1: Bound for $B_{x}$}
To prove the estimates above, we begin with the bounds for $B_x$. By definition 
$B_x$ has zero Cauchy data at time $0$, therefore we have 
\[
\begin{split}
B_x = & \ \A_x(\tphi,\tphi,\tA) -  P_{<k^*} \A_x(\phi,\phi,A)
\\  = & \  
\A_x(\tphi,\tphi,\tA) -  \A_x(\phi_{<k^*},\phi_{<k^*},A_{<k^*}) 
+\A_x(\phi_{<k^*},\phi_{<k^*},A_{<k^*})   -  P_{<k^*} \A_x(\phi,\phi,A)  
\end{split}
\]
In the first difference above we substitute $\tA= A_{<k^*}+B$ and $\tphi = \phi_{<k^*}+\psi$, then use \eqref{lin}, \eqref{axnl}-\eqref{axnl3}, \eqref{axnl-eps}-\eqref{axnl-eps3}, \eqref{axnl-hm}-\eqref{axnl-hm3} and \eqref{axnl-hm-eps}-\eqref{axnl-hm-eps3} (see also Remark~\ref{rem:disp-freq-loc}) to obtain
\[
\| \A_x(\tphi,\tphi,\tA) -  \A_x(\phi_{<k^*},\phi_{<k^*},A_{<k^*}) \|_{S^1_{c^*}}
\lesssim_F \veps^{\dlt_{1}} \|(B_x,\psi)\|_{S^1_{c^*}} +   \|(B_x,\psi)\|_{S^1_{c^*}}^2+
 \|(B_x,\psi)\|_{S^1_{c^*}}^3
\]
The second difference is localized at frequency $< 2^{k^*+5}$, and all the $\phi$
factors are $\veps$-energy dispersed. Further, we may rewrite this difference as
\begin{align*}
& P_{<k^* + 5} \bb( \A_x(\phi_{<k^*},\phi_{<k^*},A_{<k^*}) - \A_x(\phi_{<k^*-5},\phi_{<k^*-5},A_{<k^*-5}) \bb) \\
&  - P_{<k^*} \bb( \A_x(\phi,\phi,A)  - \A_{x}(\phi_{<k^*-5} , \phi_{<k^*-5}, A_{<k^*-5}) \bb)
\end{align*}
which shows that of the two or three inputs, at least one has frequency $\geq 2^{k^* - 10}$.
This input can be measured with the frequency envelope $c^*$ at frequencies below $2^{k^* - 10}$.  Thus, applying
\eqref{lin}, \eqref{axnl-eps}-\eqref{axnl-eps3} and \eqref{axnl-hm-eps}-\eqref{axnl-hm-eps3} yields
\[
\| \A_x(\phi_{<k^*},\phi_{<k^*},A_{<k^*})   -  P_{<k^*} \A_x(\phi,\phi,A)  \|_{S^1_{c^*}}
\lesssim_F \veps^{\dlt_{1}} \ .
\]
Summing up the last two bounds, we get
\begin{equation}\label{bx}
\|B_x\|_{S^1_{c^*}}
\lesssim_F (\veps^{\dlt_{1}}+   \|(B_x,\psi)\|_{S^1_{c^*}}^2) (1+ \|(B_x,\psi)\|_{S^1_{c^*}}) 
\end{equation}

\pfstep{Step 1.2: Bound for $B_{0}$}
The analysis for $B_0$ is very similar. Precisely,
$B_0$ solves the equation
\[
\begin{split}
B_0 = & \ \A_0(\tphi,\tphi,\tA) -  P_{<k^*} \A_0(\phi,\phi,A)
\\  = & \  
\A_0(\tphi,\tphi,\tA) -  \A_0(\phi_{<k^*},\phi_{<k^*},A_{<k^*}) 
+\A_0(\phi_{<k^*},\phi_{<k^*},A_{<k^*})   -  P_{<k^*} \A_0(\phi,\phi,A)  
\end{split}
\]
and the terms on the right can be estimated using \eqref{a0}-\eqref{a0-disp}.
The same applies for $\partial_t B_0$. We obtain 
\begin{equation}\label{b0}
\|\nabla B_0\|_{L^2 \dot H^\frac12_{c^*}}
\lesssim_F( \veps^{\dlt_{1}} + \|(B_x,\psi)\|_{S^1_{c^*}}^2)(1+ \|(B_x,\psi)\|_{S^1_{c^*}}
+ \|\nabla B_0\|_{L^2   \dot H^\frac12_{c^*}} ) 
\end{equation}

\pfstep{Step 1.3: Bound for $\psi$}
We now consider $\psi$, which solves
\begin{equation} \label{psi-eq}
\begin{split}
\Box_{\tA} \psi =  &  - (\Box_{\tA}- \Box_{A_{<k^*}}) \phi_{<k^*}
- (\Box_{A_{<k^*}}  \phi_{<k^*} - P_{<k^*} \Box_{A} \phi) 
\end{split}
\end{equation}
We start by estimating the right hand side in $(L^{2} \dot{H}^{-\frac{1}{2}} \cap L^{\frac{5}{9}} \dot{H}^{-\frac{4}{9}})_{c^*}$. For the first difference, we write $\tA = A_{<k^{*}} + B$ and observe that at least one input is $B$ (which can be measured using $c^{*}$) and $\phi_{<k^{*}}$ is $\veps$-energy dispersed. Hence by \eqref{phi-main-hm-eps}-\eqref{phi-cubic-hm-eps}, we have
\begin{equation} \label{psi-hm-1}
\begin{aligned}
	& \nrm{(\Box_{\tA}- \Box_{A_{<k^*}}) \phi_{<k^*}}_{(L^{2} \dot{H}^{-\frac{1}{2}} \cap L^{\frac{5}{9}} \dot{H}^{-\frac{4}{9}})_{c^{*}}} \\
	& \quad \aleq_{F} \veps^{\dlt_{1}} (\nrm{B_{x}}_{S^{1}_{c^*}} + \nrm{\nb B_{0}}_{L^{2} \dot{H}^{\frac{1}{2}}_{c^*}}) 
					(1 + \nrm{B_{x}}_{S^{1}_{c^*}} + \nrm{\nb B_{0}}_{L^{2} \dot{H}^{\frac{1}{2}}_{c^*}}).
\end{aligned}
\end{equation}
For the second difference in \eqref{psi-eq}, we claim that the following bound holds:
\begin{equation} \label{psi-hm-2}
\begin{aligned}
	& \nrm{(\Box_{A_{<k^*}}  \phi_{<k^*} - P_{<k^*} \Box_{A} \phi)}_{(L^{2} \dot{H}^{-\frac{1}{2}} \cap L^{\frac{5}{9}} \dot{H}^{-\frac{4}{9}})_{c^*}}  \\
	& \quad \aleq_{F} \veps^{\dlt_{1}} (\nrm{B_{x}}_{S^{1}_{c^*}} + \nrm{\nb B_{0}}_{L^{2} \dot{H}^{\frac{1}{2}}_{c^*}}) 
					(1 + \nrm{B_{x}}_{S^{1}_{c^*}} + \nrm{\nb B_{0}}_{L^{2} \dot{H}^{\frac{1}{2}}_{c^*}}).
\end{aligned}
\end{equation}
To prove this bound, we divide further into the following cases:
\begin{itemize}
\item [(i)] At least one of the $A$ frequencies is $> 2^{k^{*} - 10}$. Note that the output frequency is localized to $< 2^{k^* + 5}$. Hence by measuring the high frequency input with $c^{*}$, using the $\veps$-energy dispersion of $\phi$ and applying \eqref{phi-main-hm-eps}-\eqref{phi-cubic-hm-eps}, we can bound this contribution in $(L^{2} \dot{H}^{-\frac{1}{2}} \cap L^{\frac{5}{9}} \dot{H}^{-\frac{4}{9}})_{c^{*}}$ by $\aleq_{F} \veps^{\dlt_{1}}$.
\item [(ii)] The term $\Box_{A_{<k^{*}-10}} \phi_{<k^{*}} - P_{<k^{*}} (\Box_{A_{<k^{*} -10}} \phi)$. The contribution of $\phi_{<k^{*} - 5}$ and $\phi_{>k^{*}+5}$ is zero, so we may assume that $\phi$ and the output are frequency localized near $2^{k^{*}}$. Then by $\veps$-energy dispersion of $\phi$ and \eqref{phi-main-hm-eps}-\eqref{phi-cubic-hm-eps}, the desired estimate follows.
\end{itemize}

To estimate the $N_{c^*}$ norm of the right hand side in \eqref{psi-eq}, we use a frequency gap
parameter $m$ to be chosen later.  The first difference in
\eqref{psi-eq} is expressed in the form
\[
(\Box_{\tA}- \Box_{A_{<k^*}}) \phi_{<k^*} = (\mathcal M_\tA^m -  \mathcal M_{A_{<k^*}}^m)
\phi_{<k^*} 
+  (\Diff_{\A}^m (\tphi,\tphi,\tA) -  \Diff_{\A}^m (\phi,\phi,A)) \phi_{<k^*}
\]
In the first term, note that one of the inputs must be $B$. Then we use the frequency envelope $c^*$ for $B$, the $\veps$-energy dispersion of $\phi_{<k^{\ast}}$, \eqref{box-phi-key-boot} and \eqref{a0-key-boot} via \eqref{man-disp}-\eqref{man-dispa} to obtain
\begin{equation} \label{psilow-1}
\| (\mathcal M_\tA^m -  \mathcal M_{A_{<k^*}}^m)
\phi_{<k^*} \|_{N_{c^*}} \lesssim_F   (\veps^{\dlt_{1}} 2^{Cm} \nrm{B_{x}}_{S^{1}_{c^*}} + \nrm{\nb B_{0}}_{L^{2} \dot{H}^{\frac{1}{2}}_{c^*}}) 
					(1 + \nrm{B_{x}}_{S^{1}_{c^*}} + \nrm{\nb B_{0}}_{L^{2} \dot{H}^{\frac{1}{2}}_{c^*}})
\end{equation}
In the second term, we first replace the argument $(\phi,\phi,A)$ by $(\phi_{<k^*}, \phi_{<k^*},
A_{<k^*})$, and estimate the corresponding difference via \eqref{diffa}-\eqref{diffaa} as
\begin{equation} \label{psilow-2}
\|(\Diff_{\A}^m (\phi,\phi,A) - \Diff_{\A}^m(\phi_{<k^*}, \phi_{<k^*},
A_{<k^*})) \phi_{<k^*} \|_{N_{c^*}}  \lesssim_F  2^{-\dlt_{0} m}
\end{equation}
where both the frequency envelope control of $N_{c^*}$ and the gain $2^{-\dlt_{0} m}$ come from the frequency gap between the difference of
the magnetic coefficients $\A(\phi,\phi,A) - \A(\phi_{<k^*},
\phi_{<k^*}, A_{<k*})$ (which is only used at frequencies below
$2^{k^*-m}$) and its arguments $(\phi,\phi,A)$ (of which at least one must
have frequency no smaller than $2^{k^*}$; we use $c^*$ to measure this input). Then we are left to establish
\begin{equation} \label{psilow-3}
\begin{aligned}
& \|  (\Diff_{\A}^m (\tphi,\tphi,\tA) - \Diff_{\A}^m(\phi_{<k^*}, \phi_{<k^*}, A_{<k*})) \phi_{<k^*} \| _{N_{c^*}}   \\
& \quad \lesssim_F 2^{-\dlt_{0} m} (\|(B_{x},\psi)\|_{S^1_{c^*}} + \nrm{\nb B_{0}}_{L^{2} \dot{H}^{\frac{1}{2}}_{c^*}})(1 + \|\psi\|_{S^1_{c^*}}^2)
\end{aligned}
\end{equation}
Note that one of the inputs must be $(B, \psi)$. Then \eqref{psilow-3} follows again from \eqref{diffa}-\eqref{diffaa}, and using the frequency envelope $c^*$ to measure $(B, \psi)$.

Finally, we still have the second difference in \eqref{psi-eq} in $N_{c^{\ast}}$, for which we claim that
\begin{equation} \label{psilow-4} 
\| \Box_{A_{<k^*}} \phi_{<k^*} -
  P_{<k^*} \Box_{A} \phi\|_{N_{c^*}} \lesssim_F \veps^{\dlt_{1}} 2^{Cm}  
  + 2^{-c m}
\end{equation}
 To see this we write it as 
\[
(\Box_{A_{<k^*}}  \phi_{<k^*} - P_{<k^*} \Box_{A} \phi)  = 
(\mathcal M_{A_{<k^*}}^m  \phi_{<k^*} - P_{<k^*} \mathcal M^m_{A} \phi) +   [P_{<k^*}, \Diff^m_A] \phi 
\]
For both differences, note that the output frequency is localized to $< 2^{k^{*} + 5}$. 
Canceling the like terms in the first difference, we are left with three types
of frequency scenarios:

\begin{itemize}
\item[(i)] The frequency of one of the $A$'s is at least $k^*-m$.
Then we can apply \eqref{man-disp}-\eqref{man-dispa} and \eqref{a0-key-boot} to obtain
an $ \veps^{\dlt_{1}} 2^{Cm}$ bound, where we use $c^{*}$ to measure the high frequency input. 

\item[(ii)] We have a quadratic term of the form $\partial_t A_{0,<k^*-m} \phi_{[k^*-5, k^*+5]}$,
which can be directly estimated by $2^{-c m}$ using Strichartz bounds and Bernstein's
inequality.

\item[(iii)] We have a cubic term 
of the form $A_{<k^*-m}^2 \phi_{[k^*-5, k^*+5]}$, which in turn can be estimated directly in $L^1 L^2$
using non-sharp Strichartz estimates, to get an $\veps^{\dlt_{1}}$ bound.
\end{itemize} 

It remains to consider the commutator term  $[ P_{<k^*} , \Diff^m_A] \phi $; 
we claim that the contribution of this term can be estimated by $2^{-c m}$.
It is clear that only the frequencies close to $k^*$ in $\phi$ are relevant here,
therefore the commutator can be  expressed as
\[
[ P_{<k^*}, A^\alpha_{<k^*-m}] \partial_\alpha \phi_{k^*} =       
2^{-k^*} L(\nabla  A^\alpha_{<k^*-m},  \partial_\alpha \phi_{k^*}  )
\]
for a bilinear form $L$ with translation invariant integrable kernel.
The $A_0$ term is easy to deal with using the $L^2 \dot H^\frac32$
bound for $A_0$, Strichartz for $\phi$ and Bernstein's inequality.
Thus we are left with the expression $2^{-k^*} L(\nabla
A^j_{<k^*-m}, \partial_j \phi_{k^*} )$, which has both a null
structure and a favorable frequency balance.
This we can treat using the bound \eqref{null} in the beginning of the next section
(see also \eqref{Aphi-null}). Hence \eqref{psilow-4} follows.

Thus, summing up all cases in \eqref{psi-hm-1}-\eqref{psilow-4}, we obtain
\begin{align*}
& \| \Box_{\tA} \psi\|_{(N \cap L^{2} \dot{H}^{-\frac{1}{2}} \cap L^{\frac{9}{5}} \dot{H}^{-\frac{4}{9}})_{c^*}} \lesssim_F \\
&  \bb(2^{-cm} + 2^{-\dlt_{0} m }+ \veps^{\dlt_{1}} 2^{Cm}(1 +  \|(B_x,\psi)\|_{S^1_{c^*}}) + \nrm{\nb B_{0}}_{L^{2} \dot{H}^{\frac{1}{2}}_{c^*}}  \bb)
(1+\|(B_x,\psi)\|_{S^1_{c^*}}^2 + \nrm{\nb B_{0}}_{L^{2} \dot{H}^{\frac{1}{2}}_{c^*}}^{2}) 
\end{align*}
Optimizing the choice of $m$, this gives (with $\delta_{\ast} \ll \delta_{0} \ll c $) 
\begin{align*}
& \| \Box_{\tA} \psi\|_{(N \cap L^{2} \dot{H}^{-\frac{1}{2}} \cap L^{\frac{9}{5}} \dot{H}^{-\frac{4}{9}})_{c^*}}  \\
& \quad \lesssim_F \bb( \veps^{\dlt_{\ast}}(1 +  \|(B_x,\psi)\|_{S^1_{c^*}}) + \nrm{\nb B_{0}}_{L^{2} \dot{H}^{\frac{1}{2}}_{c^*}} \bb)
(1+\|(B_x,\psi)\|_{S^1_{c^*}}^2 + \nrm{\nb B_{0}}_{L^{2} \dot{H}^{\frac{1}{2}}_{c^*}}^{2}) 
\end{align*}
Recalling that $\psi$ has zero initial data, by Theorem~\ref{t:structure}\ref{item:structure:1} this 
implies the estimate
\begin{equation}\label{psi}
\| \psi\|_{S^1_{c^*}} \lesssim_{F}
 \bb( \veps^{\dlt_{\ast}}(1 +  \|(B_x,\psi)\|_{S^1_{c^*}}) + \nrm{\nb B_{0}}_{L^{2} \dot{H}^{\frac{1}{2}}_{c^*}} \bb)
(1+\|(B_x,\psi)\|_{S^1_{c^*}}^2 + \nrm{\nb B_{0}}_{L^{2} \dot{H}^{\frac{1}{2}}_{c^*}}^{2}) .
\end{equation}
Now we can combine this with \eqref{bx} and \eqref{b0}, and 
close to prove \eqref{bd-low}, provided that $\veps$ is small enough. 
We carefully observe here that the smallness of $\veps$ depends on $F$.
In turn, we will want later that the choice of $F$ is independent of $\veps$. 
\bigskip 

\pfstep{Step 2: The high frequency bound} Here we consider the high frequency differences
$(B^{high}, \psi^{high})$ given by
\[
B^{high} = A - \tA, \qquad \psi^{high} = \phi - \tphi
\]
and prove that they satisfy the $S^1$ bound
\begin{equation} \label{high(I)}
\begin{aligned}
\| (B_x^{high},   \psi^{high}) \|_{S^1}+ \| \nabla B_0\|_{L^2 \dot H^\frac12} \lesssim_{F(E)} 1,
\end{aligned}
\end{equation}
provided that $c_{0} = c_{0}(E)$ is chosen small enough compared to $E$, but independent of $F(E)$. 

\pfstep{Step 2.1: Energy estimate for $(B^{high}_{x}, \psi^{high})$ and weak divisibility}
Here we take the necessary steps to ensure the independence of $c_{0}$ on $F(E)$. 
We first use the energy conservation for $(A,\phi)$ and $(\tA,\tphi)$, together with the low frequency  estimates of the previous step,
to conclude that the energy norm for $(B^{high}, \psi^{high})$ stays bounded, i.e.,
\begin{equation}\label{en-high}
\| (\nabla B_x^{high}, \nabla \psi^{high})\|_{L^\infty L^2}^2 \lesssim c_{0} + O_F(\veps^{\frac{1}{4} \dlt_{\ast}})
\end{equation}
Indeed, for each $t \in (0, T)$ we have
\begin{align*}
	E_{lin}(A, \phi)
	= & E_{lin}(A_{\geq k^*}, \phi_{\geq k^*}) + E_{lin}(A_{< k^*}, \phi_{< k^*}) 
	+ \brk{A_{\geq k^*}, A_{< k^*}}_{\dot{H}^{1} \times L^{2}} + \brk{\phi_{\geq k^*}, \phi_{< k^*}}_{\dot{H}^{1} \times L^{2}} \\
	\geq & E_{lin}(A_{\geq k^*}, \phi_{\geq k^*}) + E_{lin}(A_{< k^*}, \phi_{< k^*}).
\end{align*}
where we omitted writing $[t]$ and the subscript $x$ from $A = A_{x}$. We have used the fact that the operator $P_{< k^*} P_{\geq k^*}$ is non-negative, as it has a non-negative symbol. By Step 1, we know that $(A_{<k^*}, \phi_{<k^*})$ is equal to $(\tA, \tphi)$ up to an error of size $O_{F}(\veps^{\dlt_{\ast}})$ in $S^{1}_{c^*}$. Therefore, we have
\begin{align*}
	E_{lin}(B^{high}, \psi^{high})
	= & E_{lin}( A_{\geq k^*}, \phi_{\geq k^*}) + O_{F}(\veps^{\dlt_{\ast}}) \\
	\leq & E_{lin}(A, \phi) - E_{lin}(A_{<k^*}, \phi_{<k^*}) + O_{F}(\veps^{\dlt_{\ast}}) \\
	= & E_{lin}(A, \phi) - E_{lin}(\tA, \tphi) + O_{F}(\veps^{\dlt_{\ast}}).
\end{align*}
By Proposition~\ref{p:energy}, recall that $E_{lin}(A, \phi)$ and $E_{lin}(\tA, \tphi)$ are $O_{E}(\veps^{\frac{1}{4} \dlt_{\ast}})$ close to the corresponding conserved energies $\calE[A, \phi]$ and $\calE[\tA, \tphi]$, respectively. Hence by the definition of $(\tA, \tphi)$, the desired estimate \eqref{en-high} follows.

Next, we use the weak divisibility of the $S^1$ norm in
Theorem~\ref{t:structure}\ref{item:structure:5} to split the time interval $I$ into
$O_{F(E)} (1)$ subintervals, on each of which
 \begin{equation}\label{I-split}
\| (\tA_{x}, \tphi)\|_{S^1[J]} \lesssim_E 1.
\end{equation}
We remark that this bound also relies on the conservation of energy for $(\tA, \tphi)$.
Due to the uniform bound in \eqref{en-high}, on each such subinterval
$J$ we can reinitialize the data for $(B^{high}_{x}, \psi^{high})$ and we
no longer have any trace of $F(E)$ or $\eps(E)$.  Instead, $(\tA_{x},
\tphi)$ has $S^1$ norm $O_E(1)$ and energy dispersion
$\tilde{C}_{F} \veps^{\dlt_{\ast}}$.  Thus, it remains to show that on each $J$ we
have the improved bound
\begin{equation}\label{high(J)}
\begin{aligned}
\| (B^{high},\psi^{high}) \|_{S^1[J]} \lesssim_{E} 1
\end{aligned}
\end{equation}
Then \eqref{high(I)} would follow by adding the above over $O_{F(E)}(1)$
intervals, using \eqref{s-sum}.  

\begin{remark} \label{rem:disp-psih-tphi}
Compared to the low frequency estimate in Step 1,
here we have a key advantage that we can exploit small energy dispersion for both $\tphi$ and $\psi^{high}$, albeit at the expense of using the $S^{1}$ norm of $\phi$ on the larger interval $(0, T)$. More precisely, Step 1 implies
\begin{equation} \label{psih-phi}
	\nrm{\tphi - \phi_{< k^*}}_{S^{1}(0, T)} + \nrm{\psi^{high} - \phi_{\geq k^*}}_{S^{1}(0, T)} \leq \tilde{C}_{F} \veps^{\dlt_{\ast}}.
\end{equation}
for some constant $\tilde{C}_{F} > 1$.
Moreover, $\phi$ is $\veps$-energy dispersed by hypothesis and obeys \eqref{box-phi-key-boot} on the large interval $(0, T)$. By Remark~\ref{rem:disp-freq-loc} and a simple extension procedure\footnote{Technically, one extends all non energy-dispersed inputs of the form $\phi, A_{x}$ by homogeneous waves outside $J$ to $(0, T)$ (see Proposition~\ref{p:intervals}) and $A_{0}$ by a standard Sobolev extension so that $\nrm{\nb A_{0}}_{L^{2} \dot{H}^{\frac{1}{2}}(0, T)} \aleq \nrm{\nb A_{0}}_{L^{2} \dot{H}^{\frac{1}{2}}[J]}$.}, we may gain $\tilde{C}_{F} \veps^{\dlt_{\ast}} + \veps^{\dlt_{1}} \nrm{\phi}_{S^{1}(0, T)}$ from $\tphi$ or $\psi^{high}$ whenever any of the $\veps$-energy dispersion bounds \eqref{axnl-eps}-\eqref{axnl-eps3}, \eqref{axnl-hm-eps}-\eqref{axnl-hm-eps3}, \eqref{a0-disp}, \eqref{man-disp}-\eqref{man-dispa}, \eqref{phi-main-hm-eps}-\eqref{phi-cubic-hm-eps} are applicable on the smaller interval $J$.
\end{remark}


\pfstep{Step 2.2: Bound for $B^{high}_{x}$}
The bound for $B^{high}_{x}$ is an easy consequence of small energy dispersion. Indeed, using \eqref{psih-phi}, $\veps$-energy dispersion of $\phi$, \eqref{box-phi-key-boot} as well as the estimates \eqref{axnl}-\eqref{axnl3}, \eqref{axnl-eps}-\eqref{axnl-eps3}, \eqref{axnl-hm}-\eqref{axnl-hm3} and \eqref{axnl-hm-eps}-\eqref{axnl-hm-eps3} (see also Remark~\ref{rem:disp-psih-tphi}), we have
\begin{equation} \label{boxB-high}
\| \Box B^{high}_{x}\|_{N \cap L^{2} \dot{H}^{-\frac{1}{2}} \cap L^{\frac{9}{5}} \dot{H}^{-\frac{4}{9}} [J]} \lesssim_F \veps^{\dlt_{\ast}},
\end{equation}
where we used the fact that $\dlt_{\ast} \ll \dlt$. Then by the linear estimate \eqref{lin} and \eqref{en-high}, it follows that
\begin{equation} \label{B-high}
\| B^{high}\|_{S^{1}[J]} \lesssim_E c_{0} + O_{F}(\veps^{\frac{1}{4}\dlt_{\ast}}) \ .
\end{equation}
This bound is stronger than what we need for \eqref{high(J)}, but it will be useful in the next step. 

\pfstep{Step 2.3: Bound for $\psi^{high}$}
For $\psi^{high}$, we claim that a similar bound to \eqref{boxB-high} but with respect to the $\Box_A$ flow holds:
\begin{equation}\label{boxA-psih}
\|\Box_A \psi^{high}\|_{N \cap L^{2} \dot{H}^{-\frac{1}{2}} \cap L^{\frac{9}{5}} \dot{H}^{-\frac{4}{9}}[J]} \lesssim_F \veps^{\dlt_{\ast \ast}}.
\end{equation}
where $0 < \dlt_{\ast \ast} \ll \dlt_{\ast}$.
Assuming that \eqref{boxA-psih} holds, we can conclude \eqref{high(J)} using the following simple additional bootstrap argument in time. Denoting the initial time in $J$ by $t_{0}$, it follows from \eqref{en-high} and continuity of the $S^{1}[J']$ norm that we have
\begin{equation} \label{phih-prove}
	\nrm{\psi^{high}}_{S^{1}[J']} \leq C_{E} (c_{0} + O_{F}(\veps^{\frac{1}{4} \dlt_{\ast}}))
\end{equation}
for a suitably large constant $C_{E} > 1$ and a sufficiently short interval $J'$ containing $t_{0}$. Then to prove \eqref{phih-prove} for $J' = J$, it suffices to establish \eqref{phih-prove} under the bootstrap assumption
\begin{equation} \label{phih-boot}
	\nrm{\psi^{high}}_{S^{1}[J']} \leq 2 C_{E} (c_{0} + O_{F}(\veps^{\frac{1}{4} \dlt_{\ast}})).
\end{equation}
Choosing $c_{0}$ sufficiently small depending on $E$ and $\veps \ll_{F} 1$, it follows from \eqref{en-high}, \eqref{I-split}, \eqref{B-high} and \eqref{phih-boot} that
\begin{equation*}
	\nrm{(A_{x}, \phi)}_{S^{1}[J']} \aleq_{E} 1,
\end{equation*}
where the implicit constant is twice that of \eqref{I-split}. Now applying the linear $\Box_A$ bound \eqref{boxA-lin} in Theorem~\ref{t:structure}\ref{item:structure:1} and enlarging $C_{E}$ to be larger than the implicit constant in \eqref{boxA-lin}, the desired estimate \eqref{phih-prove} follows.
We remark that the size of $c_{0}$ essentially depends on the implicit constant in \eqref{boxA-lin}, which in turn depends on the constant in Theorem~\ref{t:para-free}.

We now turn to the proof of \eqref{boxA-psih}. We first estimate the $L^{2} \dot{H}^{-\frac{1}{2}} \cap L^{\frac{9}{5}} \dot{H}^{-\frac{4}{9}}$ norm, which is easier. Note that
\begin{equation*}
	\Box_{A} \psi^{high} = (\Box_{\tA} - \Box_{A}) \tphi.
\end{equation*}
Exploiting the small energy dispersion of $\tphi$ as in Remark~\ref{rem:disp-psih-tphi} and applying \eqref{phi-main-hm}-\eqref{phi-para-hm}, \eqref{phi-main-hm-eps}-\eqref{phi-cubic-hm-eps}, as well as \eqref{bd-low}, \eqref{a0-key-boot} for $A_{0}$, $\tA_{0}$, we obtain
\begin{equation} \label{boxA-psih-hm}
	\nrm{\Box_{A} \psi^{high}}_{L^{2} \dot{H}^{-\frac{1}{2}} \cap L^{\frac{9}{5}} \dot{H}^{-\frac{4}{9}}[J]}
	\aleq_{F} \veps^{\dlt_{\ast}}.
\end{equation}

To bound the $N$ norm in \eqref{boxA-psih}, we introduce a frequency gap $m$ to be chosen later. Then we
 write the $\Box_A$ equation for $\psi^{high}$
as follows:
\begin{align}
\Box_A \psi^{high} =& \   (\Box_{\tA} - \Box_A) \tphi 			\label{boxA-psih-decomp}
\\ = & (\mathcal M_\tA^m - \mathcal M_A^m) \tphi+   
(\Box^{p, m}_{\tA^{free}} - \Box^{p, m}_{A^{free}}) \tphi  + (\Diff_\A^m(\tphi,\tphi,\tA) - \Diff_\A^m(\phi,\phi,A))\tphi, \notag
\end{align}
where the decomposition $A = A^{free} + A^{nl}$ is performed with respect to some fixed initial time $t_{0} \in J$.

In the first term on the right hand side of \eqref{boxA-psih-decomp}, we take advantage of the small energy dispersion of $\tphi$
as in Remark~\ref{rem:disp-psih-tphi} and apply \eqref{man}-\eqref{man-small}, \eqref{man-disp}-\eqref{man-dispa}, as well as \eqref{bd-low}, \eqref{a0-key-boot} for $A_{0}$, $\tA_{0}$ (here it is crucial to use the smallness factor $\veps^{\dlt_{1}}$ in \eqref{a0-key-boot}), to obtain
\[
\|  (\mathcal M_\tA^m - \mathcal M_A^m) \tphi\|_{N[J]} \lesssim_F 2^{Cm} \veps^{\dlt_{\ast}}.
\]
In the term $(\Box^{p, m}_{\tA^{free}} - \Box^{p, m}_{A^{free}}) \tphi$, we make a further decomposition as follows:
\begin{align*}
(\Box^{p, m}_{\tA^{free}} - \Box^{p, m}_{A^{free}}) \tphi
= (\calM^{m, 2}_{\tA^{free}} - \calM^{m, 2}_{A^{free}}) \tphi + 2 i \sum_{k > k^{\ast} + m} (\tA^{free} - A^{free})_{< k-m}^{j} \rd_{j} \tphi_{k}
\end{align*}
For the first difference, we use the bounds \eqref{man-small}-\eqref{man-large} and \eqref{man-disp}, where we exploit the small energy dispersion of $\tphi$ as in Remark~\ref{rem:disp-psih-tphi}. For the second difference, we use the null form estimate \eqref{null(Afree)} together with the high frequency decay of $\tphi$ and low frequency decay of $\tA - A[t_{0}]$ due to the $c^*$ envelope bound \eqref{bd-low}. We conclude that
\begin{equation*}
\nrm{(\Box^{p, m}_{\tA^{free}} - \Box^{p, m}_{A^{free}}) \tphi}_{N[J]} \aleq_{F} 2^{Cm} \veps^{\dlt_{\ast}} + 2^{-cm}.
\end{equation*}
Finally, for the third term in \eqref{boxA-psih-decomp}, we use \eqref{diffa} and
\eqref{diffaa}. The gain comes from  the low frequency bound \eqref{bd-low} from Step 1.
This guarantees that, on one hand, $\tphi$ decays 
at high frequencies $> 2^{k^*}$, and on the other hand the differences $(A-\tA, \phi-\tphi)$ 
decay at low frequency $< 2^{k^*}$. As the frequency gap enforces a separation of 
at least $m$, from \eqref{diffa} and
\eqref{diffaa} we obtain
\[
\| (\Diff^{m}_{\bfA}(\tphi,\tphi,\tA) - \Diff^{m}_{\bfA}(\phi,\phi,A))\tphi \|_{N[J]} \lesssim_{F}  2^{-\dlt_{0} m}.
\]
 Summing up, the bounds for the three terms in $\Box_A \psi^{high}$, we conclude that 
\begin{equation}
\| \Box_A \psi^{high}\|_{N \cap L^{2} \dot{H}^{-\frac{1}{2}} \cap L^{\frac{9}{5}} \dot{H}^{-\frac{4}{9}}[J]} \lesssim_{F} 2^{Cm} \veps^{\dlt_{\ast}} + 2^{-cm} + 2^{-\dlt_{0} m} .
\end{equation}
Optimizing the choice of $m$, the desired estimate \eqref{boxA-psih} follows. \qedhere

\end{proof}

\section{Bilinear null form estimates}
\label{s:bi}

We begin our discussion with the bilinear null form estimates, which
play a key role in our analysis. These occur in both equations in the
MKG-CG system \eqref{MKG}. In the $\phi$ equation we have the
expression $A^j \partial_j \phi$, under the Coulomb gauge condition
$\rd^{j} A_{j} = 0$. We can rewrite this as
\begin{equation}\label{Aphi-null}
A^j \partial_j \phi = \partial_{k} \partial^{k} \Delta^{-1} A^j \partial_j \phi
= Q_{kj}(  \partial^{k} \Delta^{-1} A^{j}, \phi)
\end{equation}
where $Q_{kj}$ is the standard null form 
\[
Q_{kj}(u,v) = \partial_k u \partial_j v - \partial_j u \partial_k v
\]
In the $A$ equation, on the other hand, we encounter the expression
\begin{equation}\label{phiphi-null}
\mathcal P_j ( \phi \partial_x \bar \phi) 
=  \phi \partial_j \bar \phi - \partial^{k} \partial_j 
\Delta^{-1} (\phi \partial_k \bar \phi) = \partial^{k} \Delta^{-1} Q_{kj}(\phi,\bar \phi)
\end{equation}
Thus, it suffices to produce good estimates for the null form $Q_{ij}$.
 For that we have

\begin{proposition}
  Let $\mathcal N$ be one of the $Q_{ij}$ null form. Then the
  following bilinear estimates hold:
\begin{equation}\label{null}
\| P_{j} \mathcal N(\phi_k,\psi_l)\|_{N} \lesssim  2^{j} 2^{-\delta (|j-k|+|j-l|)}  \|\phi_k\|_{S^1} \|\psi_l\|_{S^1}  
\end{equation}
\begin{equation}\label{null-small}
\| Q_{<j-m} P_{j} \mathcal N(Q_{< k-m} \phi_k ,Q_{<l-m} \psi_l )\|_{N} 
\lesssim  2^{j} 2^{-\dlt m} 2^{C (|j-k|+|k-l|)}  \|\phi_k\|_{S^1} \|\psi_l\|_{S^1}  
\end{equation}
\end{proposition}

The first estimate above is the bound (131) in \cite{Krieger:2012vj}.
For the second bound we can harmlessly assume that $|j-k|+ |k-l| \ll C$. 
Then \eqref{null-small} is a consequence of the bound 
(143) in \cite{Krieger:2012vj}. 

We remark that the first bound \eqref{null} easily transfers to an interval $I$.
However, the second one involves modulation localizations, which are inconsistent
with interval localizations.

\bigskip

We now use the above null form estimates to conclude the proof 
of all the remaining results in Section~\ref{s:dec}, except for Proposition~\ref{p:diff}.

\bigskip

\begin{proof}[Proof of Proposition~\ref{p:ax}]
As in Remark~\ref{rem:ax-S}, it suffices to estimate the $\nrm{\Box (\cdot)}_{N}$ norms.
We begin with the quadratic part  $\A^{2}_x$. For simplicity, we concentrate on the case when the first two inputs are identical; the general case is a minor extension. We have 
\[
\Box \A_i^{2} (\phi,\phi)  =  \partial_j \Delta^{-1} \mathcal N(\phi,\bar \phi)
\]
therefore \eqref{axnl} follows by dyadic summation from \eqref{null}.

To prove the more refined bounds for $\A^{2}_x$ we use a large frequency gap $m$
 to first split 
\[
\begin{split}
\Box \A_i^{2} (\phi,\phi) = & \!\!\! \sum_{\max \set{|k-k_i|} \geq m} \!\!
P_k   \partial_j \Delta^{-1} \mathcal N(\phi_{k_1},\bar \phi_{k_2})
+ \sum_{|k-k_i| < m} \!\!
P_k   \partial_j \Delta^{-1} \mathcal N(\phi_{k_1},\bar \phi_{k_2})
\end{split}
\]
The first sum is estimated using \eqref{null} with a $2^{-cm}$
constant.

The second sum is essentially diagonal, so it suffices to estimate it
for fixed $k$. For this we consider two cases depending on the
relative size of the interval $I$.  The case of short intervals $|I|
\leq 2^{-k+m}$ is easy to dispense with, as we have
\[
\begin{split}
\| P_k   \partial_j \Delta^{-1} \mathcal N(\phi_{k_1},\bar \phi_{k_2})\|_{N[I]}
\lesssim & \ 2^{-k} \| \calN(\phi_{k_1},\bar \phi_{k_2})\|_{L^1 L^2[I]} 
\\  \lesssim & \  
 2^{Cm} (2^{k} |I|)^\frac12 \| |D|^\frac14 \phi_{k_1}\|_{L^4[I]} \| |D|^\frac14 \phi_{k_2}\|_{L^4[I]}. 
\end{split}
\]
Here we have a large $2^{Cm}$ constant, but the Strichartz norms on the right is divisible so this suffices for \eqref{axnl-large}.
Moreover, since $(4, 4)$ is a non-sharp pair of Strichartz exponents, it will be sufficient for \eqref{axnl-eps} too, as we explain below.


We are left with the most interesting case. To summarize, we have 
$|k-k_j| < m$ and $|I| > 2^{-k+Cm}$. To continue the proof we need to use
modulation localizations.   In order to be able to do that we extend $\phi$ outside our 
interval $I$ by homogeneous waves. Then we decompose
\[
P_k    \partial_j \Delta^{-1} \mathcal N(\phi_{k_1},\bar \phi_{k_2})
= Q_{<k-Cm}P_k\partial_j \Delta^{-1} 
\mathcal N(Q_{<k-Cm}\phi_{k_1},Q_{<k-Cm}\bar \phi_{k_2})+
  err(k,k_1,k_2)
\]
where the error corresponds to  at least one
modulation larger than $k-Cm$. The first  term is estimated using
 \eqref{null-small} with a $2^{-cm}$ constant. 
For the error we produce instead a direct bound, with two
cases:

\begin{itemize}
\item[(i)] High modulation output:
\[
\begin{split}
\|  Q_{>k-Cm} P_k   \partial_j \Delta^{-1} \mathcal N(\phi_{k_1},\bar \phi_{k_2})\|_{N[I]}
\lesssim & \ 2^{Cm} 2^{-\frac{3k}2} \| \chi_I^k \mathcal N(\phi_{k_1},\bar \phi_{k_2})\|_{L^2}
\\ \lesssim & \ 2^{Cm} \|  \chi_I^k |D|^\frac14 \phi_{k_1}\|_{L^4} \| \chi_I^k
 |D|^\frac14 \phi_{k_2}\|_{L^4} 
\\ \lesssim & \ 2^{Cm} \| |D|^\frac14 \phi_{k_1}\|_{L^4[I]} \|  |D|^\frac14 \phi_{k_2}\|_{L^4[I]} 
\end{split}
\]
where we used Proposition~\ref{p:ext} on the last line.
Here the relaxed cutoff $\chi_I^k$ was inserted in order to account
for the fact that the operator $ Q_{>k-Cm}$ is nonlocal in time.  Its
kernel decays rapidly on the $2^{Cm} 2^{-k}$ time scale, and this is exactly
the scale captured by $\chi_I^k$. Again the Strichartz norms on the
right are both divisible and non-sharp, so this bound suffices for
both \eqref{axnl-large} and \eqref{axnl-eps} (see below).

\item[(ii)] One high modulation input:
\[
\begin{split}
\|  Q_{<k-Cm} P_k \mathcal N(Q_{> k-Cm}\phi_{k_1},\bar \phi_{k_2})\|_{N[I]}
\lesssim & \  2^{Cm} 
 \|   \chi_I^k \mathcal N(Q_{> k-Cm} \phi_{k_1},\bar \phi_{k_2})\|_{L^1 L^2}
\\ \lesssim & \ 2^{Cm}   \| \Box  \phi_{k_1}\|_{L^2}  \| \chi_I^k \phi_{k_2}\|_{L^2 L^\infty} 
\end{split}
\]
This suffices for \eqref{axnl-large}. To complete the proof of \eqref{axnl-eps}
we also need to account for the case when $\phi_{k_{2}}$ has high modulation.
Then we have the following small variation of the previous computation:
\[
\begin{split}
\|  Q_{<k-Cm} P_k \mathcal N(\phi_{k_1},\overline{ Q_{> k-Cm} \phi_{k_2}})\|_{N[I]}
\lesssim & \  2^{Cm} 
 \|   \chi_I^k \mathcal N(\phi_{k_1},\overline{ Q_{> k-Cm} \phi_{k_2}})\|_{L^1 L^2}
\\
 \lesssim & \ 2^{Cm}  \| \chi_I^k \phi_{k_1}\|_{L^\frac94 L^\infty}   \| \Box  \phi_{k_2}\|_{L^{\frac95} L^2} \end{split}
\]
where the point is that $(\frac{9}{4}, \infty)$ is a non-sharp Strichartz exponent.
\end{itemize}

To conclude the proof of \eqref{axnl-eps} we observe that the above 
estimates allow us to use the $\veps$-energy dispersion and \eqref{box-phi} for all the large parts of $\A_x^2$.
Hence  we obtain a bound of the form
\[
\| \Box \A_x^{2}(\phi_1,\phi_2) \|_{N_{c}[I]} \lesssim 
(2^{-cm} + 2^{Cm} \veps^{\dlt_{1}})  \|\phi_1\|_{S^1[I]} \|\phi_2\|_{S^1_c[I]}
\]
Now \eqref{axnl-eps} easily follows by optimizing the choice of $m$.

Finally we consider the cubic terms $\A_i^{3}$,  which satisfy 
\[
\Box \A_i^{3} (\phi,\phi ,A_x) = \mathcal P( \phi \bar \phi A_x)
\]
At the dyadic level, using Bernstein's inequality in a favorable way we obtain 
\[
 \| P_k  (\phi_{k_1} \bar \phi_{k_2} A_{k_3}) \|_{L^1 L^2[I]} \lesssim 2^{-\delta \max |k-k_i|}
\|  |D|^{\frac14} \phi_{k_1}\|_{L^4[I]}  \| |D|^\frac14  \phi_{k_2}\|_{L^4[I]}    \|  A_{k_3}\|_{L^2 L^8[I]}
\]
All norms on the right are Strichartz norms and are bounded by the $S^1$ norms,
so \eqref{axnl} follows. Further, if say $\phi_1$ is $\veps$-energy dispersed, then we can bound 
its non-sharp Strichartz norm $L^4$ using the energy dispersion at the expense of 
losing the frequency envelope information, in order to obtain \eqref{axnl-eps3}. \qedhere 
\end{proof}

\bigskip

%
%

\begin{proof}[Proof of Proposition~\ref{p:ma}]
For the leading part  
\[
\mathcal M_A^{m,main} = 2i \sum_k P_{>k-m} A^j \partial_j P_k
\]
of $\mathcal M_A^m$ we have
\[
A^j \partial_j \phi = \mathcal N(  \partial_k \Delta^{-1} A_j, \phi)
\]
To decompose it into a small and a large part we first consider the frequency balance
of the two inputs and the output, depending on the frequency gap parameter $m \gg 1$. 
\[
\begin{split}
\mathcal M_{A}^{m,main} \phi =     2i \sum_k & \ 
  P_{ \geq k+m} A^j \partial_j P_k \phi  
+  P_{<k-m} (P_{< k-m} A^j \partial_j P_k \phi) + P_{\geq k-m}(  P_{[k-m,k+m)} A^j \partial_j   P_k \phi)
\end{split}
\]
The first two terms are estimated with a favorable $2^{-cm}$ constant
using \eqref{null}, and thus placed in $\mathcal M_{A,small}^{m, 2}$.  It
remains to consider the last term. This is essentially diagonal in
$k$, so we can freeze the three frequencies in the allowed range.

Now we consider the size of $I$.  As in the proof of
Proposition~\ref{p:ax} there is one easy case, namely when $|I|
\leq 2^{-k+m}$. Dispensing with that, from here on we assume 
that $|I| > 2^{-k+m}$.  The remaining argument uses modulation localizations.
To allow for that we extend both $A$ and $\phi$ outside $I$ 
as free waves. Then we decompose the last term above as
\[
\begin{split}
P_{\geq k-m}(  P_{[k-m,k+m)} A^j \partial_j   P_k \phi)   = & \ 
Q_{< k -Cm}   P_{\geq k-m}( Q_{<k-Cm}  P_{[k-m,k+m)} A^j \partial_j Q_{<k-Cm}  P_k \phi) 
\\ & \ + \mathcal M_{A,large}^{m,main} P_k \phi 
\end{split}
\]
In the first term, we gain $2^{-c m}$ by \eqref{null-small}. Hence this part can be put into $\mathcal M_{A, small}^{m, 2}$.
The remaining part $\mathcal M_{A,large}^{m,main} $ contains only terms where all three
frequencies are balanced, and at least one modulation is large. But this is estimated 
exactly as in the proof of Proposition~\ref{p:ax}:

\begin{itemize}
\item[(i)]  If the output has high modulation, then we bound it in $L^2$ using 
 using divisible non-sharp Strichartz norms to gain either the divisible bound
\eqref{man-large}, or  smallness via $\veps$-energy dispersion as in \eqref{man-disp}.

\item[(ii)] If the second input (i.e., $\phi$) has high modulation, then we combine  the $L^2 L^\infty$
bound for $A$ with  the  $L^2$ bound for $\Box \phi$.

\item[(iii)] If the first input (i.e., $A$) has high modulation, then we gain both
divisibility and smallness via energy dispersion by combining an
$L^\frac95 L^2$ bound for $\Box A$ and an $L^\frac94 L^\infty$ bound
for $\phi$.
\end{itemize}

\bigskip We now consider the remaining terms in $\mathcal M_A^m$.  For
the terms $P_{ \geq k-m} A_0 \partial_t \phi_k$ and $ \partial_t
A_0  \phi$ we estimate
\[
\| P_{k} \sum_{k_{2}} P_{ \geq k-m} A_0 \partial_t \phi_{k_{2}}\|_{L^1 L^2[I]} \lesssim 2^{-\dlt \max \set{\abs{k - k_{i}}}} \| P_{k_{1}} A_0 \|_{L^2 \dot H^\frac32[I]}
\| |D|^{-1} \partial_t \phi_{k_{2}} \|_{L^2 L^8[I]}
\]
respectively
\[
\| P_{k}( P_{k_{1}} \partial_t A_0  \phi_{k_{2}} )\|_{L^1 L^2[I]}  \lesssim 2^{-\dlt \max \set{\abs{k - k_{i}}}} \|  P_{k_{1}} \partial_t A_0 \|_{L^2 \dot H^\frac12 [I]}
\| \phi_{k_{2}} \|_{L^2 L^8 [I]} .
\]

Finally, the term $A^\alpha A_\alpha \phi$ is estimated in $L^1 L^2$ with off-diagonal gain
using only divisible non-endpoint Strichartz estimates, which
suffices.
\end{proof}

\section{Multilinear null form estimates}
\label{s:multi}

In this section we discuss directly the bounds for the operator 
$\Diff_\A^m$, and prove Proposition~\ref{p:diff}.  The bounds
\eqref{diffa} and \eqref{diffaa} were already proved in \cite{Krieger:2012vj}.
The delicate matter is to be able to estimate the bulk of 
$\Diff_\A^{m,2}(\phi,\phi) \psi$ in terms of the divisible norm 
$\DS$ of $\phi$. 
We split our argument into two steps:

\begin{itemize}
\item[(i)] First we review the decompositions and the estimates in \cite{Krieger:2012vj}
on the full real line, leading to the proof of \eqref{diffa} and
\eqref{diffaa}. But we do this in a careful fashion so that we can
isolate a bulk part where we get smallness from the frequency gap, and
a remaining part where this does not work. For this remaining part we
can easily produce a divisible bound. Unfortunately, this last
argument uses modulation localizations.

\item[(ii)] Secondly, we consider the changes in the previous arguments when the analysis 
is done on a compact interval $I$. The challenge here is to be able to accurately estimate 
the large but divisible part using only information localized to our interval.
\end{itemize}

\subsection{A review of \cite{Krieger:2012vj}}

We decompose $ \Diff_\A^m$ into 
\[
\Diff_\A^m = \H^* \Diff_\A^m + (I-\H^*) \Diff_\A^m
\]
where the operator $\H^*$, introduced in \cite{Krieger:2012vj}, selects the case
where $\A$ has high modulation while both the input and the output
have small modulation,
\[
 \H^* \Diff_\A^m \psi = \sum_{k_0 < k-m}  \sum_{j < k_{0}} Q_{<j- 2} ( Q_{ j}  P_{k_{0}} \A^\alpha 
Q_{<j- 2} \partial_\alpha \psi_k   )
\]
The better part $(I-\H^*) \Diff_\A^m$ can be still dealt with in a bilinear fashion
using the following result:
\begin{proposition}
We have the bilinear estimate
\begin{equation} \label{imh}
\| (I-\H^*) \Diff_{A}^m \psi\|_{N} \lesssim (\| A_x\|_{\ell^1 S^1} + 
\| \nabla A_0\|_{\ell^1 L^2 \dot H^\frac12}) \|\psi\|_{S^{1}} \ .
\end{equation}
Further, we have the low modulation improvement
\begin{equation}\label{imh-small}
\| (I-\H^{*}) \Diff_{A^{low}}^m  \psi\|_{N} \lesssim 2^{-cm}(\| A_x\|_{\ell^1 S^1} + \| \nabla A_0\|_{\ell^{1} L^2 \dot H^\frac12}),
\|\psi\|_{S^1}
\end{equation}
where
\begin{equation*}
	A^{low} = \sum_{k_{0}} Q_{<k_{0}-m} P_{k_{0}} A.
\end{equation*}

\end{proposition}

The bound \eqref{imh} is the sum of the bounds (54) and (58) in
\cite{Krieger:2012vj}.  The bound \eqref{imh-small} is a corollary of the proof of
(54)\cite{Krieger:2012vj}; it follows from a similar improvement in the bound
(132)\cite{Krieger:2012vj},  which in turn is tied to the fact that the estimate (143)\cite{Krieger:2012vj}
is only used in the case when $j_1 < k_2-m$.

We now turn our attention to the term $\H^*\Diff_{\A}^m$, where it is
no longer enough to obtain bounds depending on the above norms of $\A$.
Our first tool here is the intermediate norm $Z$, which has the
following properties:

\begin{proposition}
We have the bilinear estimates:
\begin{equation}\label{axz}
\| \H^*\Diff_{A_x}^m \psi \|_{N} \lesssim \|A_x\|_{Z} \|\psi\|_{S^1}
\end{equation}
respectively
\begin{equation}\label{a0z}
  \| \H^*\Diff_{A_0}^m \psi \|_{N} \lesssim \|A_0\|_{\Delta^{-\frac12} \Box^\frac12 Z+L^1 L^\infty}
 \|\psi\|_{S^1}
\end{equation}
as well as the low modulation improvement
\begin{equation}\label{axz-low}
\| \H^*\Diff_{A_x^{low}}^m \psi \|_{N} \lesssim 2^{-cm} \|A_x\|_{Z} \|\psi\|_{S^1}
\end{equation}
respectively
\begin{equation}\label{a0z-low}
\| \H^*\Diff_{A_0^{low}}^m \psi \|_{N} \lesssim 2^{-cm} \|A_0\|_{\Delta^{-\frac12} \Box^\frac12 Z} \|\psi\|_{S^1} \ .
\end{equation}
\end{proposition}
These are the bounds (133) and (140) in \cite{Krieger:2012vj}, where the low modulation improvement is again apparent from the proofs.

Combining the  estimates \eqref{imh}, \eqref{axz}, \eqref{a0z}  and \eqref{B_embed} we 
can eliminate the modulation localizations and obtain 
\begin{corollary}
The following estimate holds:
\begin{equation}\label{a3-all}
\| \Diff_A^{m} \psi\|_{N} \lesssim (\| A_x\|_{\ell^1 S^1} +   \| \Box A_x\|_{\ell^1 L^1 L^2} +
\| \nabla A_0\|_{\ell^1 L^2 \dot H^\frac12} + \| A_0\|_{L^1 L^\infty}) \|\psi\|_{S^1}
\end{equation}
\end{corollary}
Using this estimate, we can dispense with the cubic contributions due to $\A^{3} = (\A_{0}^{3}, \A_{x}^{3})$. Indeed, combined with the bounds \eqref{axnl3}, \eqref{axnl-hm3} and \eqref{a0}, as well as \eqref{lin} to control $\nrm{\A_{x}^{3}}_{\ell^{1} S^{1}}$, we can use \eqref{a3-all} to establish \eqref{diffaa}. We remark that the frequency envelope bound in \eqref{diffaa} is clear from the frequency gap $m$ between the two inputs $\bfA^{3}$ and $\phi$.


The output of the quadratic part of $\A$ cannot be all dealt with
using the $Z$ norm, but a good portion of it is amenable to this
strategy.  This is described using the operators $\H_{k_0}$ defined by
\[
\H_{k_0}  \A^{2}(\phi_{k_1},\phi_{k_2})= \sum_{j< k_0} Q_j P_{k_0}  \A^{2}(Q_{<j}\phi_{k_1},
Q_{<j}\phi_{k_2})
\]

Precisely, the portion of  $\A^{2}$ which does not have good $Z$ bounds is
\[
\H^m \A^{2} = \sum_{k_0 < k_1-m} \H_{k_0}  \A^{2}(\phi_{k_1},\phi_{k_2})
\]
A key result in \cite{Krieger:2012vj} is to treat the output of this part in a
genuine trilinear fashion, taking advantage of a cancellation between
the $A_0$ and $A_x$ parts, which have otherwise been treated
separately. Precisely, we have

\begin{proposition}
For any admissible frequency envelopes $c, d, e$, we have
\begin{equation}\label{cubic}
\| \H^* \Diff^m_{\H^m \A^{2} (\phi_{1},\phi_{2})} \psi \|_{N_{f}} \lesssim 2^{-cm} \|\phi_{1}\|_{S^1_{c}} \|\phi_{2}\|_{S^1_{d}} \|\psi\|_{S^1_{e}},
\end{equation}
where $f(k)$ is as in \eqref{diffa-freqenv}.
\end{proposition}
For this we refer the reader to the estimate (60) in \cite{Krieger:2012vj} and
its dyadic versions (136)-(138), where the frequency envelope bound and the gain with respect to $m$ are
apparent.

Hence it remains to bound
\[
\| (I - \H^m) \A_x^{2}(\phi,\phi)\|_{Z + L^1 L^\infty}, \qquad \|
(I - \H^m) \A_0^{2}(\phi,\phi)\|_{\Delta^{-\frac12} \Box^\frac12
  Z+L^1L^\infty} \ .
\]
Considering the dyadic portions
\[
P_{k_0} \A^{2}(\phi_{k_1},\phi_{k_2}),
\]
the case of high-high interactions  was also discussed in \cite{Krieger:2012vj}.
Precisely, from the bounds (134) and (141) in \cite{Krieger:2012vj} we have

\begin{proposition}
For $k_{0} < k_{1} - C$, we have the dyadic bound
\begin{equation}\label{hh-z}
\begin{aligned}
& \| (I - \H^m) P_{k_0} \A_x^{2}(\phi_{k_1},\phi_{k_2})\|_{Z}+ \|
(I - \H^m) P_{k_0} \A_0^{2}(\phi_{k_1},\phi_{k_2})\|_{\Delta^{-\frac12} \Box^\frac12
  Z} \\
  & \quad \lesssim 2^{- \dlt |k_0-k_1|} \| \phi_{k_{1}} \|_{S^1} \| \phi_{k_{2}} \|_{S^1} \ .
\end{aligned}\end{equation}
\end{proposition}

This suggests that we should decompose $\A^{2}$ into a $high \times high \to low$
portion and a better reminder. We will be more accurate and set
\[
\begin{split}
\A^{2} = & \ \sum_{k_1 \geq k_0+m}  P_{k_0} \A^{2}(\phi_{k_1},\phi_{k_2})
+ \sum_{k_0-m < k_{1,2} < k_0+m}  P_{k_0} \A^{2}(\phi_{k_1},\phi_{k_2})
+ \sum_{k_{min} < k_0- m}  P_{k_0} \A^{2}(\phi_{k_1},\phi_{k_2}) \\ := & \ 
 \A^{2,hh(m)} +\A^{2,med(m)} + 
 \A^{2,hl(m)} ,
\end{split}
\]
where $k_{min} = \min \set{k_{1}, k_{2}}$.
Note that no modulation localizations are present here. 

We first handle the part $\bfA^{2, hh(m)}$. Recall from the proofs of \eqref{axnl}, \eqref{axnl-hm}, \eqref{a0} and \eqref{d0a0} that there is a bound with an off-diagonal decay of the form
\begin{equation} \label{hh-rest}
\begin{aligned}
 	\nrm{P_{k_{0}} \A_{x}^{2}(\phi_{k_{1}}, \phi_{k_{2}})}_{S^{1}}
	+ \nrm{P_{k_{0}} \A_{0}^{2}(\phi_{k_{1}}, \phi_{k_{2}})}_{L^{2} \dot{H}^{\frac{3}{2}}}  
	\aleq 2^{- \dlt (\abs{k_{0} - k_{1}} + \abs{k_{0} - k_{1}})} \nrm{\phi_{k_{1}}}_{S^{1}} \nrm{\phi_{k_{2}}}_{S^{1}} 
\end{aligned}\end{equation}
when $k_{0} < k_{1} - C$.

Combining 
the bounds \eqref{cubic}, \eqref{hh-z} with \eqref{axz}, \eqref{a0z} for the $\H^*$ portion
and \eqref{hh-rest} with \eqref{imh} for 
the $I-\H^*$ portion, we obtain the following:	
\begin{corollary}
For any admissible frequency envelopes $c, d, e$, we have
\begin{equation}\label{hh-a2}
\| \Diff_{\A^{2,hh(m)}}^m(\phi_{1}, \phi_{2}) \psi\|_{N_{f}} \lesssim 2^{-cm} \| \phi_{1} \|_{S^1_{c}} \| \phi_{2} \|_{S^1_{d}} \|\psi\|_{S^1_{e}} \, 
\end{equation} 
where $f(k)$ is as in \eqref{diffa-freqenv}. 
\end{corollary}
Again, no modulation localizations are present here.

The remaining parts of $\A^2$ have no contributions from $\H^m \A^2$,
so we will estimate them entirely using the $Z$ norm or the simpler
$L^1 L^\infty$ bound. The latter suffices in the case of $\A_0$,
whose dyadic pieces are  readily bounded by 

\begin{equation} \label{hl-med-a0}
\begin{split}
  \| P_{k_0} \A_0^{2}(\phi_{k_1},\phi_{k_2})\|_{L^1 L^\infty}
  \lesssim & \ 2^{-2k_0} \|P_{k_0} \Delta
  \A_0^{2}(\phi_{k_1},\phi_{k_2})\|_{L^1 L^\infty} \\ \lesssim & \
  2^{-2(k_0 - k_{max}) -\frac12 |k_1-k_2|} \prod_{j=1,2} \| (|D|^{-\frac12}
  \phi_{k_j}, \abs{D}^{-\frac{3}{2}} \rd_{t} \phi_{k_{j}}) \|_{L^2 L^\infty} 
 \end{split} 
\end{equation}
where $k_{max} = \max\set{k_{1}, k_{2}}$. In both $\A_{0}^{2, hl(m)}$ and $\A_{0}^{2, med(m)}$, note that
we have $k_{\max} \leq k_{0} + m$ for (say) $m \geq 3$. Moreover, the above dyadic bound 
sums up easily due to the off-diagonal decay. Using the embedding $L^{1} L^{2} \subset N$, we easily obtain
\begin{corollary} 
For any admissible frequency envelopes $c, d, e$, we have
\begin{equation}\label{hl-a02}
\| \Diff_{\A_0^{2,hl(m)}}^m(\phi_{1},\phi_{2}) \psi\|_{N_{f}} 
+\| \Diff_{\A_0^{2,med(m)}}^m(\phi_{1},\phi_{2}) \psi\|_{N_{f}}  
\lesssim 2^{Cm} \| \phi_{1} \|_{\DS_{c}} \nrm{\phi_{2}}_{\DS_{d}} \|\psi\|_{S^1_{e}}
\end{equation} 
where $f(k)$ is as in \eqref{diffa-freqenv}. 
\end{corollary}

We now  consider the contributions of $\A_x^{2,hl(m)}$ and $\A_x^{2,med(m)}$. Our first tool
is due to the estimates (134) and (135) in \cite{Krieger:2012vj}, which give
\begin{proposition}
The following estimate holds:
\begin{equation}\label{hl-z}
\| P_{k_{0}} \A_x^{2}(\phi_{k_1},\phi_{k_2})\|_{Z} \lesssim 2^{C|k_{0}-k_{max}|} 2^{-\delta |k_1-k_2|} \| \phi_{k_1}\|_{S^1} \| \phi_{k_2}\|_{S^1} 
\end{equation}
\end{proposition}
This gives a gain for the high-low portion of $A_x$. Hence in combination with \eqref{axz}, \eqref{a0z},  \eqref{hh-z} for the $\H^*$ portion and
\eqref{imh}, \eqref{hh-rest} for 
the $I-\H^*$ portion, we obtain a result with no modulation localizations:

\begin{corollary}
For any admissible frequency envelopes $c, d, e$, we have
\begin{equation}\label{hl-a2}
\| \Diff_{\A^{2,hl(m)}}^m(\phi_{1},\phi_{2}) \psi \|_{N_{f}} \lesssim 2^{-cm} \| \phi_{1} \|_{S^1_{c}} \| \phi_{2} \|_{S^1_{d}} \|\psi\|_{S^1_{e}}
\end{equation} 
where $f(k)$ is as in \eqref{diffa-freqenv}. 
\end{corollary}

Finally, it remains to consider the contribution of $\A_x^{2,med(m)}$. There the estimate 
\eqref{hl-z} suffices  for the bound \eqref{diffaa}, but provides no divisible norm estimate.
To summarize, we are left with the case
\begin{equation*}
	k_{0} - m < k_{1}, k_{2} < k_{0} + m.
\end{equation*}
Here we can take advantage of the low modulation decay in
\eqref{imh-small} and \eqref{axz-low} to obtain 
\begin{corollary}
The following bound holds for large enough $C$ and $k_0$, $k_1$, $k_2$ as above:
\begin{equation}\label{med-a2}
  \| \Diff^{m}_{\A_x^{2,med(m), low}(\phi_{k_1},\phi_{k_2})} \psi_{k}\|_{N} \lesssim 2^{-cm} \| \phi_{k_{1}} \|_{S^1} \nrm{\phi_{k_{2}}}_{S^{1}} \|\psi_{k}\|_{S^1}
\end{equation} 
where 
\begin{equation*}
\A_x^{2,med(m), low} = \sum_{k_{0}} Q_{<k_{0} - m} P_{k_{0}} \A_{x}^{2, med(m)} \ .
\end{equation*}
\end{corollary}
Thus we can restrict ourselves to
high modulations in $\A_x$, i.e., 
\begin{equation*}
\A_{x}^{2, med(m), high}(\phi_{1}, \phi_{2}) = \sum_{k_{0}} P_{k_{0}} Q_{>k_{0} -Cm} \A_{x}^{2, med(m)}(\phi_{1}, \phi_{2}).
\end{equation*}
For this part, we can use the $L^1L^\infty$ norm. 
Precisely, each dyadic piece obeys the estimate
\begin{align*}
\nrm{P_{k_{0}} Q_{>k_{0} -Cm} \A_{x}^{2}(\phi_{k_{1}}, \phi_{k_{2}})}_{L^{1} L^{\infty}}
\aleq & 2^{-2k_{0}+Cm} \nrm{\Box \A_{x}^{2}(\phi_{k_{1}}, \phi_{k_{2}})}_{L^{1} L^{\infty}}  \\
\aleq & 2^{Cm} \nrm{\abs{D}^{-\frac{1}{2}} \phi_{k_{1}}}_{L^{2} L^{\infty}} \nrm{\abs{D}^{-\frac{1}{2}} \phi_{k_{2}}}_{L^{2} L^{\infty}} \ .
\end{align*}
Recall that we are in the scenario $k_{0} - m < k_{1}, k_{2} < k_{0} + m$. Combined with the embeddings $L^{1} L^{2} \subset N$ and $\nb S^{1} \subset L^{\infty} L^{2}$, we obtain
\begin{equation*}
  \| \Diff^{m}_{\A_{x}^{2, med(m), high}(\phi_{1}, \phi_{2})} \psi \|_{N_{f}} \lesssim 2^{Cm} \| \phi \|_{\DS_{c}} \nrm{\phi}_{\DS_{d}} \|\psi_{k}\|_{S^1_{e}}
\end{equation*}
where $c, d, e$ are any admissible frequency envelopes and $f$ is as in \eqref{diffa-freqenv}.
Thus the proof of Proposition~\ref{p:diff} is concluded on the entire real line.

\subsection{Interval localized bounds}

Here we seek to prove the result of Proposition~\ref{p:diff} in a time interval $I$.
Due to the paradifferential nature of the operator $\Diff_\A^{m}$, we can fix the 
frequency $2^k$ of the input $\psi$ and simply estimate the expression
$\Diff_\A^m \psi_k$. For $\A$ we consider its components successively:

a) The cubic terms $\A^3$. Here we simply extend $\A^3$ outside $I$ as a homogeneous 
wave, and then use the bound \eqref{a3-all}. 
By Propositions~\ref{p:ax}, \ref{p:axnl-high-mod} and \ref{p:a0},
we know that $\A^{3}$ is entirely estimated by divisible norms.
%

b) The contributions of $\A^{2,hh(m)}$ and $\A_x^{2,hl(m)}$. Here we extend $\phi$ outside 
$I$ as a homogeneous wave, and then apply  \eqref{hh-a2}, respectively \eqref{hl-a2}.

c) The contributions of $\A_0^{2,med(m)}$ and $\A_0^{2,hl(m)}$. These are estimated 
directly via \eqref{hl-a02}; no extensions are necessary.

d) The contribution of $\A_x^{2,med(m)}$. This is the  part where  the divisible bound 
is more difficult to gain. In what follows, we simply write $\A_{x} = \A^{2}_{x}$. To review, we have to estimate the expression 
\[
\| P_{k_0} \A^i(\phi_{k_1}, \phi_{k_2})    \partial_i \psi_k\|_{N}
\]
where the frequency balance is 
\[
k_0 < k-m, \qquad k_0-m < k_1, k_2 < k_0+m \ .
\]
This is where  the length of the time interval $I$ plays a role. 
Comparing it to $k_0$, we distinguish 
two scenarios:

(i) Short time intervals, $|I| \leq 2^{-k_0+m}$. Then we have a direct estimate,
\[
\begin{split}
  \| P_{k_0} \A^i(\phi_{k_1}, \phi_{k_2}) \partial_i \psi_k\|_{N[I]}
  \lesssim & \ \| P_{k_0} \A^i(\phi_{k_1}, \phi_{k_2}) \partial_i
  \psi_k\|_{L^1 L^2[I]} 
  \\ \lesssim & |I|^{\frac12} \| P_{k_0} 
  \A_x(\phi_{k_1}, \phi_{k_2}) \|_{ L^2 L^\infty[I]} 
  \| \nabla_{x} \psi_k\|_{L^\infty L^2[I]}
  \\
  \lesssim & \ 
  2^{Cm} \| |D|^{-\frac12} P_{k_{0}} \A_{x}(\phi_{k_{1}}, \phi_{k_{2}}) \|_{L^2 L^\infty[I]} 
  \| \nb_{x} \psi_k\|_{L^\infty L^2[I]} \ . 
\end{split}
\]
Summing over $k_{0}, k_{1}, k_{2}$ and recalling the definition of the $S^{str}_{k_{0}} \subset S_{k_{0}}$, we obtain
\begin{equation*}
\nrm{\Diff^{m}_{\A^{2, med(m)}_{x}(\phi_{1}, \phi_{2})} \psi_{k}}_{N[I]} \aleq 2^{Cm} \nrm{\nb P_{<k-m} \bfA^{2, med(m)}_{x}(\phi_{1}, \phi_{2})}_{\ell^{1} S[I]} \nrm{\psi_{k}}_{S^{1}[I]},
\end{equation*}
The right hand side can be controlled by Proposition~\ref{p:ax}. The splitting into small and large parts is then achieved\footnote{Technically, \eqref{axnl-small}-\eqref{axnl-large} apply to the full operator $\A_{x}^{2}$. Nevertheless, $\A_{x}^{2} - \A_{x}^{2, med(m)}$ gains $2^{-cm}$ by \eqref{hh-rest}, and thus this difference can be put into the `small' part.} by using the corresponding statements \eqref{axnl-small}-\eqref{axnl-large} for $\A_{x}^{2}$.

(ii) Long time intervals, $|I| >  2^{-k_0+m}$. This is the difficult
case.
 Our proof here involves
modulation localizations, so we need to consider appropriate extensions
of $\A_x$ and $\psi_k$.  Since $\psi_k$ is an independent variable, for it we can 
simply use the canonical extension as homogeneous waves. For $\A_{x}$, instead,
we extend its arguments $\phi_1$ and $\phi_2$  as homogeneous waves.

The bound \eqref{med-a2} suffices for low modulations of $\A^{2,med(h)}$,
therefore it suffices to estimate
\[\begin{split}
& 
\sum_{k_{0} < k - m}  \| Q_{> k_0-Cm} P_{k_0} \A_x( \phi_{k_1},\phi_{k_2}) \nabla_x \psi_{k}\|_{N[I]}
\lesssim  \ 
\\ & \qquad \qquad \sum_{k_{0} < k-m} \| Q_{> k_0-Cm} P_{k_0} \A_x( \phi_{k_1},\phi_{k_2})\|_{L^1 L^\infty[I]}
 \|\nabla_x \psi_{k} \|_{L^\infty L^2}
\end{split}\]
To estimate the localized $L^1 L^\infty[I]$ norm we write
\[
\begin{split}
Q_{> k_0-Cm} P_{k_0} \A_x( \phi_{k_1},\phi_{k_2}) = & \ 
Q_{> k_0-Cm} P_{k_0} \Box^{-1} \bb( \frac{1}{2} \calP_{x} \Im(\phi_{k_1} \nabla_{x} \overline{\phi}_{k_2}+ 
\nabla_{x} \overline{\phi}_{k_1} \phi_{k_2}) \bb)
\\ = & \ 2^{-k_0+2Cm}  L(\phi_{k_1},\phi_{k_2})
\end{split}
\]
where $L$ is a bilinear translation invariant form whose kernel is localized
near $0$ on the $2^{-k_0}$ scale in space-time. This allows us to estimate the tails outside $I$
as follows:
\[
\| Q_{> k_0-Cm} P_{k_0} \A_x( \phi_{k_1},\phi_{k_2})\|_{L^1 L^\infty[I]}
\lesssim  2^{-k_{0}} 2^{Cm} \| \chi_{I}^{k_0} \phi_{k_1}\|_{L^2 L^\infty}   
\| \chi_{I}^{k_0} \phi_{k_2}\|_{L^2 L^\infty}  .
\]
Since $k_1$ and $k_2$ are close to $k$, we conclude using Proposition~\ref{p:ext}
that
\[
\| Q_{> k_0-Cm} P_{k_0} \A_x( \phi_{k_1},\phi_{k_2})\|_{L^1 L^\infty[I]}
\lesssim 2^{Cm} \| \phi_{k_1}\|_{\DS[I]}  \| \phi_{k_2}\|_{\DS[I]} 
\]
which is the sought after divisible bound.  The proof of
Proposition~\ref{p:diff} is concluded.

\section{The paradifferential parametrix}
\label{s:para}

The goal of this section is to prove
Theorem~\ref{t:para-free}. Instead of producing an exact solution
operator, it is easier to produce parametrix with small errors.  Then
the exact solution is obtained in a straightforward iterative
fashion. The result we produce here is as follows:

\begin{theorem} \label{t:app}
Let $A_x$ be a Coulomb magnetic potential solving the free wave equation
with energy $E$, and let $m > 5$. 
Consider any finite energy initial data $(\phi_{0}, \phi_{1})$ localized in frequency $\approx 1$, and a source $f \in N$ which is localized in frequency $\approx 1$ and modulation $\aleq 1$. 
Then there exists an approximate solution $\phi$ so that 
\begin{equation}
\begin{gathered}
\|\phi\|_{S_{0}} \lesssim_{E} \nrm{(\phi_{0}, \phi_{1})}_{L^{2} \times L^{2}} + \nrm{f}_{N_{0}}, \\
\|\phi[0]-(\phi_0,\phi_1)\|_{L^{2} \times L^{2}}+
 \| \Box_{A}^{p,m} \phi -f\|_{N_{0}} \lesssim_E 2^{-c m} (\|(\phi_0,\phi_1)\|_{L^{2} \times L^{2}} + \|f\|_{N_{0}}).
\end{gathered}
\end{equation}
\end{theorem} 
We remark that the frequency support of the approximate solution $\phi$ is only slightly larger compared to $\phi_{0}, \phi_{1}$ and $f$; it is essentially also localized at frequency $\approx 1$ and modulation $\aleq 1$. After choosing $m$ sufficiently large, Theorem~\ref{t:app} directly implies Theorem~\ref{t:para-free}; see \cite[Proof of Theorem 6.3]{Krieger:2012vj}.  

The definition of
our parametrix is identical to the one used for the small data problem
in \cite{Krieger:2012vj}, which was based on \cite{MR2100060}. The main difference is in the source of smallness for
the errors.  In \cite{Krieger:2012vj} this comes from the smallness of the energy
of $A$. Here, we rely instead on the frequency gap $m$, which must be
large in terms of the energy $E$.

 The parametrix is  constructed using pseudodifferential
operators with rough symbols. Given a symbol $a(t, x, \tau, \xi)$, its left- and
right-quantizations are denoted $a(t, x, D)$ and $a(D, y, s)$,
respectively. We also use the standard convention $D_{\mu} =
\frac{1}{i} \rd_{\mu}$. 

To prove the theorem it suffices 
to consider initial data $\phi[0]$ and source $f$ with frequency localization in $\approx 1$,
and construct the approximate solution $\phi$ with a similar localization.
Thus we work with  the unit-frequency localized paradifferential magnetic wave operator
\begin{equation} \label{eq:paraMagWave}
	\Box_{A}^{p,m} = \Box  + 2 i A_{< - m}^{j} \rd_{j} .
\end{equation}
where $A$ solves the free wave equation $\Box A = 0$ with initial data
$A[0] \in \dot{H}^{1}_{x} \times L^{2}_{x}$.

Given an additional small angular localization parameter $0 < \sgm <
1/2$, we construct a parametrix for \eqref{eq:paraMagWave} as follows.
For $\xi \in \bbR^{4}$  we define
\begin{equation*}
	\omg := \frac{\xi}{\abs{\xi}}, \qquad
	\oL_{\pm} := \rd_{t} \pm \omg \cdot \nb_{x}, \qquad
	\opLap := \lap - (\omg \cdot \nb_{x})^{2}.
\end{equation*}

Note that
\begin{equation*}
	\Box = - L^{\omg}_{+} L^{\omg}_{-} + \opLap.
\end{equation*}

Define the angular sector projection $\oPi_{> \tht}$ by the formula
\begin{equation*}
	\calF[\oPi_{> \tht} f](\xi) := \big(1 - \eta( \frac{\angle(\xi, \omg)}{\tht}) \big)  \big(1 - \eta( \frac{\angle(- \xi, \omg)}{\tht}) \big) \widehat{f}(\xi).
\end{equation*}

It is important to note that if $f$ is real, then so is $\oPi_{> \tht} f$. We also define
\begin{equation*}
	\oPi_{\tht} := \oPi_{>\frac{\tht}{2}} - \oPi_{>\tht}, \qquad
	\oPi_{\leq \tht} := 1 - \oPi_{> \tht}.
\end{equation*}

For each $\ell \leq 0$, we define $\psi_{\ell, \pm}$ to be
\begin{equation} \label{}
	\psi_{\ell, \pm}(t,x, \xi) := \pm \oL_{\pm} \opLap^{-1} \oPi_{> 2^{\sgm \ell}} (\omg \cdot P_{\ell} A) 
\end{equation}

The full phase $\psi_{\pm}$ is then defined to be
\begin{equation} \label{eq:defn4psi}
	\psi_{\pm}(t,x, \xi) = \psi_{<-m, \pm}(t,x,\xi) := \sum_{\ell < -m} \psi_{\ell, \pm}(t,x, \xi).
\end{equation}

Note that we have
\begin{equation} \label{eq:oLPsi}
	\oL_{\mp} \psi_{\pm} = \pm \sum_{\ell < - m} \oPi_{> 2^{\sgm \ell}} (\omg \cdot P_{\ell} A)
\end{equation}
In other words, $\psi_{\pm}$ represent roughly the output of the integration of the 
(bulk of the) magnetic potential $A$ along light rays. Here we exclude the output
of small angle interactions, which is on one hand perturbative, and on the other hand
would yield a bad dependence of $\psi_{\pm}$ on $\xi$. This is akin to symbol smoothing 
for rough pdo's.

We use the pseudodifferential gauge transform
\begin{equation*}
	e^{- i \psi_{\pm}}_{<0}(t, x, D) := (S_{<0} e^{- i \psi_{\pm}})(t, x, D)
\end{equation*}
where $S_{<0}$ is taken with respect to the $(t,x)$ variables of the symbol. Its dual is 
\begin{equation*}
	e^{i \psi_{\pm}}_{<0}(D, y, s) 
\end{equation*}
As the symbol is independent of the time Fourier variable $\xi_{0} =
\tau$, we see that the left and right quantizations with respect to
$t$ are the same, i.e.,
\begin{equation*}
	e^{i \psi_{\pm}}_{<0}(D, y, s)  = e^{i \psi_{\pm}}_{<0}(t, D, y).
\end{equation*}

The operators $e^{- i \psi_{\pm}}_{<0}(t, x, D)$, respectively $e^{i
  \psi_{\pm}}_{<0}(D, y, s) $ are used on the left and on the right in
order to conjugate the paradifferential operator $\Box_{A}^{p,m}$ to
the d'Alembertian $\Box$. Precisely, our parametrix is given by
\begin{equation}
\begin{split}
\phi(t,x) = & \ e^{- i \psi_{\pm}}_{<0}(t, x, D)  |D|^{-1} e^{\pm it|D|} 	e^{i \psi_{\pm}}_{<0}(D, y, 0) 
( |D| \phi_0 \pm \phi_1)  \\ & \ 
 + \frac12 \int_{0}^t  e^{- i \psi_{\pm}}_{<0}(t, x, D)  |D|^{-1} e^{\pm i(t-s)|D|} e^{i \psi_{\pm}}_{<0}(D, y, s ) Q_{\pm} f(s) ds
\end{split} 
\end{equation}

To show that the above parametrix satisfies the bounds in
Theorem~\ref{t:app} we need the following mapping properties for the
operators $e^{- i \psi_{\pm}}_{<0}(t, x, D)$, respectively $e^{i\psi_{\pm}}_{<0}(D, y, s) $:
 
\begin{theorem} \label{thm:mEst} For $m > 0$, let $\psi_{\pm}$ be
  defined as in \eqref{eq:defn4psi}. Then the following mapping
  properties hold with $Z \in \{ N_0, L^2,N_0^*\}$, with implicit
  constants which depend on the energy $E$ of $A$:
\begin{enumerate}
\item (Boundedness)
\begin{equation}\label{para-bd}
e^{\pm i \psi_{\pm'}}_{<0}(t,x,D) : \quad  Z \to Z
\end{equation}

\item (Dispersive estimates)
\begin{equation}\label{para-disp}
	e^{\pm i \psi_{\pm'}}_{<0}(t,x,D) : \quad  S^{\sharp}_{0} \to S_{0}
\end{equation}

\item (Derivative bounds)
\begin{equation} \label{eq:mEst:DtPtx}
	(\nb e^{\pm i \psi_{\pm}}_{<0})(t,x) : Z \to 2^{-m} Z
\end{equation}

\item (Approximate unitarity of $e^{i \psi_{\pm}}$ on $L^{2}_{x}$)
For each $t \in \bbR$, we have
\begin{equation} \label{eq:mEst:L2}
e^{- i \psi_{\pm}}_{< 0}(t,x, D) e^{i \psi_{\pm}}_{<0} (t, D, y) - I : 
L^{2}_{x} \to 2^{-(1-\dlt_{0}) m} L^{2}_{x}
\end{equation}

\item (Approximate unitarity of $e^{i \psi_{\pm}}$ on $N$) We have
\begin{equation} \label{eq:mEst:N}
e^{- i \psi_{\pm}}_{< 0}(t,x, D) e^{i \psi_{\pm}}_{<0} (D, y, s) - I : N_{0} \to 2^{-\dlt_{1} m} N_{0}
\end{equation}
 
\item (Parametrix error estimate) We have
\begin{equation} \label{eq:mEst:ptxError}
	e^{- i \psi_{\pm}}_{<0}(t,x,D) \Box - \Box^{p}_{A} e^{- i \psi_{\pm}}_{<0} (t,x,D) : S^{\sharp}_{0, \pm} \to 2^{- \dlt_{2} m} N_{0, \pm}
\end{equation}
\end{enumerate}
\end{theorem}

\begin{remark}
The small constants $\sgm, \dlt_{0}, \dlt_{1}, \dlt_{2}$ and $\dlt$ are now \emph{different} from those used in the earlier part of the paper. They are chosen in the following logical order: $\sgm$, $\dlt_{0}$, $\dlt_{1}$, $\dlt_{2}$. On the other hand, we reserve the symbol $\dlt > 0$ for a free small number, whose value may vary depending on the usage.
\end{remark}
This result  mirrors Theorem~3 in \cite{Krieger:2012vj}, with the key difference that the 
smallness is now due to the frequency gap parameter $m$.
 Assuming these bounds, the conclusion of  Theorem~\ref{t:app} follows in the 
same way as in \cite{Krieger:2012vj}. 

To prove the above theorem, we may directly borrow the estimates from
\cite{Krieger:2012vj} which do not involve smallness, namely
\eqref{para-bd} and \eqref{para-disp}. The implicit constant in these estimates 
will now depend on the energy $E$ of $A$. The remainder of the section 
is devoted to the proof of the new bounds  \eqref{eq:mEst:DtPtx}, \eqref{eq:mEst:L2},
\eqref{eq:mEst:N} and \eqref{eq:mEst:ptxError}. 

\subsection{Review of decomposability calculus}
Here we give a brief review of the notion of \emph{decomposable symbols} developed in \cite{MR2100060, Krieger:2005wh, Krieger:2012vj}, which provides 
a convenient way to keep track of mixed $L^{q}_{t} L^{r}_{x}$-type bounds. The particular version we use is from \cite{Krieger:2005wh, Krieger:2012vj}.

Given $\tht \in 2^{\bbZ_{-}}$, where $\bbZ_{-}$ denotes the set of nonpositive integers, consider a covering of the unit sphere $\bbS^{3} = \set{\xi: \abs{\xi} = 1} \subseteq \bbR^{4}$ by solid angular caps of the form $\set{\xi \in \bbS^{3} : \abs{\phi - \frac{\xi}{\abs{\xi}}} < \tht}$ with uniformly finite overlaps. We enumerate these caps by the centers $\phi \in \bbS^{3}$, and denote by $\set{b^{\phi}_{\tht}(\xi)}_{\phi}$ the associated smooth partition of unity on $\bbS^{3}$. 

Consider a smooth symbol $c(t, x; \xi)$ which is homogeneous of degree zero in $\xi$, i.e., depends only on the angular variable $\omg := \frac{\xi}{\abs{\xi}}$. We say that $c(t, x; \xi)$ is \emph{decomposable in $L^{q}_{t} L^{r}_{x}$} (where $1 \leq q, r \leq \infty$) if there exists an expansion $c = \sum_{\tht \in 2^{\bbZ_{-}}} c^{(\tht)}$ such that
\begin{equation} \label{eq:DLqLr}
	\sum_{\tht \in 2^{\bbZ_{-}}} \nrm{c^{(\tht)}}_{D_{\tht} L^{q}_{t} L^{r}_{x}} < \infty,
\end{equation} 
where 
\begin{equation} \label{eq:DthtLqLr}
	\nrm{c^{(\tht)}}_{D_{\tht} L^{q}_{t} L^{r}_{x}} 
	:= \nrm{\bb( \sum_{k=0}^{40} \sum_{\phi} \sup_{\omg} \nrm{b^{\phi}_{\tht}(\omg) \tht^{k} \rd_{\xi}^{(k)} c^{(\tht)}}_{L^{r}_{x}}^{2} \bb)^{\frac{1}{2}}}_{L^{q}_{t}}.
\end{equation}
We denote the class of such symbols by $D L^{q}_{t} L^{r}$. For $c \in D L^{q}_{t} L^{r}_{x}$, we define the norm $\nrm{c}_{D L^{q}_{t} L^{r}_{x}}$ by taking the infimum of \eqref{eq:DLqLr} over all possible decompositions $c = \sum_{\tht \in 2^{\bbZ_{-}}} c^{(\tht)}$.

The class $D L^{q}_{t} L^{r}_{x}$ provides a convenient framework for establishing $L^{q}_{t} L^{r}_{x}$-type estimates for pseudo-differential operators arising from products of symbols. The following lemma collects the key properties that we need.
\begin{lemma} \label{lem:decomp-key}
The following statements concerning the class $D L^{q}_{t} L^{r}_{x}$ hold.
\begin{enumerate}
\item For any symbols $c \in D L^{q_{1}} L^{r_{1}}$ and $c \in D L^{q_{2}} L^{r_{2}}$, its product obeys the H\"older-type bound
\begin{equation*}
	\nrm{cd}_{D L^{q}_{t} L^{r}_{x}} \aleq \nrm{c}_{D L^{q_{1}}_{t} L^{r_{1}}_{x}} \nrm{d}_{D L^{q_{2}}_{t} L^{r_{2}}_{x}}
\end{equation*}
where $1 \leq q_{1}, q_{2}, q, r_{1}, r_{2}, r \leq \infty$, $\frac{1}{q_{1}} + \frac{1}{q_{2}} = \frac{1}{q}$ and $\frac{1}{r_{1}} + \frac{1}{r_{2}} = \frac{1}{r}$.
\item Let $a(t, x; \xi)$ be a smooth symbol whose left quantization $a(t, x; D)$ satisfies the fixed time bound
\begin{equation*}
	\sup_{t} \nrm{a(t, x; D)}_{L^{2}_{x} \to L^{2}_{x}} \leq A.
\end{equation*}
Then for any symbol $c \in D L^{q}_{t} L^{r}_{x}$, we have the space-time bounds
\begin{equation*}
	\nrm{(ac)(t, x; D)}_{L^{q_{1}}_{t} L^{2}_{x} \to L^{q_{2}}_{t} L^{r_{2}}_{x}} \aleq A \nrm{c}_{D L^{q}_{t} L^{r}_{x}}
\end{equation*}
where $1 \leq q_{1}, q_{2}, q, r_{2}, r \leq \infty$, $\frac{1}{q_{1}} + \frac{1}{q} = \frac{1}{q_{2}}$ and $\frac{1}{2} + \frac{1}{r} = \frac{1}{r_{2}}$. 
An analogous statement holds in the case of right-quantization.
\end{enumerate}
\end{lemma}
For a proof, see \cite[Chapter 10]{Krieger:2005wh} and \cite[Lemma~7.1]{Krieger:2012vj}.

We borrow another lemma from \cite{Krieger:2012vj}, which relates the product of quantized operators with the product of the corresponding symbols within the framework of decomposable symbols.
\begin{lemma} \label{lem:decomp-prod}
Let $a(t, x; \xi)$, $b(t, x; \xi)$ be smooth symbols, where we assume furthermore that $a$ is homogeneous of degree zero in $\xi$. Then we have
\begin{equation*}
\begin{aligned}
& \hskip-2em
	\nrm{a(t, x; D) b(t, x; D) - (ab)(t, x; D)}_{L^{q_{0}}_{t} L^{2}_{x} \to L^{q}_{t} L^{2}_{x}} \\
	\aleq & \nrm{(\rd_{\xi} a)(t, x; D)}_{D L^{q_{2}}_{t} L^{\infty}_{x}} \nrm{(\rd_{x} b)(t, x; D)}_{L^{q_{0}}_{t} L^{2}_{x} \to L^{q_{1}}_{t} L^{2}_{x}}
\end{aligned}\end{equation*}
where $\frac{1}{q} = \frac{1}{q_{1}} + \frac{1}{q_{2}}$.
An analogous statement holds in the case of right-quantization.
\end{lemma}
For a proof, see \cite[Lemma~7.2]{Krieger:2012vj}.

\subsection{Symbol bounds for $\psi$}

We first consider the size and regularity of the dyadic pieces of $\psi_{k, \pm}$, namely 
\[
  \psi^{(\tht)}_{k, \pm} (t,x, \xi):= (\oPi_{\tht} \psi_{k, \pm})(t,x, \xi). 
\]
Given the symbol dependence on the angle, it is useful to keep 
in mind that the size of  $ \psi^{(\tht)}_{k, \pm} (t,x, \xi)$ is roughly given by
\[
\psi^{(\tht)}_{k, \pm} \approx 2^{-k} \tht^{-2} \oPi_{\tht} (\omg \cdot A_{k}).
\]

We borrow the following decomposability estimates for the symbol $\psi^{(\tht)}_{k, \pm}$ from \cite{Krieger:2012vj}:
\begin{lemma}[Decomposability estimates {\cite[Section~7.3]{Krieger:2012vj}}] 
For $\frac{2}{q} + \frac{3}{r} \leq \frac{3}{2}$, we have
\begin{align} 
	\nrm{(\psi^{(\tht)}_{k, \pm}, 2^{-k} \nb \psi^{(\tht)}_{k, \pm})}_{D L^{q}_{t} L^{r}_{x}}
	\aleq 2^{-(\frac{1}{q} + \frac{4}{r}) k} \tht^{\frac{1}{2} - \frac{2}{q} - \frac{3}{r}} E, \label{eq:symbBnd:decomp} 
\end{align}

Moreover, for any $\bt \geq 0$, we have
\begin{equation} 
	\nrm{\rd_{\xi}^{\bt} \oPi_{\tht} (\omg \cdot A_{k})}_{D L^{q}_{t} L^{r}_{x}}
	\aleq 2^{(1-\frac{1}{q} - \frac{4}{r}) k} \tht^{\frac{5}{2} - \frac{2}{q} - \frac{3}{r} - \bt} E. \label{eq:symbBnd:decomp:A} 
\end{equation}
\end{lemma}
In particular, for $q > 4$,
\begin{equation} \label{eq:symbBnd:decomp:summed}
	\nrm{\rd_{x}^{\alp-1} \nb \psi_{k, \pm}}_{D L^{q}_{t} L^{\infty}_{x}} \aleq 2^{(\alp - \frac{1}{q}) k} E.
\end{equation}

\begin{remark} 
From the decomposability bound \eqref{eq:symbBnd:decomp:summed} with $q=\infty$, \eqref{eq:mEst:DtPtx} follows easily.
\end{remark}

We also collect here additional symbol bounds which are cruder but
useful for estimating oscillatory kernels:
 
\begin{lemma} [Symbol bound for $\psi^{(\tht)}_{\pm}$ {\cite[Section~7.3]{Krieger:2012vj}}]  
The following symbol bounds hold.
\begin{enumerate}
\item For any $\alp, \bt \geq 0$ and $2 \leq q \leq \infty$ we have
\begin{align}
	\nrm{\rd^{\alp-1}_{x} \nb \rd^{\bt}_{\xi} \psi^{(\tht)}_{k, \pm}}_{L^{q}_{t} L^{\infty}_{x, \xi}} 
	& \aleq 2^{(-\frac{1}{q} + \alp) k} \tht^{\frac{1}{2} - \frac{2}{q} - \bt} E. \label{eq:symbBnd:simple} 
\end{align}
When $\alp = 0$, we interpret the expression on the left hand side as $\rd^{\bt}_{\xi} \psi^{(\tht)}_{k, \pm}$.

\item For $q > 4$ and $1 \leq \bt \leq \sgm^{-1}(1-\frac{1}{q})$, we have
\begin{equation} \label{eq:symbBnd:xiReg:Lq}
	\nrm{\rd^{\alp}_{x}\rd^{\bt}_{\xi} \psi_{k, \pm}}_{L^{q}_{t} L^{\infty}_{x, \xi}} 
	\aleq 2^{(-\frac{1}{q} + \alp) k} 2^{\sgm(\frac{1}{2} - \frac{2}{q} - \bt) k} E. 
\end{equation}

\item For $1 \leq \bt \leq \sgm^{-1}$, we have
\begin{equation}
	\abs{\rd^{\bt}_{\xi} (\psi_{\pm}(t,x, \xi) - \psi_{\pm}(t, y, \xi))}
	 \aleq \brk{x-y}^{\sgm(\bt-\frac{1}{2})} E . 	 \label{eq:symbBnd:xiReg}
\end{equation}

\end{enumerate}
\end{lemma}

%

\subsection{Fixed-time $L^{2}$ bounds}
Here we prove \eqref{eq:mEst:L2}. For later use, we prove the
following stronger result:
\begin{proposition} \label{prop:mEst:L2} For sufficiently small $\sgm
  > 0$, there exists $\dlt_{0} > 0$ such that the following is true:
  For every $\ell, k \leq 0$ with $\ell + C \leq k$, we have
  \begin{equation} \label{eq:mEst:strongL2} \nrm{(e^{- i \psi_{<\ell,
          \pm}}_{< k} (t,x, D) e^{i \psi_{< \ell, \pm}}_{<k} (D, y, t)
      - 1)P_{0}}_{L^{2}_{x} \to L^{2}_{x}}
    \aleq_{E} 2^{(1-\dlt_{0})
      \ell} + 2^{10(\ell - k)}.
  \end{equation}
  where the constant is independent of $k, \ell$.
\end{proposition}

We remind the reader that $e^{i \psi_{\pm}}_{<0} (D, y, s)= e^{i
  \psi_{\pm}}_{<0} (t, D, y)$, since the symbol is independent of
$\tau = \xi_{0}$. In particular, this pseudodifferential operator
makes sense on every fixed time slice. Note that \eqref{eq:mEst:L2}
follows by taking $k = 0$ and noting that $\psi_{\pm} = \psi_{<-m}$.

To begin the proof of Proposition \ref{prop:mEst:L2}, we prove a
closely related estimate which does not involve space-time
Littlewood-Paley projections for $e^{i \psi_{<\ell, \pm}}$.
\begin{lemma} \label{lem:prelimL2} Let $\ell \leq 0$ and $a(D)$ be a
  multiplier such that $a(\xi)$ is a smooth bump function adapted to
  $\set{ \abs{\xi} \aleq 1}$. Then we have
  \begin{align}
    \nrm{e^{- i \psi_{<\ell, \pm}} (t,x,D) a(D) e^{i \psi_{<\ell,
          \pm}}(D, y, t) - a(D)}_{L^{2}_{x} \to L^{2}_{x}} &
    \aleq_{E} 2^{(1-\dlt_{0})
      \ell}.  \label{eq:prelimL2}
  \end{align}

  Furthermore, for any $k \in \bbR$ we have
  \begin{align}
    \nrm{e^{- i \psi_{<\ell, \pm}} (t,x,D) P_{0}}_{L^{2}_{x} \to
      L^{2}_{x}}
    & \aleq_{E} \, 1,					\label{eq:prelimL2:bdd} \\
    \nrm{e^{- i \psi_{<\ell, \pm}}_{<k} (t,x,D) P_{0}}_{L^{2}_{x} \to
      L^{2}_{x}} & \aleq_{E} \,
    1.  \label{eq:prelimL2:bdd:<k}
  \end{align}
\end{lemma}

\begin{proof} 
  We first reduce \eqref{eq:prelimL2:bdd} and
  \eqref{eq:prelimL2:bdd:<k} to proving \eqref{eq:prelimL2}.  By a
  $TT^{\ast}$ argument, \eqref{eq:prelimL2:bdd} is equivalent to
  $L^{2}_{x}$ boundedness of $e^{- i \psi_{<\ell, \pm}} (t,x,D)
  P_{0}^{2} e^{i \psi_{<\ell, \pm}}(D, y, s)$, which follows from
  \eqref{eq:prelimL2} and the $L^{2}_{x}$ boundedness of $a(D) =
  P_{0}^{2}$. Next, note that
  \begin{equation*}
    e^{-i \psi_{<\ell, \pm}}_{<k}(t,x,\xi) = \int e^{-i \psi_{<\ell, \pm}}((t,x)-z,\xi) 2^{5k} m(2^{k} z ) \, \ud^{1+4} z
  \end{equation*}
  where $m(z)$ is the kernel for $S_{<0}$. As the hypotheses for
  \eqref{eq:prelimL2:bdd} is obviously invariant under translations,
  the left- (and also right-) quantization of each $e^{-i \psi_{<\ell,
      \pm}}((t,x)-z,\xi)$ obeys the same bound as
  \eqref{eq:prelimL2:bdd}. Therefore, by the rapid decay of
  $m(\cdot)$, \eqref{eq:prelimL2:bdd:<k} follows.

  The proof of \eqref{eq:prelimL2} is an easy consequence of
  non-stationary phase, thanks to the fact that $\sgm > 0$ can be
  taken arbitrarily small. The kernel of the operator in
  \eqref{eq:prelimL2} is given by
  \begin{align*}
    K_{1}(t,x, y)
    :=& \ C \int (e^{i (\psi_{<\ell, \pm}(t,x, \xi) - \psi_{<\ell, \pm}(t,y, \xi))} - 1 ) a(\xi) e^{i \xi \cdot (x-y)}\,  \ud^{4} \xi \\
    = & C \int \int_{0}^{1} \Psi_{\pm} e^{i \rho \Psi_{\pm}} (t,x,t,y,
    \xi) a(\xi) e^{i \xi \cdot (x-y)} \, \ud \rho \, \ud^{4} \xi
  \end{align*}
  where
  \begin{equation*}
    \Psi_{\pm} (t,x, s,y, \xi) := \psi_{<\ell, \pm}(t,x, \xi) - \psi_{< \ell, \pm}(s,y, \xi).
  \end{equation*}

  We divide into two cases, namely when $\abs{x-y} \aleq 2^{- \dlt
    \ell}$ and $\abs{x-y} \ageq 2^{- \dlt \ell}$.

  \pfstep{Case 1: $\abs{x-y} \aleq 2^{-\dlt \ell}$} In this case,
  using \eqref{eq:symbBnd:simple} with $q=\infty$ and $\abs{\alp} = 1$
  and $\bt =0$ for each frequency $\ell'$ and summing up in $\ell' <
  \ell$, we obtain
  \begin{equation*}
    \abs{\Psi_{\pm}(t,x,t,y, \xi)} \aleq 2^{(1-\dlt) \ell}  E.
  \end{equation*}
 
  Since $\Psi_{\pm}$ is real-valued and $\supp \, a \subseteq
  \set{\abs{\xi} \aleq 1}$, it easily follows that
  \begin{equation}
    \abs{K_{1}(t, x,y) } \aleq 2^{(1-\dlt) \ell} E
    \qquad \hbox{ for } \abs{x-y} \aleq 2^{-\dlt \ell}.
  \end{equation}

  \pfstep{Case 2: $\abs{x-y} \ageq 2^{- \dlt \ell}$} Here we integrate
  by parts in $\xi$ for $N$-times and use the bound
  \eqref{eq:symbBnd:xiReg}.  Then we obtain
  \begin{equation}
    \abs{K_{1}(t, x,y)} \aleq_{\sgm, N, E} \frac{1}{\abs{x-y}^{(1-\sgm) N + \frac{1}{2} \sgm}}
    \qquad \hbox{ for } \abs{x-y} \ageq 2^{- \dlt \ell}.
  \end{equation}

  Combining Cases 1 and 2, it follows that
  \begin{equation*}
    \sup_{x} \int \abs{K_{1}(t, x,y)} \, \ud^{4} y +  \sup_{y} \int \abs{K_{1}(t, x,y)} \, \ud^{4} x \aleq_{E} 2^{(1-\dlt_{0}) \ell}
  \end{equation*}
  if $\sgm$, $\dlt$ are small enough and $N$ is sufficiently
  large. Estimate \eqref{eq:prelimL2} now follows. \qedhere
\end{proof}

Next, we borrow a lemma from \cite{Krieger:2012vj}, which is useful
for handling $e^{i \psi_{< \ell, \pm}}_{k}$ when $k > \ell$.
\begin{lemma} \label{lem:k>l} For $\ell + C \leq k$ and every $t \in
  \bbR$, we have
  \begin{equation} \label{eq:k>l:L2} \nrm{e^{- i \psi_{< \ell,
          \pm}}_{k}(t,x, D) P_{0} }_{L^{2}_{x} \to L^{2}_{x}}
    \aleq_{E} 2^{10(\ell-k)}.
  \end{equation}

  Furthermore, for $1 \leq q \leq p \leq \infty$ and $\ell + C \leq
  k$, we have
  \begin{equation} \label{eq:k>l:spt} \nrm{e^{- i \psi_{< \ell,
          \pm}}_{k}(t,x, D) P_{0} }_{L^{p}_{t}L^{2}_{x} \to L^{q}_{t}
      L^{2}_{x}} \aleq_{E}
    2^{(\frac{1}{p} - \frac{1}{q}) \ell} 2^{10(\ell-k)}.
  \end{equation}

  These estimates also hold for $e^{-i \psi_{<\ell, \pm}}_{k}(D, y,
  s)$.
\end{lemma}

\begin{remark} 
  The specific factor $10$ in the gain $2^{10(\ell - k)}$ is
  irrelevant, but it is important to note that this number is much
  bigger than $1$. This will be very useful in our proof of
  \eqref{eq:mEst:N}, where we will use this factor to dominate smaller
  factors. In fact, a variant of the proof below allows us to make
  this gain as large as we want, by making the implicit constant
  larger.
\end{remark}
\begin{proof}
  Consider frequency projections $S^{(1)}_{k}, \ldots, S^{(6)}_{k}$,
  $\widetilde{S}^{(1)}_{k}, \ldots, \widetilde{S}^{(5)}_{k}$, which
  obey the same bounds as $S_{k}$ and furthermore satisfy
  \begin{equation*}
    S^{(1)}_{k} := S_{k}, \quad
    S^{(i)}_{k} = 2^{-2k} \widetilde{S}^{(i)}_{k} (\rd_{t}^{2} + \lap), \quad
    \widetilde{S}^{(i)}_{k} = \widetilde{S}^{(i)}_{k} S^{(i+1)}_{k}
  \end{equation*}
  for $i=1,\ldots,5$.  Thanks to the assumption $\ell + C \leq k$, we
  may write at the level of symbols
  \begin{align*}
    e^{- i \psi_{<\ell, \pm}}_{k}
    =& 2^{-2k} \widetilde{S}^{(1)}_{k} (\rd_{t}^{2} + \lap) e^{- i \psi_{<\ell, \pm}} \\
    =& 2^{-2k} \widetilde{S}^{(1)}_{k} (-2 i \lap \psi_{<\ell, \pm}  - \abs{\nb \psi_{<\ell, \pm}}^{2}) S^{(2)}_{k} e^{- i \psi_{<\ell, \pm}} \\
    =& \cdots = 2^{-10k} \prod_{j=1}^{5} \bb[ \widetilde{S}^{(j)}_{k}
    (-2 i \lap \psi_{<\ell, \pm} - \abs{\nb \psi_{<\ell,
        \pm}}^{2}) \bb] e^{- i \psi_{<\ell, \pm}}
  \end{align*}

  Here we used the fact that $\psi_{<\ell, \pm}(t,x,\xi)$ solves the
  free wave equation $\rd_{t}^{2} \psi_{<\ell, \pm} = \lap
  \psi_{<\ell, \pm}$ for each $\xi$, since $A$ does. Disposing of the
  nested projections $\widetilde{S}^{(j)}_{k}$ by translation
  invariance, using the decomposability bound
  \eqref{eq:symbBnd:decomp:summed} and $L^{2}_{x}$ boundedness of
  $e^{-i \psi_{<\ell, \pm}} (t,x,D) P_{0}$, the desired estimate
  follows. \qedhere
\end{proof}

We are now ready to prove Proposition \ref{prop:mEst:L2}.
\begin{proof} [Proof of Proposition \ref{prop:mEst:L2}]
  Thanks to the frequency localization of the symbol $e^{i
    \psi_{<\ell, \pm}}_{<k}(s,y,\xi)$, note that we can harmlessly put
  in a multipler $a(D)$ whose symbol is a smooth bump function adapted
  to $\set{\abs{\xi} \aleq 1}$. The operator in
  \eqref{eq:mEst:strongL2} therefore equals
  \begin{equation*}
    ( e^{- i \psi_{<\ell, \pm}}_{<0} (t, x, D) a(D) e^{i \psi_{< \ell, \pm}}_{<0} (D, y, s) - a(D) ) P_{0}.
  \end{equation*}

  For the purpose of proving \eqref{eq:mEst:strongL2}, we can safely
  dispose $P_{0}$ on the right. Next, note that
  \begin{align*}
    & e^{-i \psi_{<\ell, \pm}}(t,x,D) a(D) e^{i \psi_{<\ell, \pm}}(D,
    y, s)
    - e^{-i \psi_{<\ell, \pm}}_{<k}(t,x,D) a(D) e^{i \psi_{<\ell, \pm}}_{<k}(D, y, s) \\
    & \quad = e^{-i \psi_{<\ell, \pm}}_{\geq k}(t,x,D) a(D) e^{i
      \psi_{<\ell, \pm}}(D, y, s) + e^{-i \psi_{<\ell, \pm}}_{<
      k}(t,x,D) a(D) e^{i \psi_{<\ell, \pm}}_{\geq k}(D, y, s) .
  \end{align*}

  By Lemma \ref{lem:k>l} and \eqref{eq:prelimL2:bdd}, the operators on
  the right hand side obey
  \begin{align*}
    \nrm{e^{-i \psi_{<\ell, \pm}}_{\geq k}(t,x,D) a(D) e^{i
        \psi_{<\ell, \pm}}(D, y, s)}_{L^{2}_{x} \to L^{2}_{x}}
    \aleq_{E} & \, 2^{10(\ell-k)}, \\
    \nrm{e^{-i \psi_{<\ell, \pm}}_{< k}(t,x,D) a(D) e^{i \psi_{<\ell,
          \pm}}_{\geq k}(D, y, s)}_{L^{2}_{x} \to L^{2}_{x}} \aleq_{E} & \,
    2^{10(\ell - k)}.
  \end{align*}

  Combining these bounds with \eqref{eq:prelimL2} established in Lemma
  \ref{lem:prelimL2}, \eqref{eq:mEst:strongL2} follows. \qedhere
\end{proof}

\subsection{Space-time bounds}
Here we establish \eqref{eq:mEst:N}.  More precisely, we will show
that:

\begin{proposition} \label{prop:Xsb} For $\sgm > 0$ sufficiently
  small, there exists $\dlt_{1} > 0$ such that the following holds:
  For $j \leq C$, we have
  \begin{equation} \label{eq:Xsb} \nrm{Q_{j} [e^{- i \psi_{\pm}}_{<0}
      (t,x,D) e^{i \psi_{\pm}}_{<0}(D, y, s) -1] P_{0}
      Q_{<0}}_{N^{\ast} \to X^{0, 1/2}_{\infty}}
    \aleq_{E} 2^{-\dlt_{1} m}.
  \end{equation}
\end{proposition}

The estimate \eqref{eq:Xsb} proves the $X^{0,1/2}_{\infty}$ part of
\eqref{eq:mEst:N}. Note that the $L^{\infty}_{t} L^{2}_{x}$ portion of
\eqref{eq:mEst:N} follows immediately from \eqref{eq:mEst:L2}.

To ease the notation, we omit writing $\pm$ in $\psi_{\pm}$. Also, we
omit the dependence of the constants on $E$. It will be convenient to define the compound symbols
\begin{align*}
  \Psi(t,x,s,y, \xi) :=& \psi (t,x,\xi) - \psi(s,y,\xi), \\
  \Psi_{< \ell}(t,x,s,y, \xi) :=& \psi_{<\ell} (t,x,\xi) -
  \psi_{<\ell}(s,y,\xi).
\end{align*}
The symbol $\Psi_{\ell}$ is defined in the obvious way.

Given a compound symbol $a(t,x,s,y,\xi)$, we define the double
space-time frequency projection
\begin{align*}
  a_{\dless k} (t,x,s,y,\xi) := S^{t,x}_{<k} S^{s,y}_{<k} a
  (t,x,s,y,\xi),
\end{align*}
where $S^{t,x}_{<k}$ is the space-time frequency projection applied to
$(t,x)$, etc.  Therefore, according to our conventions,
\begin{equation*}
  e^{- i \Psi}_{\dless k}(t,x,D,y,s) = e^{- i \psi}_{<k} (t,x,D) e^{i \psi}_{<k}(D, y, s).
\end{equation*}

We begin with a lemma for frequency localizing the gauge transform
$e^{-i \Psi_{<\ell}}$, which will be used several times in our
argument.
\begin{lemma} \label{lem:k>lPsi} For $2 \leq q \leq \infty$ and $\ell
  + C\leq k \leq 0$, we have
  \begin{align}
    \nrm{(e^{- i \Psi_{<\ell}}_{\dless C} - e^{-i
        \Psi_{<\ell}}_{\dless k}) P_{0} }_{L^{p}_{t} L^{2}_{x} \to
      L^{q}_{t} L^{2}_{x}} \aleq 2^{(\frac{1}{p} - \frac{1}{q}) \ell}
    2^{ 10 (\ell-k)}.  \label{eq:k>lPsi:<0}
  \end{align}
\end{lemma}

\begin{proof} 
  Proceeding as in the last part of the proof of Proposition
  \ref{prop:mEst:L2}, we may write
  \begin{align*}
    & (e^{- i \Psi_{<\ell}}_{\dless C} - e^{-i \Psi_{<\ell}}_{\dless k}) P_{0} \\
    & \quad = [e^{- i \psi_{<\ell}}_{< C} (t,x, D) a(D) e^{i
      \psi_{<\ell}}_{< C} (D, y, s)
    - e^{- i \psi_{<\ell}}_{< k} (t,x, D) a(D) e^{i \psi_{<\ell}}_{< k} (D, y, s) ] P_{0} \\
    & \quad = [e^{- i \psi_{<\ell}}_{k \leq \cdot < C} (t,x, D) a(D)
    e^{i \psi_{<\ell}}_{< C} (D, y, s) + e^{- i \psi_{<\ell}}_{< k}
    (t,x, D) a(D) e^{i \psi_{<\ell}}_{k \leq \cdot < C} (D, y, s)]
    P_{0}
  \end{align*}
  where $a(\xi)$ is a smooth bump function adapted to $\set{\abs{\xi}
    \aleq 1}$. Then \eqref{eq:k>lPsi:<0} follows from Lemma
  \ref{lem:k>l}. \qedhere
\end{proof}

We are now ready to prove Proposition \ref{prop:Xsb}.
\begin{proof} [Proof of Proposition \ref{prop:Xsb}]
  We proceed in several steps. Let $\dlt > 0$ be a small number to be
  determined later.

  \pfstep{Step 1: High modulation input} For $j' \geq j - C$, we claim
  that
  \begin{equation} \label{eq:Xsb:highMod} \nrm{Q_{j} [e^{- i
        \Psi}_{\dless 0}(t,x, D, y, s) -1] P_{0} Q_{j'}}_{N^{\ast}
      \to X^{0, 1/2}_{\infty}} \aleq 2^{-\dlt_{1} m} 2^{\frac{1}{2} (j
      - j')}.
  \end{equation}

  Using the $X^{0,1/2}_{\infty}$ portion of $N^{\ast}_{0}$,
  \eqref{eq:Xsb:highMod} follows from
  \begin{equation*}
    \nrm{Q_{j} [e^{- i \Psi}_{\dless 0} -1] P_{0} Q_{j'}}_{L^{2}_{t,x} \to L^{2}_{t,x}} 
    \aleq 2^{-(1-\dlt_{0}) m} .
  \end{equation*}
  Since $Q_{j}, Q_{j'}$ are easily disposable, this estimate follows
  easily from \eqref{eq:mEst:L2}.

  \pfstep{Step 2: Low modulation input, $-\frac{1}{2} m \leq j \leq
    C$} In this step, we take care of the easy case $-\frac{1}{2} m
  \leq j \leq C$. Under this assumption, we claim that
  \begin{equation}
    \nrm{Q_{j} [e^{- i \Psi}_{\dless 0} (t,x,D, y, s) - 1]P_{0} Q_{< j - C}}_{N^{\ast} \to X^{0, 1/2}_{\infty}} 
    \aleq 2^{-\dlt_{1} m} .
  \end{equation}

  Note that
  \begin{equation*}
    Q_{j} [ e^{- i \Psi}_{\dless j-C} - 1] P_{0} Q_{< j-C} = 0
  \end{equation*}
  by modulation localization. Using the $L^{\infty}_{t} L^{2}_{x}$
  portion of $N^{\ast}_{0}$, it suffices to prove
  \begin{equation*}
    \nrm{Q_{j} [e^{- i \Psi}_{\dless 0}  - e^{ - i \Psi}_{\dless j-C}]P_{0} Q_{< j - C}}_{L^{\infty}_{t} L^{2}_{x} \to L^{2}_{t,x}} 
    \aleq 2^{- 4 m} 2^{-\frac{1}{2} j}.
  \end{equation*}
  Since $Q_{j}$, $Q_{<j-C}$ are disposable on $L^{2}_{t,x}$ and
  $L^{\infty}_{t} L^{2}_{x}$, this estimate follows from
  \eqref{eq:k>lPsi:<0} and the fact that $\Psi = \Psi_{<-m}$.

  \pfstep{Step 3: Low modulation input, $ j < -\frac{1}{2} m$, main
    decomposition} Henceforth, we consider the case $j < - \frac{1}{2}
  m$. The goal of Steps 3--6 is to establish
  \begin{equation} \label{eq:Xsb:lowMod} \nrm{Q_{j} [e^{- i
        \Psi}_{\dless 0} (t,x,D, y, s) - e^{- i \Psi_{< j - \dlt
          m}}_{\dless 0} (t,x,D,y,s)] P_{0} Q_{< j - C}}_{N^{\ast}
      \to X^{0, 1/2}_{\infty}} \aleq 2^{-\dlt_{1} m} .
  \end{equation}

  At the level of symbols, we begin by writing
  \begin{align*}
    e^{- i \Psi} - e^{- i \Psi_{<j - \dlt m}}
    = & - i \int_{\ell \geq j - \dlt m} \Psi_{\ell} \, e^{- i \Psi_{<j
        - \dlt m}}\, \ud \ell
    -\iint_{\ell \geq \ell' \geq j - \dlt m} \Psi_{\ell} \Psi_{\ell'} \, e^{- i \Psi_{<j - \dlt m}}\, \ud \ell' \ud \ell \\
    & 	+ i \iiint_{\ell \geq \ell' \geq \ell'' \geq j - \dlt m} \Psi_{\ell} \Psi_{\ell'} \Psi_{\ell''} \, e^{- i \Psi_{<\ell''}}\, \ud \ell'' \ud \ell' \ud \ell \\
    = & \!\!: \calL + \calQ + \calC.
  \end{align*}

  We treat $\calL$, $\calQ$ and $\calC$ in Steps 4, 5, and 6,
  respectively.

  \pfstep{Step 4: Low modulation input, $ j < -\frac{1}{2} m$,
    contribution of $\calL$}
  In this step, we prove
  \begin{equation} \label{eq:Xsb:Lest} \nrm{Q_{j} \calL_{\ll
        0}(t,x,D,y,s) P_{0} Q_{<j-C}}_{N^{\ast} \to X^{0,
        1/2}_{\infty}} \aleq 2^{- \dlt_{1} m}.
  \end{equation}

  We further decompose $\calL$ as follows. We first separate out the
  low frequency part of the gauge transform, then decompose according
  to the frequency of $\Psi_{\ell}$ (depending on whether $\ell$ is
  higher or comparable to $j$), and finally replace the gauge
  transform by $1$:
  \begin{align}
    \calL
    = &- i \int_{\ell \geq j - \dlt m} \Psi_{\ell} \, (e^{- i \Psi_{<j - \dlt m}} - e^{- i \Psi_{<j - \dlt m}}_{\dless j - C} )\, \ud \ell \label{eq:Xsb:L1} \\
    &- i \int_{\ell \geq j + 10 \dlt m} \Psi_{\ell} \, e^{- i \Psi_{<j - \dlt m}}_{\dless j - C}\, \ud \ell \label{eq:Xsb:L2}\\
    & - i \int_{j - \dlt m \leq \ell \leq j + 10 \dlt m} \Psi_{\ell} \, (e^{- i \Psi_{<j - \dlt m}}_{\dless j - C} - 1)\, \ud \ell  \label{eq:Xsb:L3}\\
    & - i \int_{j - \dlt m \leq \ell \leq j + 10 \dlt m} \Psi_{\ell} \, \ud \ell  \label{eq:Xsb:L4} \\
    = & \!\!: \calL_{1} + \calL_{2} + \calL_{3} + \calL_{4}. \notag
  \end{align}

  We treat the contribution of $\calL_{1}, \ldots, \calL_{4}$
  separately.

  \pfstep{Step 4.1: Contribution of $\calL_{1}$} For $\calL_{1}$, the
  double frequency localization $(\calL_{1})_{\dless 0}$ and the fact
  that $\ell < -m < 0$ allow us to write (at the level of symbols)
  \begin{equation*}
    (\calL_{1})_{\dless 0} 
    = \bb( - i \int_{\ell \geq j - \dlt m} \Psi_{\ell} \, (e^{- i \Psi_{<j - \dlt m}}_{\dless C} - e^{- i \Psi_{<j - \dlt m}}_{\dless j - C} )\, \ud \ell \bb)_{\dless 0}.
  \end{equation*}

  As the rest of the argument for $\calL_{1}$ will be translation
  invariant, we can easily dispose the double frequency localization
  $(\cdot)_{\dless \ell}$. We are now reduced to proving
  \begin{equation*}
    \nrm{\Psi_{\ell} \, (e^{- i \Psi_{<j - \dlt m}}_{\dless C} - e^{- i \Psi_{<j - \dlt m}}_{\dless j - C} ) P_{0} }_{L^{\infty}_{t} L^{2}_{x} \to L^{2}_{t,x}} \aleq 2^{-\frac{1}{2}j} 2^{-(10 - \frac{1}{2}) \dlt m} 2^{\frac{1}{6} (j - \dlt m -\ell)}
  \end{equation*}
  for $\ell \geq j - \dlt m$. This estimate follows from the
  decomposability bound \eqref{eq:symbBnd:decomp:summed} with $q=6$
  and \eqref{eq:k>lPsi:<0} with $(p,q)=(\infty,\frac{1}{3})$.

  \pfstep{Step 4.2: Contribution of $\calL_{2}$} Before we begin, note
  that the double frequency localization $(\cdot)_{\dless 0}$ does
  nothing to $\calL_{2}$, $\calL_{3}$ and $\calL_{4}$, thanks to their
  frequency localization properties. Therefore, we drop
  $(\cdot)_{\dless 0}$ from now on.

  In this step, the desired gain in $m$ will be obtained from $\ell
  \geq j + 10 \dlt m$, and we do not exploit the difference structure
  in $\Psi_{\ell}$. In fact, in order to apply decomposability bounds,
  we divide $\Psi_{\ell} (t,x,s,y, \xi) = \psi_{\ell}(t,x,\xi) -
  \psi_{\ell}(s,y,\xi)$ and treat each term separately. Here we only
  consider the case $\psi_{\ell}(t,x,\xi)$; the argument for the other
  case is analogous.

  Thanks to the frequency localization $e^{i \psi_{<j-\dlt
      m}}_{<j-C}$, the contribution of $\psi_{\ell}(t,x,\xi)$ in the
  integrand in \eqref{eq:Xsb:L2} equals
  \begin{equation*}
    Q_{j} (\psi_{\ell} e^{-i \psi_{<j - \dlt m}}_{<j-C})(t,x,D) \widetilde{P}_{0} \widetilde{Q}_{<j-C}
    e^{i \psi_{j-\dlt m}}_{<j-C}(D, y, s) P_{0} Q_{<j-C}
  \end{equation*}
  where $\widetilde{P}_{0} \widetilde{Q}_{<j-C}$ is a slightly
  enlarged version of $P_{0} Q_{<j-C}$. Then by the frequency
  localization of $e^{i \psi_{<j-\dlt m}}_{<j-C}$ (now applied to the
  one on the left), the modulation of the output forces an angular
  separation between the spatial frequency of $\psi_{\ell}(t,x,\xi)$
  and $\xi$ of the size $\aeq
  2^{-\frac{1}{2}(\ell-j)_{+}}$. Therefore, the preceding operator
  equals
  \begin{equation*}
    Q_{j} (\oPi_{>2^{-\frac{1}{2}(\ell-j)_{+}}}\psi_{\ell} e^{-i \psi_{<j - \dlt m}}_{<j-C})(t,x,D) 
    e^{i \psi_{<j-\dlt m}}_{<j-C}(D, y, s) P_{0} Q_{<j-C},
  \end{equation*}
  where we dropped $\widetilde{P}_{0} \widetilde{Q}_{<j-C}$ as it is
  of no more use.  Using fixed-time $L^{2}_{x}$ boundedness of $e^{i
    \psi_{<j-\dlt m}}_{<j-C} P_{0}$ and the decomposability bound
  \eqref{eq:symbBnd:decomp} summed over $\tht \ageq
  2^{-\frac{1}{2}(\ell-j)_{+}}$, it follows that
  \begin{equation} \label{eq:Xsb:L2:pf}
    \begin{aligned} 
      & \nrm{Q_{j} (\oPi_{>2^{-\frac{1}{2}(\ell-j)_{+}}}\psi_{\ell}
        e^{-i \psi_{<j - \dlt m}}_{<j-C})(t,x,D)
	e^{i \psi_{<j-\dlt m}}_{<j-C}(D, y, s) P_{0} Q_{<j-C}}_{L^{\infty}_{t} L^{2}_{x} \to L^{2}_{t,x}} \\
      & \quad \aleq 2^{-\frac{1}{2} j} 2^{\frac{1}{4}(j - \ell)}
    \end{aligned}
  \end{equation}
  for $\ell \geq j$. Now integrating this bound over $\ell \geq j+ 10
  \dlt m$, we obtain a gain of $2^{-\frac{1}{4} \dlt m}$ from the
  factor $2^{\frac{1}{4}(j - \ell)}$, which is acceptable.

  \pfstep{Step 4.3: Contribution of $\calL_{3}$} For $\calL_{3}$, we
  make use of the difference structure in the phase $\Psi_{<j - \dlt
    m}$, but not for the symbol $\Psi_{\ell}$. Thus we again only
  consider $\psi_{\ell}(t,x,\xi)$. We remind the reader that in this
  case, the $\ell$-integral in \eqref{eq:Xsb:L3} is taken over $j -
  \dlt m \leq \ell \leq j + 10 \dlt$.

  Proceeding as in Step 4.2, the contribution of this term in the
  integrand for \eqref{eq:Xsb:L3} equals
  \begin{equation*}
    \widetilde{\calL}_{3}(\ell) := Q_{j} (\oPi_{>2^{-\frac{1}{2}(\ell-j)_{+}}} \psi_{\ell} (e^{-i \Psi_{<j - \dlt m}}_{\dless j-C} - 1) )(t,x,D,y,s) P_{0} Q_{<j-C} \, .
  \end{equation*}

  Then proceeding as in the proof of \eqref{eq:Xsb:L2:pf}, but using
  \eqref{eq:mEst:strongL2} instead of mere boundedness of $e^{-i
    \Psi_{<j - \dlt m}}_{\dless j-C}$, and integrating over $j - \dlt
  m \leq \ell \leq j + 10 \dlt$, we obtain
  \begin{equation*}
    \nrm{\int_{j - \dlt m \leq \ell \leq j + 10 \dlt m} \widetilde{\calL}_{3}(\ell) \, \ud \ell}_{L^{\infty}_{t} L^{2}_{x} \to L^{2}_{t,x}} 
    \aleq 2^{-\frac{1}{2} j} 2^{\frac{1}{2} \dlt m} (2^{(1-\dlt_{0}) (j-\dlt m)} + 2^{- 10 \dlt m}).
  \end{equation*}
  This bound is good if $\dlt > 0$ is sufficiently small, since $j <
  -\frac{1}{2} m$.

  \pfstep{Step 4.4: Contribution of $\calL_{4}$} As in the previous
  step, the $\ell$-integral in \eqref{eq:Xsb:L4} is taken over $j -
  \dlt m \leq \ell \leq j + 10 \dlt$. Here we make use of the
  difference structure of the symbol $\Psi_{\ell}$. The modulation
  localization properties again allow us to write
  \begin{equation*}
    Q_{j} \Psi_{\ell}(t,x,D,y,s) P_{0} Q_{<j-C} 
    = \sum_{\tht \ageq 2^{-\frac{1}{2}(\ell-j)_{+}} } Q_{j} (\psi^{(\tht)}_{\ell}(t,x,D) - \psi^{(\tht)}_{\ell}(D,y,s)) P_{0} Q_{<j-C} \, .
  \end{equation*}

  As usual, we can harmlessly put in an operator $a(D)$ which is a
  slightly enlarged version of $P_{0}$. We now claim that for $2 \leq
  q \leq \infty$, the following bound holds:
  \begin{equation} \label{eq:Xsb:L4:pf}
    \nrm{\psi^{(\tht)}_{\ell}(t,x,D)a(D) - a(D)
      \psi^{(\tht)}_{\ell}(D, y, s)}_{L^{\infty}_{t} L^{2}_{x} \to
      L^{q}_{t} L^{2}_{x}} \aleq 2^{-\frac{1}{q} \ell} 2^{\ell}
    \tht^{-C}
  \end{equation}

  To prove this bound, we compute the kernel of
  $\psi^{(\tht)}_{\ell}(t,x,D) a(D) - a(D)
  \psi^{(\tht)}_{\ell}(D,y,s)$.
  \begin{align*}
    K_{2}(t,x,y)
    = & C \int (\psi^{(\tht)}_{\ell}(t,x,\xi) - \psi^{(\tht)}_{\ell}(t, y, \xi)) a(\xi) e^{i \xi \cdot (x-y)}\, \ud^{4} \xi \\
    = & C \int \int_{0}^{1} (x-y) \cdot (\rd_{x} \psi^{(\tht)}_{\ell})(t, \rho x + (1-\rho) y, \xi)  a(\xi) e^{i \xi \cdot (x-y)} \, \ud \rho \, \ud^{4} \xi \\
    = & C \int \int_{0}^{1} \rd_{\xi} \cdot [(\rd_{x}
    \psi^{(\tht)}_{\ell})(t, \rho x + (1-\rho) y, \xi) a(\xi)] e^{i
      \xi \cdot (x-y)} \, \ud \rho \, \ud^{4} \xi.
  \end{align*}

  Integrating by parts in $\xi$ several times and using the symbol
  bound \eqref{eq:symbBnd:simple}, we obtain a kernel bound which
  implies \eqref{eq:Xsb:L4:pf}.

  Applying \eqref{eq:Xsb:L4:pf} with $q=2$, it follows that
  \begin{equation*}
    \nrm{Q_{j} (\psi^{(\tht)}_{\ell}(t,x,D) - \psi^{(\tht)}_{\ell}(D,y,s)) P_{0} Q_{<j-C}}_{L^{\infty}_{t} L^{2}_{x} \to L^{2}_{t,x}}
    \aleq 2^{-\frac{1}{2} \ell} 2^{\ell} \tht^{-C}.
  \end{equation*}
  Summing over $\tht \ageq 2^{-\frac{1}{2}(\ell-j)_{+}}$ and
  integrating over $j - \dlt m \leq \ell \leq j + \dlt m$, we arrive
  at
  \begin{equation*}
    \nrm{\int_{j - \dlt m \leq \ell \leq j + 10 \dlt m} Q_{j} \Psi_{\ell}(t,x,D,y,s) P_{0} Q_{<j-C} \, \ud \ell}_{L^{\infty}_{t} L^{2}_{x} \to L^{2}_{t,x}}
    \aleq 2^{-\frac{1}{2} j} 2^{j} 2^{C \dlt m}.
  \end{equation*}
  This is good for sufficiently small $\dlt > 0$, as $j <
  -\frac{1}{2}m$. This completes the proof of \eqref{eq:Xsb:Lest}.

  \pfstep{Step 5: Low modulation input, $ j < -\frac{1}{2} m$,
    contribution of $\calQ$} Here we prove
  \begin{equation} \label{eq:Xsb:Qest} \nrm{Q_{j} \calQ_{\ll
        0}(t,x,D,y,s) P_{0} Q_{<j-C}}_{N^{\ast} \to X^{0,
        1/2}_{\infty}} \aleq 2^{- \dlt_{1} m}.
  \end{equation}

  As in Step 4, we begin by further decomposing $\calQ$:
  \begin{align}
    \calQ
    =& -\iint_{\ell \geq \ell' \geq j - \dlt m} \Psi_{\ell}
    \Psi_{\ell'} \,
    (e^{- i \Psi_{<j - \dlt m}} - e^{- i \Psi_{< j - \dlt m}}_{\dless j - C})\, \ud \ell' \ud \ell \label{eq:Xsb:Q1} \\
    & -\iint_{\substack{\ell \geq \ell' \geq j - \dlt m \\ \ell \geq j
        + 10 \dlt m}} \Psi_{\ell} \Psi_{\ell'} \,
    e^{- i \Psi_{< j - \dlt m}}_{\dless j - C}\, \ud \ell' \ud \ell \label{eq:Xsb:Q2} \\
    & -\iint_{j - \dlt m \leq \ell' \leq \ell \leq j + 10 \dlt m}
    \Psi_{\ell} \Psi_{\ell'} \,
    (e^{- i \Psi_{< j - \dlt m}}_{\dless j - C} - 1) \, \ud \ell' \ud \ell  \label{eq:Xsb:Q3} \\
    & 	-\iint_{j - \dlt m \leq \ell' \leq \ell \leq j + 10 \dlt m} \Psi_{\ell} \Psi_{\ell'} \, \ud \ell' \ud \ell \label{eq:Xsb:Q4} \\
    = & \!\!: \calQ_{1} + \calQ_{2} + \calQ_{3} + \calQ_{4}. \notag
  \end{align}

  We treat each of these terms below.  \pfstep{Step 5.1: Contribution
    of $\calQ_{1}$} Proceeding as in Step 4.1, we have
  \begin{equation*}
    (\calQ_{1})_{\dless 0} 
    = \bb( - \iint_{\ell \geq \ell' \geq j - \dlt m} \Psi_{\ell} \Psi_{\ell'} 
    (e^{- i \Psi_{<j - \dlt m}}_{\dless C} - e^{-i \Psi_{<j - \dlt m}}_{\dless j-C}) \, \ud \ell' \ud \ell \bb)_{\dless 0}
  \end{equation*}
  and the outer $(\cdot)_{\dless 0}$ can be disposed by translation
  invariance as before. Next, by \eqref{eq:symbBnd:decomp:summed}
  (with $q = 6$ for $\Psi_{\ell}$, $q=\infty$ for $\Psi_{\ell'}$) and
  \eqref{eq:k>lPsi:<0} with $(p,q) = (\infty, 3)$, we have
  \begin{equation*}
    \nrm{\Psi_{\ell} \Psi_{\ell'} (e^{- i \Psi_{<j - \dlt m}}_{\dless C} - e^{-i \Psi_{<j - \dlt m}}_{\dless j-C})}_{L^{\infty}_{t} L^{2}_{x} \to L^{2}_{t,x}} 
    \aleq 2^{-\frac{1}{2} j} 2^{-(10 - \frac{1}{2}) \dlt m} 2^{-\frac{1}{6}(j-\dlt m - \ell)}.
  \end{equation*}
  Integrating over $\ell \geq \ell' \geq j- \dlt m$, we see the
  desired gain of $2^{-(10 - \frac{1}{2}) \dlt m}$.

  \pfstep{Step 5.2: Contribution of $\calQ_{2}$} As in Steps 4.2, 4.3
  and 4.4, $(\cdot)_{\dless 0}$ does nothing to $\calQ_{2}, \calQ_{3},
  \calQ_{4}$, and therefore can be removed.  Also, in this step we
  split $\Psi_{\ell}(t,x,s,y,\xi) = \psi_{\ell}(t,x,\xi) -
  \psi_{\ell}(s,y,\xi)$ and handle only the contribution of
  $\psi_{\ell}(t,x,\xi) \psi_{\ell'}(t,x,\xi)$, as the argument for
  the other parts is the same.

  As in Step 4.2, the contribution of $\psi_{\ell}(t,x,\xi)
  \psi_{\ell'}(t,x,\xi)$ in the integrand in \eqref{eq:Xsb:Q2} equals
  \begin{equation*}
    Q_{j} (\psi_{\ell} \psi_{\ell'} e^{- i \psi_{<j - \dlt m}}_{< j - C})(t,x,D) \widetilde{P}_{0} \widetilde{Q}_{<j-C} e^{i \psi_{<j - \dlt m}}_{< j - C} (D, y, s) P_{0} Q_{<j-C} \, .
  \end{equation*}
  We first split
  \begin{equation*}
    \psi_{\ell'} (t,x,\xi)= \oPi_{> 2^{-\frac{1}{2}(\ell'-j)_{+}-C'}} \psi_{\ell'} (t,x,\xi) + \oPi_{\leq 2^{-\frac{1}{2}(\ell'-j)_{+}-C'}} \psi_{\ell'} (t,x,\xi).
  \end{equation*}
  The first term is good, as we already see an angular separation. For
  the contribution of the second term, we can apply an argument
  similar to Step 4.2 to conclude that there is an angular separation
  between the spatial frequency of $\psi_{\ell}(t,x,\xi)$ and $\xi$ of
  size $\aeq 2^{-\frac{1}{2}(\ell-j)_{+}}$. Therefore, the preceding
  operator equals
  \begin{align*}
    &	Q_{j} (\psi_{\ell} \oPi_{> 2^{- \frac{1}{2}(\ell'-j)_{+}-C'}} \psi_{\ell'} e^{- i \psi_{<j - \dlt m}}_{< j - C})(t,x,D) e^{i \psi_{<j - \dlt m}}_{< j - C} (D, y, s) P_{0} Q_{<j-C} \\
    &\quad +
    Q_{j} ( \oPi_{\ageq 2^{-\frac{1}{2}(\ell - j)_{+}}} \psi_{\ell} \, \oPi_{\leq 2^{- \frac{1}{2}(\ell'-j)_{+} - C'}} \psi_{\ell'} e^{- i \psi_{<j - \dlt m}}_{< j - C})(t,x,D) e^{i \psi_{<j - \dlt m}}_{< j - C} (D, y, s) P_{0} Q_{<j-C} \\
    &\quad =: \widetilde{\calQ}_{2,1}(\ell, \ell') +
    \widetilde{\calQ}_{2,2}(\ell, \ell')
  \end{align*}

  For $\widetilde{\calQ}_{2,1}$ we use $\psi_{\ell} \in D L^{6}_{t}
  L^{\infty}_{x}$ and $\oPi_{\geq 2^{\frac{1}{2}(j-\ell') - C'}}
  \psi_{\ell'} \in D L^{3}_{t} L^{\infty}_{x}$, and vice versa for
  $\widetilde{\calQ}_{2,2}$; see \eqref{eq:symbBnd:decomp},
  \eqref{eq:symbBnd:decomp:summed}. We also use fixed-time $L^{2}_{x}$
  boundedness of $e^{i \psi_{<j-\dlt m}}_{<j - C} P_{0}$ in both
  cases. Then
  \begin{align*}
    \nrm{\widetilde{\calQ}_{2,1}(\ell, \ell')}_{L^{\infty}_{t}
      L^{2}_{x} \to L^{2}_{t,x}}
    \aleq & 2^{-\frac{1}{2} j} 2^{- \frac{4}{3} \dlt m} 2^{\frac{1}{6}(j + 10 \dlt m -\ell)} 2^{\frac{1}{3}(j - \dlt m -\ell')} \\
    \nrm{\widetilde{\calQ}_{2,2}(\ell, \ell')}_{L^{\infty}_{t}
      L^{2}_{x} \to L^{2}_{t,x}} \aleq & 2^{-\frac{1}{2} j} 2^{-
      \frac{19}{6} \dlt m} 2^{\frac{1}{3}(j + 10 \dlt m -\ell)}
    2^{\frac{1}{6}(j - \dlt m -\ell')}.
  \end{align*}
  which are good once integrated over $\set{\ell \geq \ell' \geq j -
    \dlt m} \cap \set{\ell \geq j + 10 \dlt m}$.

  \pfstep{Step 5.3: Contribution of $\calQ_{3}$} We again only
  consider $\psi_{\ell}(t,x,\xi) \psi_{\ell'}(t,x,\xi)$. Proceeding as
  in the previous step, the contribution of this term in the integrand
  in \eqref{eq:Xsb:Q3} equals
  \begin{align*}
    &	Q_{j} (\psi_{\ell} \oPi_{\ageq 2^{-\frac{1}{2}(\ell'-j)_{+}}} \psi_{\ell'} (e^{- i \Psi_{<j - \dlt m}}_{\dless j - C}-1))(t,x,D, y, s) P_{0} Q_{<j-C} \\
    &\quad +
    Q_{j} ( \oPi_{\ageq 2^{-\frac{1}{2}(\ell-j)_{+}}} \psi_{\ell} \, \oPi_{\aleq 2^{-\frac{1}{2}(\ell'-j)_{+}}} \psi_{\ell'} (e^{- i \Psi_{<j - \dlt m}}_{\dless j - C} -1))(t,x,D,y,s)P_{0} Q_{<j-C} \\
    & \quad =: \widetilde{\calQ}_{3,1}(\ell, \ell') +
    \widetilde{\calQ}_{3,2}(\ell, \ell').
  \end{align*}

  We proceed as in Step 5.2, but replace the use of $L^{2}_{x}$
  boundedness of $e^{i \psi_{<j-\dlt m}}_{<j- C} P_{0}$ by
  \eqref{eq:mEst:strongL2}. Integrating these bounds over $j - \dlt m
  \leq \ell' \leq \ell \leq j + 10 \dlt m$, we obtain
  \begin{align*}
    \nrm{\int_{j - \dlt m \leq \ell' \leq \ell \leq j + 10 \dlt m}
      \widetilde{\calQ}_{3,1}(\ell, \ell') \, \ud \ell \ud
      \ell'}_{L^{\infty}_{t} L^{2}_{x} \to L^{2}_{t,x}}
    \aleq & 2^{-\frac{1}{2} j} (2^{(1-\dlt_{0})(j - \dlt m)} + 2^{- 10 \dlt m}) \\
    \nrm{\int_{j - \dlt m \leq \ell' \leq \ell \leq j + 10 \dlt m}
      \widetilde{\calQ}_{3,2}(\ell, \ell') \, \ud \ell \ud
      \ell'}_{L^{\infty}_{t} L^{2}_{x} \to L^{2}_{t,x}} \aleq &
    2^{-\frac{1}{2} j} 2^{\frac{1}{2} \dlt m} (2^{(1-\dlt_{0})(j -
      \dlt m)} + 2^{- 10 \dlt m}).
  \end{align*}
  Taking $\dlt > 0$ sufficiently small and using the fact that $j <
  -\frac{1}{2} m$, the desired gain in $m$ follows.

  \pfstep{Step 5.4: Contribution of $\calQ_{4}$} Proceeding as in
  Steps 5.2 and 5.3 for every possible contribution of
  \begin{equation*}
    (\psi_{\ell} (t,x,\xi) - \psi_{\ell} (s,y,\xi)) (\psi_{\ell'} (t,x,\xi) - \psi_{\ell'} (s,y,\xi))
  \end{equation*}
  and recombining the expressions, it follows that
  \begin{align*}
    Q_{j} \Psi_{\ell} \Psi_{\ell'} (t,x,D,y,s) P_{0} Q_{<j-C}
    =& 	Q_{j} \Psi_{\ell} \Psi^{(\ageq 2^{-\frac{1}{2} (\ell'-j)_{+}})}_{\ell'} (t,x, D, y, s)P_{0} Q_{<j-C} \\
    &	+ Q_{j} \Psi_{\ell}^{(\ageq 2^{-\frac{1}{2}(\ell-j)_{+} } ) } \Psi^{(\aleq 2^{-\frac{1}{2} (\ell'-j)_{+}})}_{\ell'} (t,x,D,y,s) P_{0} Q_{<j-C} \\
    =&\!\!: \widetilde{\calQ}_{4, 1}(\ell, \ell') +
    \widetilde{\calQ}_{4, 2}(\ell, \ell')
  \end{align*}
  where
  \begin{equation*}
    \Psi_{\ell} = \Psi^{(>\tht)}_{\ell} + \Psi^{(\leq \tht)}, \qquad
    \Psi^{(>\tht)}_{\ell} (t,x,s,y,\xi) := \oPi_{>\tht} \psi_{\ell}(t,x,\xi) - \oPi_{>\tht} \psi_{\ell}(\xi, y, s).
  \end{equation*}
  Using \eqref{eq:Xsb:L4:pf} with $q = 3$ and summing up in $\tht
  \ageq 2^{-\frac{1}{2}(\ell'-j)_{+}}$, we obtain
  \begin{equation*}
    \nrm{\Psi^{(\ageq 2^{-\frac{1}{2} (\ell'-j)_{+}})}_{\ell'}(t,x,D,y,s) P_{0}}_{L^{\infty}_{t} L^{2}_{x} \to L^{3}_{t} L^{2}_{x}}
    \aleq 2^{-\frac{1}{3} \ell'} 2^{\ell'} 2^{\frac{C}{2} (\ell' - j)_{+}}.
  \end{equation*}
  By the decomposability bound \eqref{eq:symbBnd:decomp:summed} with
  $q = 6$ for $\Psi_{\ell}$, it follows that
  \begin{equation*}
    \nrm{\widetilde{\calQ}_{4,1}(\ell, \ell')}_{L^{\infty}_{t} L^{2}_{x} \to L^{2}_{t,x}}
    \aleq 2^{-\frac{1}{2} j} 2^{\ell'} 2^{\frac{1}{6}(j-\ell)} 2^{C (\ell'-j)_{+}} .
  \end{equation*}
  Then integrating over $j - \dlt m \leq \ell' \leq \ell \leq j + 10
  \dlt m$, we arrive at
  \begin{equation*}
    \nrm{\int_{j - \dlt m \leq \ell' \leq \ell \leq j + 10 \dlt m} \widetilde{\calQ}_{4,1}(\ell, \ell') \, \ud \ell \ud \ell'}_{L^{\infty}_{t} L^{2}_{x} \to L^{2}_{t,x}}
    \aleq 2^{-\frac{1}{2} j} 2^{j} 2^{C \dlt m}
  \end{equation*}
  which is acceptable for $\dlt > 0$ sufficiently small, since $j < -
  \frac{1}{2} m$. The term $\widetilde{\calQ}_{4,2}$ is treated
  similarly, with the roles of $\Psi_{\ell}$ and $\Psi_{\ell'}$
  swapped. This completes the proof of \eqref{eq:Xsb:Qest}.

  \pfstep{Step 6: Low modulation input, $ j < -\frac{1}{2} m$,
    contribution of $\calC$} In this step, we establish
  \begin{equation} \label{eq:Xsb:Cest} \nrm{Q_{j} \calC_{\ll
        0}(t,x,D,y,s) P_{0} Q_{<j-C}}_{N^{\ast} \to X^{0,
        1/2}_{\infty}} \aleq 2^{- \dlt_{1} m}.
  \end{equation}

  This step is easier than Steps 4 and 5, as we do not need to get the
  angle separation to apply the decomposability bound
  \eqref{eq:symbBnd:decomp}; instead, we can use
  \eqref{eq:symbBnd:decomp:summed}. Thanks to this fact, the gauge
  transform need not be as finely localized in frequency as $\calL$
  and $\calQ$. Accordingly, we make the following decomposition:
  \begin{align}
    \calC
    = & i \iiint_{\ell \geq \ell' \geq \ell'' \geq j - \dlt m}
    \Psi_{\ell} \Psi_{\ell'} \Psi_{\ell''} \,
    (e^{- i \Psi_{<\ell''}} - e^{- i \Psi_{<\ell''}}_{\dless -C})\, \ud \ell'' \ud \ell' \ud \ell \label{eq:Xsb:C1} \\
    & + i \iiint_{\substack{\ell \geq \ell' \geq \ell'' \geq j - \dlt
        m \\ \ell \geq j + 10 \dlt m}} \Psi_{\ell} \Psi_{\ell'}
    \Psi_{\ell''} \,
    e^{-i \Psi_{<\ell''}}_{\dless -C}\, \ud \ell'' \ud \ell' \ud \ell \label{eq:Xsb:C2} \\
    & + i \iiint_{j - \dlt m \leq \ell'' \leq \ell' \leq \ell \leq j +
      10 \dlt m} \Psi_{\ell} \Psi_{\ell'} \Psi_{\ell''} \,
    (e^{- i \Psi_{<\ell''}}_{\dless -C} - 1)\, \ud \ell'' \ud \ell' \ud \ell \label{eq:Xsb:C3} \\
    & + i \iiint_{j - \dlt m \leq \ell'' \leq \ell' \leq \ell \leq j +
      10 \dlt m} \Psi_{\ell} \Psi_{\ell'} \Psi_{\ell''} \, \, \ud
    \ell'' \ud \ell' \ud \ell
    \label{eq:Xsb:C4} \\
    = &\!\!: \calC_{1} + \calC_{2} + \calC_{3} + \calC_{4}. \notag
  \end{align}

  We treat $\calC_{1}, \ldots, \calC_{4}$ separately.

  \pfstep{Step 6.1: Contribution of $\calC_{1}$} Proceeding as in
  Steps 4.1 and 5.1, it follows that
  \begin{equation*}
    (\calC_{1})_{\dless 0} 
    = \bb( i \iiint_{\ell \geq \ell' \geq \ell'' \geq j - \dlt m} \Psi_{\ell} \Psi_{\ell'} \Psi_{\ell''} \, 
    (e^{- i \Psi_{<\ell''}}_{\dless C} - e^{- i \Psi_{<\ell''}}_{\dless -C})\, \ud \ell'' \ud \ell' \ud \ell \bb)_{\dless 0}
  \end{equation*}
  where the outer $(\cdot)_{\dless 0}$ may be easily disposed by
  translation invariance. Moreover, we have
  \begin{equation*}
    \nrm{\Psi_{\ell} \Psi_{\ell'} \Psi_{\ell''} \, 
      (e^{- i \Psi_{<\ell''}}_{\dless C} - e^{- i \Psi_{<\ell''}}_{\dless -C})}_{L^{\infty}_{t} L^{2}_{x} \to L^{2}_{t,x}}
    \aleq 2^{- \frac{1}{2} j} 2^{\frac{1}{3}(j - \dlt m - \ell'')} 2^{\frac{1}{6}(j - \dlt m - \ell)} 2^{10 \ell''} 2^{\frac{1}{2} \dlt m}
  \end{equation*}
  by \eqref{eq:symbBnd:decomp:summed} and
  \eqref{eq:k>lPsi:<0}. Integrating over $j - \dlt m \leq \ell'' \leq
  \ell' \leq \ell \leq -m$, this is acceptable.

  \pfstep{Step 6.2: Contribution of $\calC_{2}$} As before, by
  frequency localization properties, the double frequency projection
  $(\cdot)_{\dless 0}$ leaves $\calC_{2}$, $\calC_{3}$ and $\calC_{4}$
  unchanged. Using \eqref{eq:symbBnd:decomp:summed} with $q = 6$ for
  every factor of $\psi$ and $L^{2}_{x}$ boundedness of $e^{- i
    \Psi_{<\ell''}}_{\dless -C}$, it follows that
  \begin{equation*}
    \nrm{Q_{j} \Psi_{\ell} \Psi_{\ell'} \Psi_{\ell''} e^{- i \Psi_{< \ell''}}_{\dless -C} P_{0} Q_{<j-C}}_{L^{\infty}_{t} L^{2}_{x} \to L^{2}_{t,x}}
    \aleq 2^{-\frac{1}{2} j} 2^{\frac{1}{6} (j-\ell)} 2^{\frac{1}{6} (j-\ell')} 2^{\frac{1}{6} (j-\ell'')} .
  \end{equation*}
  Integrating over $\set{\ell \geq \ell' \geq \ell'' \geq j - \dlt m}
  \cap \set{\ell \geq j + 10 \dlt m}$, this is good.

  \pfstep{Step 6.3: Contribution of $\calC_{3}$} Here we use
  \eqref{eq:symbBnd:decomp:summed} with $q = 6$ for every factor of
  $\psi$ and \eqref{eq:mEst:strongL2}. Then we have
  \begin{equation*}
    \nrm{Q_{j} \Psi_{\ell} \Psi_{\ell'} \Psi_{\ell''} (e^{- i \Psi_{< \ell''}}_{\dless -C} -1)P_{0} Q_{<j-C}}_{L^{\infty}_{t} L^{2}_{x} \to L^{2}_{t,x}}
    \aleq 2^{-\frac{1}{2} j} 2^{\frac{1}{6} (j-\ell)} 2^{\frac{1}{6} (j-\ell')} 2^{\frac{1}{6} (j-\ell'')} 2^{(1-\dlt_{0}) \ell''}
  \end{equation*}
  Integrating over $\set{j - \dlt m \leq \ell'' \leq \ell' \leq \ell
    \leq j + 10 \dlt m}$ and using the fact that $j < - \frac{1}{2}
  m$, we obtain the desired gain in $m$.

  \pfstep{Step 6.4: Contribution of $\calC_{4}$} Summing up
  \eqref{eq:Xsb:L4:pf} with $q=6$ in $\tht \ageq 2^{\sgm k}$, we
  obtain
  \begin{equation} \label{eq:Xsb:C4:pf} \nrm{\psi_{\ell}(t,x,D) a(D) -
      a(D) \psi_{\ell}(D,y,s)}_{L^{\infty}_{t} L^{2}_{x} \to L^{6}_{t}
      L^{2}_{x}} \aleq 2^{-\frac{1}{6} \ell} 2^{(1-C \sgm) \ell},
  \end{equation}
  where $a(\xi)$ is any smooth bump function adapted to
  $\set{\abs{\xi} \aleq 1}$.  Applying the decomposability bound
  \eqref{eq:symbBnd:decomp:summed} twice with $q = 6$, it follows that
  \begin{equation*}
    \nrm{Q_{j} \Psi_{\ell} \Psi_{\ell'} \Psi_{\ell''}  P_{0} Q_{<j-C}}_{L^{\infty}_{t} L^{2}_{x} \to L^{2}_{t,x}}
    \aleq 2^{-\frac{1}{2} j} 2^{\frac{1}{6} (j-\ell)} 2^{\frac{1}{6} (j-\ell')} 2^{\frac{1}{6} (j-\ell'')} 2^{(1-C \sgm) \ell}.
  \end{equation*}
  We integrate this over $\set{j - \dlt m \leq \ell'' \leq \ell' \leq
    \ell \leq j + 10 \dlt m}$. Since $j < - \frac{1}{2} m$, the
  desired gain in $m$ follows provided that $\sgm > 0$ is sufficiently
  small.

%

  \pfstep{Step 7: Low modulation input, $ j < -\frac{1}{2} m$, low
    frequency phase} To establish \eqref{eq:Xsb}, it is only left to
  prove
  \begin{equation}
    \nrm{Q_{j} [ e^{- i \Psi_{<j - \dlt m}}_{\dless 0}(t,x,D,y,s) - 1]
      P_{0} Q_{< j - C}}_{N^{\ast} \to X^{0, 1/2}_{\infty}} 
    \aleq 2^{-\dlt_{1} m} .
  \end{equation}

  Since
  \begin{equation*}
    Q_{j} [ e^{- i \Psi_{< j - \dlt m}}_{\dless j -C} - 1] P_{0} Q_{< j-C} = 0
  \end{equation*}
  by modulation localization, it suffices to establish
  \begin{equation*}
    \nrm{Q_{j} [ e^{- i \Psi_{<j - \dlt m}}_{\dless 0}(t,x,D,y,s) - e^{- i \Psi_{< j - \dlt m}}_{\dless j - C}]
      P_{0} Q_{< j - C}}_{N^{\ast} \to X^{0, 1/2}_{\infty}} 
    \aleq 2^{-(10 + \frac{1}{2}) \dlt m} .
  \end{equation*}
  Proceeding as in Step 2, this estimate is reduced to
  \begin{equation*}
    \nrm{[e^{- i \Psi_{<j - \dlt m}}_{\dless 0}(t,x,D,y,s) - e^{- i \Psi_{< j - \dlt m}}_{\dless j - C}]
      P_{0} }_{L^{\infty}_{t} L^{2}_{x} \to L^{2}_{t,x}} 
    \aleq 2^{-\frac{1}{2} j} 2^{-(10 + \frac{1}{2}) \dlt m} .
  \end{equation*}
  The last estimate follows from \eqref{eq:k>lPsi:<0}. \qedhere
\end{proof}

\subsection{Parametrix error estimate}
Here we prove \eqref{eq:mEst:ptxError}. The argument here is essentially the same as in \cite{Krieger:2012vj}.

\pfstep{Step 1: Decomposition of the parametrix error}
At the level of left-quantized operators, we compute
\begin{align*}
&	\Box^{p}_{A} e^{- i \psi_{\pm}}_{<0} (t,x,D) - e^{- i \psi_{\pm}}_{<0}(t,x,D) \Box \\
 & \quad =		2 (\rd^{\mu} e^{-i \psi_{\pm}})_{<0} \rd_{\mu}  
 			+ (\Box e^{-i \psi_{\pm}})_{<0} 
			+ 2i A^{\ell}_{<-m} e^{-i \psi_{\pm}}_{<0} \rd_{\ell}
			+ 2 i A^{\ell}_{<-m} (\rd_{\ell}e^{-i \psi_{\pm}})_{<0} \\
& \quad =
 			2  (\omg \cdot \rd_{x} \psi_{\pm} e^{-i \psi_{\pm}})_{<0} \abs{D}
			+ 2 (\omg \cdot A_{<-m} e^{-i \psi_{\pm}}_{<0}) \abs{D} 
			- 2 (\rd_{t} \psi_{\pm} e^{-i \psi_{\pm}})_{<0} D_{t}	\\
& \qdeq		- (\rd^{\mu} \psi_{\pm} \rd_{\mu} \psi_{\pm} e^{-i \psi_{\pm}})_{<0}
			+ 2 A^{\ell}_{<-m} (\rd_{\ell} \psi_{\pm} e^{-i \psi_{\pm}})_{<0}, 
\end{align*}
where we are using the shorthand $\omg = \xi / \abs{\xi}$. This computation can be justified simply by using the direct definition of left-quantization, or by using the symbol calculus as in \cite{Krieger:2012vj}. On the last line, we used the fact that $\Box \psi_{\pm}(t,x,\xi) = 0$ as $\Box A = 0$. 

To see the cancellation between $A_{<-m}$ and $\oL_{\mp} \psi_{\pm}$, we add and subtract $2 (\pm \rd_{t} \psi_{\pm} - \omg \cdot A_{<-m}   e^{- i \psi_{\pm}})_{<0} \abs{D}$. Then we can write
\begin{align*}
\Box^{p}_{A} e^{- i \psi_{\pm}}_{<0}  - e^{- i \psi_{\pm}}_{<0} \Box 
 = &  - 2 \bb( (\pm \rd_{t} \psi_{\pm} - \omg \cdot \rd_{x} \psi_{\pm} - \omg \cdot A_{<-m}) e^{- i \psi_{\pm}}  \bb)_{<0} \abs{D} \\
& 	- 2(\rd_{t} \psi_{\pm} e^{- i \psi_{\pm}})_{<0} (D_{t} \mp \abs{D}) \\
&	- (-(\rd_{t} \psi_{\pm} \rd_{t} \psi_{\pm} + \rd_{x} \psi_{\pm} \cdot \rd_{x} \psi_{\pm}) e^{- i \psi_{\pm}})_{<0} \\
&	+ 2 A_{<-m} \cdot (\rd_{x} \psi_{\pm} e^{- i \psi_{\pm}})_{<0} \\
&	+ 2 [\omg \cdot A_{<-m}, S_{<0}] e^{- i \psi_{\pm}} \abs{D} \\
= & \!\!: \Diff_{1} +\Diff_{2} + \Diff_{3} + \Diff_{4} + \Diff_{5}. 
\end{align*}

\pfstep{Step 2: Estimate for $\Diff_{1}$}
Being highest order, this is a-priori the most dangerous term. This is precisely the point where we need $\sgm > 0$. In this step we prove
\begin{equation} \label{eq:ptxError:Diff1}
	\nrm{\Diff_{1} P_{0}}_{N^{\ast} \to N} \aleq 2^{- \frac{1}{2}\sgm m}  + 2^{-m}.
\end{equation}

\pfstep{Step 2.1: Preliminary reduction}
By \eqref{eq:defn4psi}, it follows that 
\begin{equation*}
\Diff_{1} = - 2 \bb( \sum_{k < -m} \oPi_{\leq 2^{\sgm k}} (\xi \cdot A_{k}) e^{- i \psi_{\pm}} \bb)_{<0} (t, x, D).
\end{equation*}

Note that $e^{- i \psi_{\pm}}$ can be replaced by $e^{- i \psi_{\pm}}_{<C}$ by the frequency localization of $A$. The outer $(\cdot)_{<0}$ can be easily disposed by translation invariance. Therefore, it suffices to consider
\begin{equation*}
	\calE_{1} := 2 \sum_{k < -m} (\oPi_{\leq 2^{\sgm k}} (\xi \cdot A_{k}) e^{-i \psi_{\pm}}_{<C} )(t, x, D) P_{0}
\end{equation*}

\pfstep{Step 2.2: Reduction to bilinear estimate}
Our next order of business is to remove $e^{- i \psi_{\pm}}_{<C}$. For this purpose, consider the operator
\begin{equation*}
	\calE_{2} := 2 \sum_{k < -m} (\oPi_{\leq 2^{\sgm k}} (\xi \cdot A_{k}))(t,x,D) e^{-i \psi_{\pm}}_{<C}(t, x, D) P_{0}
\end{equation*}
We claim that
\begin{equation} \label{eq:ptxError:Diff1:comm}
\nrm{\calE_{1} - \calE_{2}}_{L^{\infty}_{t} L^{2}_{x} \to L^{1}_{t} L^{2}_{x}}
\aleq 2^{-m}.
\end{equation}
This estimate contributes the term $2^{-m}$ in \eqref{eq:ptxError:Diff1}, as $N^{\ast} \subseteq L^{\infty}_{t} L^{2}_{x}$ and $L^{1}_{t} L_{2} \subseteq N$. 

Thanks to frequency localization of $e^{- i \psi_{\pm}}_{<C}$, we can harmlessly insert an operator $a(D)$ between the two pseudodifferential operators in $\calE_{2}$, where $a(\xi)$ is a smooth bump function adapted to $\set{\abs{\xi} \aeq 1}$. Then by Lemma~\ref{lem:decomp-prod}, it follows that
\begin{equation*}
	\nrm{\calE_{1} - \calE_{2}}_{L^{\infty}_{t} L^{2}_{x} \to L^{1}_{t} L^{2}_{x}} 
		\aleq \sum_{k < -m} \nrm{\rd_{\xi} (\oPi_{\leq 2^{\sgm k}} (\xi \cdot A_{k}) a(\xi))}_{D L^{2}_{t} L^{\infty}_{x}} \nrm{(-i\rd_{x} \psi_{\pm} e^{-i \psi_{\pm}})_{<C}}_{L^{\infty}_{t} L^{2}_{x} \to L^{2}_{t,x}} 
\end{equation*}

Note that 
\begin{equation*} 
\rd_{\xi} (\oPi_{\leq 2^{\sgm k}} (\xi \cdot A_{k}) a(\xi)) 
	= (\frac{\xi}{\abs{\xi}} a(\xi) + \abs{\xi} \rd_{\xi} a(\xi)) \oPi_{\leq 2^{\sgm k}} (\omg \cdot A_{k}) 
		+ \abs{\xi} a(\xi) \rd_{\xi} (\oPi_{\leq 2^{\sgm k}} (\omg \cdot A_{k}))
\end{equation*}
The factors involving only $\xi$ can easily be removed as they are bounded. Invoking \eqref{eq:symbBnd:decomp:A} and summing over $\tht \aleq 2^{\sgm k}$ and $k < -m$, it follows that
\begin{equation} \label{eq:decomp:ASummed}
	\sum_{k < -m} \nrm{\rd_{\xi} (\oPi_{\leq 2^{\sgm k}} (\xi \cdot A_{k}) a(\xi)) }_{DL^{2}_{t} L^{\infty}_{x}} 
	\aleq \sum_{k < -m} 2^{\frac{1}{2} k} 2^{\frac{1}{2} \sgm k }
	\aleq 2^{- \frac{1}{2} (1 + \sgm) m}.
\end{equation}

On the other hand, summing \eqref{eq:symbBnd:decomp} over $\tht \ageq 2^{\sgm k}$ and $k \leq -m$, we obtain
\begin{equation} \label{eq:symbBnd:decomp:dpsiSummed}
	\nrm{\nb \psi_{\pm}}_{D L^{2}_{t} L^{\infty}_{x}} \aleq 2^{-\frac{1}{2} (1-\sgm) m}.
\end{equation}

Now replacing $e^{-i \psi_{\pm}}$ by $e^{-i \psi_{\pm}}_{<2C}$, removing the outer $(\cdot)_{<C}$ by translation invariance as usual and using  \eqref{eq:symbBnd:decomp:dpsiSummed}, we obtain
\begin{equation*}
\nrm{(\rd_{x} e^{-i \psi_{\pm}})_{<C}}_{L^{\infty}_{t} L^{2}_{x} \to L^{2}_{t,x}} \aleq 2^{- \frac{1}{2} (1-\sgm) m}.
\end{equation*}
Combining \eqref{eq:decomp:ASummed} and \eqref{eq:symbBnd:decomp:dpsiSummed}, estimate \eqref{eq:ptxError:Diff1:comm} follows.

\pfstep{Step 2.3: Bilinear estimate}
It is now only left to treat $\calE_{2}$. Note that the operator $e^{-i \psi_{\pm}}_{<C} (t,x,D) P_{0}$ can be easily removed at this point, as it is bounded on $N^{\ast}_{0}$. 
Therefore, it suffices to show
\begin{equation*}
	\nrm{\sum_{k < -m} \oPi_{\leq 2^{\sgm k}} A_{k} (t,x, D) \cdot \rd_{x} \widetilde{P}_{0}}_{N^{\ast} \to N} \aleq 2^{- \frac{1}{2} \sgm m} 
\end{equation*}
where $\widetilde{P}_{0}$ is a slightly enlarged version of $P_{0}$. 

Recall that $\oPi_{\leq 2^{\sgm k}}$ localizes $A_{k}$ into angular sectors of size $\aeq 2^{\sgm k}$ centered at $\omg = \frac{\xi}{\abs{\xi}}$ (close-angle) and $- \omg$ (far-angle). Therefore, by a Whitney-type decomposition in angles, it suffices to consider the sum
\begin{align*}
\sum_{k < -m} \sum_{\ell < \sgm k} \sum_{\substack{\phi, \phi' \\ \dist (\phi, \phi') \aeq 2^{\ell}} } (P_{k} P^{\phi}_{\ell} A)  \cdot \rd_{x} \widetilde{P}_{0} P^{\phi'}_{\ell} 
 + \sum_{k < -m} \sum_{\substack{\phi, \phi' \\ \dist (\phi, \phi') \aeq 1} } (P_{k} P^{\phi}_{\sgm k} A)  \cdot \rd_{x} \widetilde{P}_{0} P^{\phi'}_{\sgm k} \, ,
\end{align*}
where the first sum corresponds to the close-angle interaction, and the second sum corresponds to the far-angle interaction. 

We begin by treating the close-angle interaction. We split this sum into two cases, depending on whether the input modulation is $>k + 2 \ell -C$ or otherwise. 

\pfstep{Step 2.3.1: Close-angle, high modulation input}
By the sharp $L^{2}_{t} L^{6}_{x}$ Strichartz estimate and Bernstein, we have
\begin{equation} \label{eq:ptxError:Str4A}
\nrm{P_{k} P^{\phi}_{\ell} A}_{L^{2}_{t} L^{\infty}_{x}} \aleq 2^{\frac{1}{2} k } 2^{\frac{1}{2} \ell}.
\end{equation}
We estimate the output in $L^{1}_{t} L^{2}_{x}$ and the input in $X^{0,1/2}_{\infty}$, using \eqref{eq:ptxError:Str4A} for $P_{k} P^{\phi}_{\ell} A$. Note that, thanks to the null structure in $A \cdot \rd_{x}$, we also gain a factor of $2^{\ell}$. Using $\ell^{2}$ summability in angles for $A$ and the input, we obtain 
\begin{equation*}
\sum_{\substack{\phi, \phi' \\ \dist(\phi, \phi') \aeq 2^{\ell}}} 
	\nrm{(P_{k} P^{\phi}_{\ell} A)  \cdot \rd_{x} \widetilde{P}_{0} P^{\phi'}_{\ell} Q_{> k + 2 \ell -C}}_{X^{0, 1/2}_{\infty} \to L^{1}_{t} L^{2}_{x}}
	\aleq 2^{\frac{1}{2} \ell}.
\end{equation*}
Summing over $\ell < \sgm k$ and then $k < -m$, the desired gain of $2^{-\frac{1}{2} \sgm m}$ follows.

\pfstep{Step 2.3.2: Close-angle, low modulation input}
In this case, by elementary geometry of the cone, the output modulation is $\aeq 2^{k + 2 \ell}$. Placing the output in $X^{0, -1/2}_{1}$ and the input in $L^{\infty}_{t} L^{2}_{x}$,
the numerology is the same as in Step 2.3.1 and we obtain a gain of $2^{- \frac{1}{2} \sgm m}$.

\pfstep{Step 2.3.3: Far-angle}
We proceed as in the case of close-angle interaction, this time splitting the input into $Q_{>k-C} + Q_{\leq k-C}$. In this case we do not gain from the null structure, but obtain the desired gain $2^{\frac{1}{2} \sgm k}$ from \eqref{eq:ptxError:Str4A}.

\pfstep{Step 3: Estimate for $\Diff_{2}$}
Here we need to use the $S^{\sharp}_{\pm}$ norm. We claim that
\begin{equation*}
	\nrm{\Diff_{2}}_{S^{\sharp}_{\pm} \to N} \aleq 2^{-m}.
\end{equation*}

This estimate follows from the obvious mapping property
\begin{equation*}
	D_{t} \mp \abs{D} : S^{\sharp}_{\pm} \to N.
\end{equation*}
and estimate \eqref{eq:mEst:DtPtx}.

\pfstep{Step 4: Estimate for $\Diff_{3}$ and $\Diff_{4}$}
Again, we replace $e^{-i \psi_{\pm}}$ by $e^{-i \psi_{\pm}}_{<C}$, and dispose the outer $(\cdot)_{<0}$ by translation invariance.
Summing up \eqref{eq:symbBnd:decomp} in $\tht \ageq 2^{\sgm k}$, we have
\begin{align*}
	\nrm{\nb \psi_{k, \pm}}_{D L^{2}_{t} L^{\infty}_{x}} \aleq  2^{\frac{1}{2} (1 - \sgm) k} \nrm{A[0]}_{\dot{H}^{1}_{x} \times L^{2}_{x}} 
\end{align*}

On the other hand, since $A_{k} = P_{k} A$ is independent of $\xi$, it follows from Strichartz that
\begin{equation*}
	\nrm{A_{k}}_{D L^{2}_{t} L^{\infty}_{x}} 
	\aleq \nrm{A_{k}}_{L^{2}_{t} L^{\infty}_{x}} 
	\aleq 2^{\frac{1}{2} k} \nrm{A_{k}[0]}_{\dot{H}^{1}_{x} \times L^{2}_{x}}
\end{equation*}

Then by decomposability and $L^{2}_{x}$ boundedness of $e^{-i \psi_{\pm}}_{<C}$, it follows that
\begin{equation*}
	\nrm{\Diff_{3} + \Diff_{4}}_{L^{\infty}_{t} L^{2}_{x} \to L^{1}_{t} L^{2}_{x}} \aleq 2^{- (1-\sgm) m}
\end{equation*}
which is enough.

\pfstep{Step 5: Estimate for $\Diff_{5}$} 
For each component, the commutator may be written as
\begin{equation*}
	[A_{<-m}, S_{<0}] (\phi) = L(\nb A_{<-m}, \phi)
\end{equation*}
where $L$ is a translation invariant bilinear operator with an integrable kernel. Using this expression, we now proceed as in Step 1. Summation in $k < -m$ is now possible thanks to the extra derivative $\nb$, and we obtain
\begin{equation*}
	\nrm{\Diff_{5}}_{N^{\ast} \to N} \aleq 2^{- m}.
\end{equation*}

Combining Steps 1--5, estimate \eqref{eq:mEst:ptxError} follows.

\bibliographystyle{amsplain}

\providecommand{\bysame}{\leavevmode\hbox to3em{\hrulefill}\thinspace}
\providecommand{\MR}{\relax\ifhmode\unskip\space\fi MR }
\providecommand{\MRhref}[2]{%
  \href{http://www.ams.org/mathscinet-getitem?mr=#1}{#2}
}
\providecommand{\href}[2]{#2}

%
%
%
%
%
%
%
%


\end{document}